\documentclass[10pt]{amsart}
\usepackage{amscd, amsfonts, amsthm, amsgen, amsmath, amssymb ,verbatim, enumerate, textcomp}
\usepackage{hyperref} 
\usepackage[cmtip, all]{xy}
\usepackage{graphicx}
\usepackage[margin=1in]{geometry}

\newtheorem{theorem}{Theorem}[section]

\newtheorem{lemma}[theorem]{Lemma}
\newtheorem{corollary}[theorem]{Corollary}
\newtheorem{proposition}[theorem]{Proposition}
\theoremstyle{definition}
\newtheorem{definition}[theorem]{Definition}

\theoremstyle{remark}
\newtheorem{remark}[theorem]{Remark}

\theoremstyle{keyobservation}

\newtheorem{keyreductions}[theorem]{Key Reductions and the Basic strategy}

\newtheorem{shypothesis}[theorem]{Standing Hypothesis}
\numberwithin{equation}{section}

\numberwithin{equation}{subsection}
\newcommand{\be}%
  {\protect\setcounter{equation}{\value{subsubsection}}}  
  \newcommand{\ee}%
   {\protect\setcounter{subsubsection}{\value{equation}}}

  {\protect\setcounter{subsubsection}{\value{equation}}}

\numberwithin{equation}{subsection}

\newcommand{\GL}{\operatorname{GL}}

\newcommand{\Spec}{\operatorname{Spec}}
\newcommand{\Perf}{\operatorname{Perf}}
\newcommand{\Pscoh}{\operatorname{Pscoh}}

\newcommand{\bG}{\mathbf G}

\newcommand{\cE}{\mathcal E}
\newcommand{\C}{\rm C}

\newcommand{\cX}{\mathcal X}
\newcommand{\codim}{{\rm codim}}

\newcommand{\bK}{\mathbf K}

\renewcommand{\L}{\mathcal L}

\newcommand{\Z}{\mathcal Z}

\renewcommand{\O}{\mathcal{O}}

\def\displaytimes_#1{\mathrel{\mathop{\times}\limits_{#1}}}

\def\displayotimes_#1{\mathrel{\mathop{\bigotimes}\limits_{#1}}}

 \def\ari[#1]{\ar@{^(->}[#1]}
 \def\are[#1]{\ar[#1]^{\txt{\'et}}}
 \def\areh[#1]{\ar[#1]|{\txt{$H$-eq}}^{\txt{\'et}}}
 \def\ars[#1]{\ar@{->>}[#1]}
 \newcommand{\dplus}{\ar@{}[d]|{\mbox{$\oplus$}}}
 \newcommand{\dtimes}{\ar@{}[d]|{\mbox{$\times$}}}

\def \rmA{\rm A}

\def \rmB{\rm B}

\def \C{\mathcal C}

\def \Cl{\mathbb C}

\def \colimm{\underset {m \rightarrow \infty}  {\hbox {lim}}}

\def \colimK.{\underset {\underset K^.  \rightarrow}  {\hbox {lim}}}

\def \colimU.{\underset {\underset U_.  \rightarrow}  {\hbox {lim}}}

\def \cosimp1{\stackrel{\rightrightarrows}{ \leftarrow}}

\def \compl{\, \, {\widehat {}}}

\def \rmE{\rm E}

\def \EG1{E{(G \times {\mathbb C}^*)}{\underset {G\times {\mathbb C}^*} \to \times}}

\def \EZ(s)1{E{(Z(s) \times {\mathbb C}^*)}{\underset {(Z(s)\times {\mathbb C}^*)} \to \times}}

\newcommand{\eps}{\boldsymbol\varepsilon}

\def \eps{\ \epsilon \ }
\def \EM(u){EM(u){\underset {M(u)} \to \times}}
\def \EM(us){EM(u,s){\underset {M(u, s)} \to \times}}

\def \bG{{\mathbf G}}

\def \rmG{\rm G}

\def\holimD{\mathop{\textrm{holim}}\limits_{\Delta }}
\def\holimDm{\mathop{\textrm{holim}}\limits_{\Delta_{\le m} }}

\def\hlimD2{\mathop{\textrm{holim}}\limits_{\Delta_{\le m_2} }}

\def \hocolimD{\underset \Delta  {\hbox {hocolim}}}

\def \holimn {\underset {\infty \leftarrow {\it n}}  {\hbox {holim}}}

\def\holim{\mathop{\textrm{holim}}}

\def \H{\mathbb H}
\def \rmH{\rm H}

\def \holimm {\underset {\infty \leftarrow m}  {\hbox {holim}}}

\def \invlim1{\underset {\infty \leftarrow q} \to {\hbox {lim}}^1}
\def \rmI{\rm I}

\def \rmK{\rm K}
\def \k{\it k}
\def \bKH{\bf KH}

\def \L3{\Lambda \times \Lambda \times \Lambda}
\def \L2{\Lambda \times \Lambda}

\def \longright2arrow{{\overset \longrightarrow \to {\overset {} \to \longrightarrow}}}

\def \L{L\times \Cl ^*}

\def \rmMap{\rm {Map}}

\def \N{\mathbb N}

\def \O{{\mathcal O}}

\def \rmp{\rm p}

\def \rmq{\rm q}

\def \ra{\rightarrow}
\def \Ra{\Rightarrow}
\def \RG^{R(G)^{\hat {}}\ }

\def \res{respectively}

\def \rmR{\rm R}

\def \rmS{\rm S}
\def \S{\mathcal S}

\def \Speck{{\rm Spec}\, {\it k}}

\def\Spt{\rm {\bf Spt}}

\def \mbS{\mathbb S}

\def \topGcoh*{^{top, *} _{G}}
\def \topGho*{ _{top,*} ^{G}}

\def \rmT{\rm T}

\def \rmU{\rm U}

\def \rmV{\rm V}

\def \wedgeKH{\overset L {\underset {\bK (\rmS, G)} \wedge}}
\def \wedgeKG{\overset L {\underset {\bK (\rmS, G)} \wedge}}

\def \rmW{\rm W}

\def \rmX{\rm X}

\def \cX{\mathcal X}

\def \rmY{\rm Y}
\def \cY{\mathcal Y}

\def \Z(s){Z(s) \times {\mathbb C}^*}
\def \Z{\mathbb Z}

\def \rmZ{\rm Z}

\begin{document}

\title{Equivariant Algebraic K-Theory and Derived completions II: the case of Equivariant Homotopy K-Theory and Equivariant K-Theory}
\author{Gunnar Carlsson}
\address{Department of Mathematics, Stanford University, Building 380, Stanford,
California 94305}
\email{gunnar@math.stanford.edu}
\thanks{  }  
\author{Roy Joshua}
\address{Department of Mathematics, Ohio State University, Columbus, Ohio,
43210, USA.}
\email{joshua.1@math.osu.edu}
\author{Pablo Pelaez}
\address{Instituto de Matem\'aticas, Ciudad Universitaria, UNAM, DF 04510, M\'exico.}
\email{pablo.pelaez@im.unam.mx}


\thanks{ The second and third authors would like to thank the Isaac Newton Institute for 
Mathematical Sciences, Cambridge, for support and hospitality during the programme {\it K-Theory, Algebraic Cycles and Motivic Homotopy Theory} 
where work on this paper began. This work was supported by EPSRC grant no EP/R014604/1. The second author was also supported by a grant from the Simons Foundation.}

\maketitle
 
\begin{abstract}
\vskip .1cm
In the mid 1980s, while working on establishing completion theorems for equivariant Algebraic K-Theory similar to the
well-known 
completion theorems for equivariant topological K-theory, 
the late Robert Thomason found the
strong finiteness conditions that are required in such theorems to be too restrictive. Then he made a  conjecture on the existence of a
completion theorem 
for
equivariant Algebraic G-theory, for actions of linear algebraic groups on schemes that holds without any of the
strong finiteness conditions that are required in such theorems proven by him. 
In an earlier work by the first two authors, we solved this conjecture by providing
a derived completion theorem for equivariant G-theory. {\it In the present paper, we provide a similar derived completion
theorem for the homotopy Algebraic K-theory of equivariant perfect complexes, on schemes that need
not be regular.}
\vskip .2cm
Our solution is
broad enough to allow actions by all linear algebraic groups, irrespective of whether they are connected or not, and
acting on any {\it normal} quasi-projective scheme of finite type over a field, irrespective of whether they are regular or projective.
This allows us therefore to consider the Equivariant Homotopy Algebraic K-Theory of large classes of varieties like all toric varieties 
(for the action of a torus) and all spherical varieties (for the action of a reductive group). With finite coefficients invertible in the base fields, we are also able to obtain such 
derived completion theorems for equivariant algebraic K-theory but with respect to actions of diagonalizable group schemes. These enable us to obtain a wide range of applications,
several of which are also explored. 
\end{abstract} 
\setcounter{tocdepth}{1}
\tableofcontents\setcounter{tocdepth}{1}
\markboth{Gunnar Carlsson, Roy Joshua and Pablo Pelaez}{Equivariant Algebraic K-Theory and Derived completions II}
\input xypic
\vfill \eject
\section{Introduction}
\vskip .2cm
Recall that the paper \cite{CJ23} by the first two authors applied the technique of derived completion to 
obtain a derived completion theorem for equivariant Algebraic G-theory, which is the algebraic K-theory of the
category of equivariant coherent sheaves on any suitably nice scheme provided with the action of a linear algebraic group or a smooth affine group scheme. When the scheme is regular, the resulting
equivariant G-theory identifies with equivariant K-theory and therefore, in this case, the above derived completion theorem applies 
to provide a derived completion theorem for equivariant Algebraic K-theory. That left open the problem of obtaining a similar derived
completion theorem for equivariant Algebraic K-theory {\it in general}, that is,  the algebraic K-theory of the category of equivariant vector
bundles and/or equivariant perfect complexes on schemes that need not be regular.
\vskip .2cm
The goal of this paper is to address this problem: we extend the derived completion theorem of \cite{CJ23} to 
{\it  equivariant homotopy algebraic K-theory}. Accordingly, the problem we are considering in this paper can be summarized as follows. Let $\rmX$ denote a scheme provided with the action of
a linear algebraic group $\rmG$. Then let $\bK(\rmX, \rmG)$ denote the spectrum associated to the symmetric monoidal
category of $\rmG$-equivariant vector bundles on $\rmX$. Next let ${\rm EG} \ra {\rm BG}$ denote  a principal $\rmG$-bundle with 
${\rm BG}$ the classifying space for $\rmG$ in a certain sense
as made clear later on. Then the pull-back along the projection $p_2: {\rm EG} \times \rmX \ra \rmX$ induces a map of 
spectra $p_2^*: \bK(\rmX, \rmG) \ra \bK({\rm EG} \times \rmX, \rmG) \simeq \bK({\rmE}\rmG\times_{\rmG}\rmX)$, where 
${\rmE}\rmG\times_{\rmG}\rmX$ is the Borel construction. The {\it main goal} of the present paper
is to prove that the map $p_2^*$ becomes a weak-equivalence after a certain derived completion has been performed at the spectrum level
on $\bK(\rmX, \rmG)$, provided one replaces the spectrum of algebraic K-theory by the spectrum of homotopy algebraic K-theory. The basic underlying principle behind our present work is similar to the one for \cite{CJ23}.\footnote{In \cite{CJ23}, we had already provided a detailed comparison with existing completion theorems in the literature: we pointed out there that none of
 them make use of the technique of derived completions, and as a result are all quite restrictive in terms of the range of applications.
We will assume the basic context and the basic terminology of \cite{CJ23}.} \footnote{It may be worthwhile pointing out that as Equivariant algebraic K-theory is not an ${\mathbb A}^1$-invariant in general, 
our methods do not apply to it, but only to Equivariant Homotopy Algebraic K-theory, which is an ${\mathbb A}^1$-invariant. 
This will become apparent, as at several points along the course of this paper, we will need to strongly make use of the
 fact that the form of Equivariant Algebraic K-theory we consider is ${\mathbb A}^1$-invariant.}

 \vskip .2cm
In addition we discuss several applications, which are left to the accompanying sequel to this paper: \cite{CJP24}. 
{\it Our main results will be stated only under the following basic assumptions.} 
 \begin{shypothesis}
 \label{stand.hyp.1} 
\begin{enumerate}[\rm(i)] {\rm
\item We will assume the base scheme $\rmS$ is the spectrum of a perfect infinite field $k$ of arbitrary characteristic $p$. The schemes we consider will always be assumed to be {\it separated and of finite type over} $\rmS$. 
The group schemes we consider will be smooth
affine group schemes which are finitely presented and separated over $\rmS$: we refer to them as linear algebraic groups over $k$. When we say a group scheme $\rmG$ is affine, we mean that
it admits a closed immersion into some $\GL_{n, \rmS}$, for some integer $n >0$. 
 Observe that any finite group $\rmG$ may be viewed as an affine group-scheme over $\rmS$ in the obvious manner, for example, 
by imbedding it as a subgroup of the monomial matrices in some $\GL _n$. 
\vskip .1cm
\item
We will fix an {\it ambient bigger group}, denoted  $\tilde \rmG$ throughout, and which contains the given linear algebraic group $\rmG$ as a closed sub-group-scheme and 
satisfies the following  strong conditions: $\tilde \rmG$ will denote a connected split reductive group over the field $k$ so that it is {\it special} (in the
 sense of Grothendieck: see \cite{Ch}), and 
 if $\tilde \rmT$ denotes a
maximal torus in $\tilde \rmG$, then $\rmR(\tilde \rmT)$ is free of finite rank over $\rmR(\tilde \rmG)$ and $\rmR(\tilde \rmG)$ is Noetherian. (Here $\rmR(\tilde \rmT)$ and $\rmR(\tilde \rmG)$ denote the
 corresponding representation rings.)
 \vskip .1cm
\item
The above hypothesis is satisfied by 
$\tilde \rmG= {\rm GL}_n$ or ${\rm {SL}}_n$, for any $n$, or any finite product of these groups. It is also trivially satisfied by all split tori.(A basic hypothesis that guarantees this condition is that the algebraic fundamental group 
$\pi_1(\tilde \rmG)$ is 
torsion-free: see section ~\ref{rep.ring}.) 
\vskip .1cm
\item 
In general, we will restrict to {\it normal} quasi-projective schemes of finite type
over $\rmS$,  and provided with an action by a 
linear algebraic group in Theorem ~\ref{main.thm.1}.\footnote{The need to restrict to such schemes comes from
Proposition ~\ref{KH.vanishing}.}
\vskip .1cm
Moreover, in order to consider Riemann-Roch transformations with values in suitable equivariant Borel-Moore homology theories,
we need to  restrict to the following class of schemes:
$\rmX$ is a {\it normal} $\rmG$-scheme over $\rmS$ so that it is {\it either} $\rmG$-quasi-projective (that is, admits a $\rmG$-equivariant locally closed immersion into a projective space over $\rmS$ on which 
$\rmG$-acts linearly), {\it or} $\rmG$ is connected and $\rmX$ is a normal quasi-projective scheme over $\rmS$ (in which case it is
also $\rmG$-quasi-projective by \cite[Theorem 2.5]{SumII}).}
\end{enumerate}
\end{shypothesis}
\vskip .2cm
{\it Next, it seems important
to clarify our basic strategy as discussed in the following key reductions
 ~\ref{remark.validity} as well as Proposition ~\ref{key.obs.2} below: these are essentially the same as those adopted in \cite{CJ23}.}
 \vskip .1cm
\begin{keyreductions}
 \label{remark.validity}
 \begin{enumerate}[\rm(i)] {\rm 
 \vskip .1cm
\item Observe that any linear algebraic group $\rmG$ can be imbedded into $\tilde \rmG$ (as a closed sub-group-scheme), where $\tilde \rmG$ is a general linear group (that is, a $\GL_n$)  or a finite product of such groups.
 Our basic strategy is to show by the following arguments that we may reduce to considering the action of the ambient group $\tilde \rmG$, which will be a finite product of
 $\GL_n$s. Then we reduce to considering actions by a maximal split torus, and eventually to the case of a $1$-dimensional torus.
 \vskip .1cm
Let $\rmX$ denote a scheme as in ~\ref{stand.hyp.1}(iv) and provided with an action by the not-necessarily connected linear algebraic group $\rmG$. We will let $\bK({\rm X}, \rmG)$ ($\bKH({\rm X}, \rmG)$) denote the
K-theory spectrum obtained from the category of $\rmG$-equivariant perfect complexes on $\rmX$: see ~\eqref{KG.def} (the corresponding homotopy K-theory spectrum 
  on $\rmX$: see ~\ref{KH.def}). 
\vskip .1cm
 \item Assume the above situation. Then we let $\tilde \rmG \times \rmG$ act on $\tilde \rmG \times \rmX$ by 
 $(\tilde g_1, g_1)\circ (\tilde g, x)= (\tilde g_1\tilde gg_1^{-1}, g_1x)$, $\tilde g_1, \tilde g \eps \tilde \rmG$, $g_1 \eps \rmG$ and $x \eps \rmX$. Now one may observe that $\tilde \rmG \times \rmG$
 has an induced action on $\tilde \rmG{\underset {\rmG} \times} \rmX$ (defined the same way), and that $\tilde \rmG \times \rmG$ acts on $\rmX$ through the given action of $\rmG $  
 on $\rmX$. (Here $\tilde \rmG{\underset {\rmG} \times }{\rmX}$ denotes the quotient of $\tilde \rmG \times {\rmX}$ by the action of $\rmG$ given by 
$g(\tilde g, x) = (\tilde gg^{-1}, gx)$.) The maps $s: \tilde \rmG \times \rmX \ra \tilde \rmG{\underset {\rmG } \times}\rmX$ and
 $r=pr_2: \tilde \rmG \times \rmX \ra  \rmX$ are $\tilde \rmG\times \rmG$ equivariant maps and the pull-backs
 \be \begin{align}
     \label{key.obs.1}
 s^*:\bK(\tilde \rmG{\underset {\rmG} \times}\rmX, \tilde \rmG) \ra \bK(\tilde \rmG \times \rmX, \tilde \rmG \times \rmG) &\mbox{ and } r^*: \bK(\rmX, \rmG) \ra \bK(\tilde \rmG \times \rmX, \tilde \rmG \times \rmG)\\
 s^*:\bKH(\tilde \rmG{\underset {\rmG} \times}\rmX, \tilde \rmG) \ra \bKH(\tilde \rmG \times \rmX, \tilde \rmG \times \rmG) &\mbox{ and } r^*: \bKH(\rmX, \rmG) \ra \bKH(\tilde \rmG \times \rmX, \tilde \rmG \times \rmG)
\end{align} \ee
 are weak-equivalences of module-spectra over $\bK(\rmS, \tilde \rmG \times \rmG)$. (See Lemma ~\ref{key.pairings} below for more details.)
Moreover, if $\rmX$ is a quasi-projective scheme, then so is $\tilde \rmG{\underset {\rmG} \times}\rmX$. 
These observations, along with Proposition ~\ref{key.obs.2} below,  enable us to just consider the equivariant $\bKH$-theory
with respect to the action of the ambient group $\tilde \rmG$.
\vskip .1cm
\item One may also observe that the induced action by the closed subgroup $\tilde \rmG \times \{1\}$ of $\tilde \rmG \times \rmG$ on $\tilde \rmG{\underset {\rmG} \times} \rmX$
identifies with the left-action by $\tilde \rmG$ on $\tilde \rmG{\underset {\rmG} \times} \rmX$. (Similarly the induced action by the closed subgroup $\rmG \cong \{1\} \times \rmG$ of $\tilde \rmG \times \rmG$
on $\rmX$ identifies with the given action of $\rmG$ on $\rmX$.)
\vskip .1cm
\item If $\rmG$ is a connected split reductive group,
 then $\rmR(\rmG) \cong \rmR(\rmT)^W$, where $\rmW$ denotes the Weyl group of $\rmG$ and $\rmT$ denotes a split maximal torus of $\rmG$. Therefore $\rmR(\rmG)$ is Noetherian, and one may find a closed imbedding $\rmG \ra \GL_{n_1} \times \cdots \GL_{n_m}$ for some $n_1, \cdots, n_m$ such that the
restriction $\rmR(\GL_{n_1} \times \cdots \GL_{n_m}) \ra \rmR(\rmG)$ is surjective, so that \cite[Theorem 1.6]{CJ23} applies. Moreover the same conclusions apply to any linear 
algebraic group $\rmG$ for which $\rmR(\rmG)$ is Noetherian.
\vskip .1cm
\item
In view of these observations, assuming the ambient group $\rmG$ is a finite product of $\GL_n$s is not a serious restriction at all. }
\end{enumerate}
\end{keyreductions}

\vskip .3cm 

We will let  $\rmE\tilde {\rmG}^{gm}$  denote the geometric classifying space
for $\tilde \rmG$ which is constructed in ~\ref{geom.class.space} as an ind-object of schemes. We let 
$\rho_{\rmG}: \bK (\rmS, \rmG) \ra \bK( \rmS)$ denote the map of commutative
ring spectra defined by restriction to the trivial subgroup-scheme. 
 For a prime $\ell \ne p$, let 
$\rho_{\ell}: \mbS \ra \H(\Z/\ell)$ denote the mod$-\ell$ reduction map, where $\mbS$ denotes the sphere spectrum and $\H(\Z/\ell)$ denotes the 
(usual) $\Z/\ell$-Eilenberg-Maclane spectrum. Let
$\rho_{\ell} \circ \rho_{\rmG}: \bK(\rmS, \rmG) \ra \bK(\rmS) {\underset {\mbS} \wedge} \H(\Z/\ell)$ 
denote the composition of $\rho_{\rmG}$ and the mod$-\ell$ 
reduction map $id_{\bK (\rmS)} \wedge \rho_{\ell}:\bK(\rmS) \ra \bK(\rmS){\underset {\mbS} \wedge} \H(\Z/\ell)$.  The derived
completions with respect to the above maps are defined in ~\eqref{der.compl.1}.
\vskip .2cm
We then state the following Proposition that completes the reduction to considering actions by the ambient group $\tilde \rmG$.
\begin{proposition}
  \label{key.obs.2} (i) Making use of the weak-equivalences in ~\eqref{key.obs.1} as module-spectra over $\bK(\rmS, \tilde \rmG \times \rmG)$, one obtains the
   weak-equivalences:
   \[s^*{\compl_{\rho_{\tilde \rmG \times \rmG}}}: \bKH(\tilde \rmG{\underset {\rmG} \times}\rmX, \tilde \rmG){\compl_{\rho_{\tilde \rmG \times \rmG}}} {\overset {\simeq} \ra} \bKH(\tilde \rmG \times \rmX, \tilde \rmG \times \rmG){\compl_{\rho_{\tilde \rmG \times \rmG}}} \mbox{ and } {\it r}^*{\compl_{\rho_{\tilde \rmG \times \rmG}}}: \bKH(\rmX, \rmG){\compl_{\rho_{\tilde  \rmG \times \rmG}}} {\overset {\simeq} \ra} \bKH(\tilde \rmG \times \rmX, \tilde \rmG \times \rmG){\compl_{\rho_{\tilde \rmG \times \rmG}}}.\]
   \vskip .1cm
   (ii) Since the restriction maps $\rmR(\tilde \rmG \times \rmG) \ra \rmR(\tilde \rmG) = \rmR(\tilde \rmG \times \{1\})$ and $\rmR(\tilde \rmG \times \rmG)  \ra \rmR(\rmG) = \rmR(\{1\}\times \rmG)$
   are split-surjective, \cite[Theorem 1.6]{CJ23}, the observation ~\ref{remark.validity}(iii) above and the connectivity statement in Theorem ~\ref{loc.seq}(ii), provide the weak-equivalences:
   \[\bKH(\rmX,  \rmG){\compl_{\rho_{\tilde \rmG }}} \simeq \bKH(\tilde \rmG{\underset {\rmG} \times}\rmX, \tilde \rmG){\compl_{\rho_{\tilde \rmG }}} \simeq \bKH(\tilde \rmG{\underset {\rmG} \times}\rmX, \tilde \rmG){\compl_{\rho_{\tilde \rmG \times \rmG}}} \mbox{ and } \bKH(\rmX, \rmG){\compl_{\rho_{\rmG}}} \simeq \bKH(\rmX, \rmG){\compl_{\rho_{\tilde \rmG \times \rmG}}}.\]
  \vskip .1cm
  (iii) Combining (i) and (ii) we obtain: 
  \[\bKH(\rmX,  \rmG){\compl_{\rho_{\tilde \rmG }}} \simeq \bKH(\tilde \rmG{\underset {\rmG} \times}\rmX, \tilde \rmG){\compl_{\rho_{\tilde \rmG }}} \simeq \bKH(\rmX,  \rmG){\compl_{\rho_{\rmG }}}.\]
  (iv) Corresponding results also hold for the completions with respect to the composition with $\rho_{\ell}$. 
    \end{proposition}
\vskip .2cm
With the above reductions in place, we obtain the following main result.
\begin{theorem} 
\label{main.thm.1}
Assume that the base scheme $\rmS= Spec \,{\it k}$ for a perfect infinite field $k$ and that
$\rmX$ denotes any normal quasi-projective scheme of finite type over $S$ 
 and provided with an action by the not-necessarily-connected linear algebraic group 
$\rmG$. Let $\tilde \rmG$ denote a fixed ambient linear algebraic group satisfying the hypotheses in   ~\ref{stand.hyp.1}(ii)
 containing $\rmG$ as a closed  sub-group-scheme and let $\rmE{\tilde \rmG}^{gm}{\underset {\rmG}  \times}X$ denote the ind-scheme defined by the Borel construction as in
section ~\ref{geom.class.space}. (The K-theory and the homotopy invariant K-theory of these objects are defined in Definition ~\ref{Borel.equiv.th}.)
\begin{enumerate}[\rm(i)]
\item Then the map $\bKH({\rm X}, \rmG) \simeq \bK(\rmS, \rmG) \wedgeKH \bKH({\rm X}, \rmG) \ra \bKH({\rm E}{\tilde \rmG}^{\rm gm}, \rmG) \wedgeKH \bKH({\rm X}, \rmG) \ra 
\bKH({\rm E}{\tilde \rmG}^{\rm gm}\times {\rm X}, \rmG) \simeq \bKH(\rmE{\tilde \rmG}^{\rm gm}{\underset {\rmG}  \times}X)$ 
factors through the derived completion of $\bKH({\rm X}, \rmG)$ at $\rho_{\rmG}$  and induces a weak-equivalence 
\vskip .2cm
$\bKH({\rm X}, \rmG) \compl_{\rho_{\rmG}} {\overset {\simeq} \ra} \bKH(\rmE{\tilde \rmG}^{gm}{\underset {\rmG}  \times}X).$
\vskip .2cm \noindent
The spectrum on the left-hand-side is the derived completion of $\bKH({\rm X}, \rmG)$ along the map $\rho_{\rmG}$. (See section 3 for further details.) 
The above map is contravariantly functorial for $\rmG$-equivariant maps. It is also covariantly functorial for 
proper $ \rmG$-equivariant maps (between schemes) that are perfect: a map of schemes is perfect
if the derived push-forward sends
perfect complexes to perfect complexes, see Definition ~\ref{pseudo.coh}.
\item Let $\ell$ denote a prime different from the characteristic of $k$.
 Then one also obtains a weak-equivalence
\vskip .2cm
$\bKH({\rm X}, \rmG)_{\ell} \compl_{\rho_{\rmG}} {\overset {\simeq} \ra} \bKH(\rmE{\tilde \rmG}^{\rm gm}{\underset {\rmG}  \times}X)_{\ell}$
\vskip .2cm \noindent
where the subscript $\ell$ denotes mod$-\ell$-variants of the appropriate spectra. (See ~\eqref{KGl} for their precise definitions.). Therefore, one also obtains the weak-equivalence:
\vskip .2cm \noindent
$\bKH({\rm X}, \rmG) \compl_{\rho_{\ell} \circ \rho_{\rmG}} {\overset {\simeq} \ra} \bKH({\rm E}{\tilde \rmG}^{\rm gm}{\underset {\rmG}  \times}X) \compl_{\rho_{\ell}}$.
\vskip .2cm \noindent
The spectrum on the left-side (right-side) denotes the derived completion of $\bKH({\rm X}, \rmG)$ ($\bKH(\rmE{\tilde \rmG}^{\rm gm}{\underset {\rmG}  \times}\rmX)$) with respect to the composite map $\rho_{\ell} \circ \rho_{\tilde \rmG}$ 
(the map $\rho_{\ell}$, \res). The functoriality statements as in (i) also hold for this theory. 
\item When $\rmG$ is {\it special}, one obtains a weak-equivalence: $\bKH(\rmE{\tilde \rmG}^{gm}{\underset {\rmG}  \times}X) \simeq \bKH(\rmE{\rmG}^{gm}{\underset {\rmG}  \times}X)$, where
 $\rmE{\rmG}^{gm}{\underset {\rmG}  \times}X$ denotes the ind-scheme associated to $\rmG$ and defined by the Borel construction as in
section ~\ref{geom.class.space}.
\end{enumerate} 
\end{theorem}
In fact, if one restricts to actions of split tori, one may also consider more general base schemes than a field: but we choose not to discuss
this extension in detail, mainly for keeping the discussion simpler. See also Corollary ~\ref{der.comp.negK}, which discusses 
completion theorems for K-theory with finite coefficients again only for actions of split tori.
The {\it strategy we adopt} to proving the above theorem is an extension of the strategies we employed for proving 
a corresponding derived completion theorem for equivariant G-theory in \cite{CJ23}. {\it Homotopy invariance is essential} to our proof, so that
there is {\it no analogue of the above theorem for equivariant algebraic K-theory in general} and our proof invokes certain subtle properties
of equivariant homotopy K-theory as in Theorem ~\ref{loc.seq}, which also seem to be not known before.
\vskip .2cm
At this point it may be important to point out the {\it main differences} with the results and techniques in \cite{CJ23}, where the first two authors prove similar results
for $\rmG$-equivariant $\bG$-theory. 
\begin{enumerate}[\rm(i)]
 \item The very first difference shows up in the proof of Proposition ~\ref{indep.class.sp}, where we prove that the Borel style equivariant homotopy Algebraic K-theory 
 is independent of the choice of a geometric classifying space: since we also allow singular schemes, we are able to prove this only for groups that are special. Moreover, there are more stringent hypotheses on  
 what are called geometric classifying spaces: see ~\ref{geom.class.space}.
 \item As a consequence, the third statement in Theorem ~\ref{main.thm.1} holds only for linear algebraic groups that are special: the corresponding
 statement for equivariant $G$-theory holds for all linear algebraic groups as shown in \cite[Theorem 1.2]{CJ23}.
\item We also need to know that the spectrum of equivariant homotopy K-theory with respect to the action of any linear algebraic group is $-n$-connected for 
some sufficiently large  integer $n$: this is proven in \cite{HK} in the setting of algebraic stacks where the stabilizer groups are all linearly reductive, and using it for the action of split tori in \cite[Theorem 1.2]{KR}. It is proved in general in Theorem ~\ref{loc.seq}(ii). 
(In contrast, the spectrum of equivariant G-theory is always $-1$-connected.)
 \item In addition, we also follow a somewhat different route to proving Theorem ~\ref{main.thm.3}, than the one adopted in \cite{CJ23}. This is because
in order to establish Riemann-Roch theorems, it is essential to prove Proposition  ~\ref{comp.towers.2}: therefore we make use of Proposition ~\ref{comp.towers.2}
to obtain a somewhat different and simpler proof of Theorem ~\ref{main.thm.3}.
 \end{enumerate}
\vskip .2cm
As in \cite{CJ23}, our results in this paper also have {\it several similarities} as well as {\it some key differences}  with  the 
proof by Atiyah and Segal (see \cite{AS69}) of their theorem. All proofs 
proceed by reducing to actions by groups that are easier to understand. In our case, we first reduce
to the case where the given  linear algebraic group $\rmG$ is replaced by the ambient connected split reductive group $\tilde \rmG$, then where
 this group $\tilde \rmG$ is replaced by  a finite product of $\GL_n$s and then finally where this product of $\GL_n$s is replaced by its split maximal torus. 
The key difference between our proof and the proof of the classical Atiyah-Segal theorem is in the use of the derived completion, which
is essential for our proof. Moreover, unlike in the classical case, for a closed sub-group-scheme $\rmG$ in $\tilde \rmG$, the derived completion of a $\bK(\rmS, \rmG)$-module 
spectrum with
respect to the maps $\rho_{\rmG}:\bK(\rmS, \rmG) \ra \bK(\rmS)$ and $\rho_{\tilde \rmG}: \bK(\rmS, \tilde \rmG) \ra \bK(\rmS)$ will be different in general.  The main exception to this was
discussed \cite[Theorem 1.6]{CJ23}.
\vskip .2cm
\cite[Proposition 4.1]{Th88} shows the analysis in \cite[Corollary 3.3]{Seg} carries over to show that the representation ring of
any linear algebraic group over any algebraically closed field of characteristic $0$ is Noetherian. \cite{Serre} shows that for any connected split
reductive group $\tilde \rmG$ defined over a field, the representation ring $\rmR(\tilde \rmG)$ is Noetherian. Therefore, in the present framework,  the need for derived completion is only so as not to
put any strong restrictions on the schemes whose equivariant homotopy K-theory and G-theory we consider.
Recall that we also showed in \cite[Theorem 1.7]{CJ23} that the derived completions reduce to the usual completions at the augmentation ideal for all projective smooth schemes
over a field.
\vskip .1cm
\subsection{\bf Basic conventions and terminology}
\label{basic.terminology}
\begin{enumerate}[\rm(i)]
 \item{First we clarify that all spectra used in
this paper are  symmetric $\rmS^1$-spectra which, when restricted to the category of smooth schemes over ${\rm S}$ 
 are presheaves on the big Nisnevich site of smooth schemes (or a suitable subcategory) over the given base $\rmS$ and are ${\mathbb A}^1$-homotopy invariant
when restricted to smooth schemes over $\rmS$. This category will be denoted
$\Spt_{S^1}(\rmS)$, with the sphere spectrum denoted $\mbS$. An example to keep in mind is the equivariant K-theory spectrum (with respect to the action of a fixed linear algebraic group $\rmG$) 
which is defined only on schemes or algebraic spaces provided with actions by $\rmG$).  It is important to observe  that
some such spectra admit extensions to schemes that are not necessarily smooth over $\rmS$, as well as to algebraic spaces, the main examples of which are the equivariant $\rmG$-theory and $\rmK$-theory spectra. Ring and module spectra 
will have the usual familiar meaning, but viewed as objects in $\Spt_{S^1}(\rmS )$. An appropriate context for much our work would be that of model categories for ring and module spectra as worked
out in \cite[Theorem 4.1]{SS} as well as \cite{Ship04}.}
\vskip .1cm
\item{We will make extensive use of the model structures defined in \cite{Ship04} to produce cofibrant replacements
for commutative algebra spectra over a given commutative ring spectrum. The commutative ring spectra that
show up in the paper are largely the K-theory spectra  and the 
equivariant K-theory spectra associated to  the actions of a 
linear algebraic group.}
\vskip .1cm
\item We also make extensive use of the work of Schlichting on negative K-theory as in \cite{Sch} and \cite{Sch11}.
\vskip .1cm
\item{Let $\Spt_{S^1}(\rmS)$ denote the category of spectra and let $\rmI$ denote a small category. 
Then we provide the category of diagrams of spectra of type $\rmI$, $\Spt_{S^1}(\rmS) ^{\rmI}$, with the projective model structure as 
in \cite[Chapter XI]{BK}. Here the fibrations (weak-equivalences) are maps $\{f^i:K^i \ra L^i|i \eps \rmI\}$
so that each $f^i$ is a fibration (weak-equivalence, \res) with the cofibrations defined by the lifting property with respect to
 trivial fibrations. Since the homotopy inverse limit functor
is not well-behaved unless one restricts to fibrant objects in this model structure, we will always
implicitly replace a given diagram of spectra functorially by a fibrant one before applying the homotopy inverse limit. If $\{K^i|i \eps \rmI \}$ is a diagram of spectra, $\holim \{K^i|i \eps \rmI\}$ actually denotes $\holim \{R(K^i)|i \eps \rmI\}$
where $\{R(K^i)| i \eps \rmI\}$ is a fibrant replacement of $\{K^i| i \eps \rmI\}$ in $\Spt_{S^1}(\rmS) ^{\rmI}$.} 
\end{enumerate}

\vskip .2cm
\subsection{Properties of the representation ring}
\label{rep.ring}
Assume (again) that the base scheme $\rmS$  is the spectrum of a field $k$.
First recall that the algebraic fundamental group associated to a split reductive group $\rmG$ over $k$ may be defined as $\Lambda/\rmX( \rmT)$, where $\Lambda$
($\rmX( \rmT)$) denotes the weight lattice (the lattice of characters of the  maximal torus in $\rmG$, \res): see for example, \cite[1.1]{Merk}. Then it is
observed in \cite[Proposition 1.22]{Merk}, making use of \cite[Theorem 1.3]{St}, that if this fundamental group  is torsion-free, then
$\rmR(\rmT)$ is a  free module over $\rmR(\rmG)$. Here $\rmT$ denotes a maximal torus in $\rmG$ and $\rmR(\rmG)$ ($\rmR(\rmT)$) 
denotes the representation ring of $\rmG$ ($\rmT$, \res).
Making use of the observation that ${\rm SL}_n$ is simply-connected (that is,  the above fundamental group is trivial), for any $n$, one may conclude that 
$\pi_1(\GL_n) \cong \pi_1({\mathbb G}_m) \cong {\mathbb Z}$ where ${\mathbb G}_m$ denotes the central torus in $\GL_n$. Therefore $\rmR(\rmT)$ is free
over $\rmR(\GL_n)$, where $\rmT$ denotes a maximal torus in $\GL_n$.
\vskip .2cm
\subsection{Outline of the paper} We review the basic properties of Equivariant Homotopy K-theory, Equivariant K-theory and Equivariant G-theory in section 2.
This is followed by a discussion of geometric classifying spaces of linear algebraic groups and the basic properties of 
equivariant homotopy K-theory on the Borel construction in section 3. Section 4 is devoted to a quick review of
several basic results on derived completion that we use in later sections of the paper, with most of the key results already worked out
 in detail in ~\cite[section 3]{CJ23}.  Sections 5 and 6 are devoted to a detailed proof of Theorems
~\ref{main.thm.1}, with section 5 discussing the reduction to the case where the group is a split torus.
In this section, we also re-interpret Theorem ~\ref{main.thm.1} in terms of
pro-spectra. Section 6 discusses
 the proof of Theorem ~\ref{main.thm.1} for the action of a split torus.  A short appendix discusses a couple of technical results. Various applications of the derived completion theorems proved in 
 this paper are discussed in the accompanying paper, \cite{CJP24}, the most notable of them being {\it a variety of Riemann-Roch theorems}. 
\section{Equivariant K-theory and Equivariant Homotopy K-theory: basic terminology and properties}
Throughout the paper we will let $\rmS = \Speck$, where $\k$ is a perfect infinite field of arbitrary characteristic. $\tilde \rmG$ will denote  a split reductive group satisfying the standing hypothesis in ~\ref{stand.hyp.1}, quite often 
this being a $\GL_n$ or a finite product of $\GL_n$s. Let $\rmG$ denote a closed linear algebraic subgroup of $\tilde \rmG$. (In particular,  it is a
smooth affine group-scheme over the base field.)
Let $\rmX$ denote a scheme of finite type over $\rmS$ provided with an action by $\rmG$ as above and let $\rmZ$ denote a (possibly empty) closed $\rmG$-stable subscheme. Let 
$\Pscoh_{\rmZ} ({\rm X}, \rmG)$ ($\Perf_{\rmZ} ({\rm X}, \rmG)$)
denote the category of pseudo-coherent complexes of $\rmG$-equivariant $\O_X$-modules with bounded coherent cohomology sheaves with supports contained in $\rmZ$ (the category
of perfect complexes of $\rmG$-equivariant $\O_X$-modules with supports contained in $\rmZ$, \res). 
Recall that a $\rmG$-equivariant complex of $\O_X$-modules is pseudo-coherent (perfect) if it is quasi-isomorphic locally on
the Zariski topology on $\rmX$  to a bounded above complex (a bounded complex, \res) of locally free $\O_{\rmX}$-modules with bounded coherent
cohomology sheaves.
We provide these categories with the structure of bi-Waldhausen categories
with cofibrations, fibrations and weak-equivalences (see \cite[1.2.4 Definition]{ThTr} or \cite[1.2]{Wald}) by letting the cofibrations be the maps of complexes that are degree-wise split monomorphisms
(fibrations be the maps of complexes that are degree-wise split epimorphisms, weak-equivalences be the maps that are quasi-isomorphisms, \res).
\be \begin{equation}
  \label{KG.def}
\bG({\rm X}, \rmG) \quad (\bK({\rm X}, \rmG), \, \bK(\rmX \, on \, \rmZ, \rmG)) 
\end{equation} \ee
will denote the K-theory spectrum obtained from $\Pscoh ({\rm X}, \rmG)$ ($\Perf({\rm X}, \rmG)$, $\Perf_{\rmZ} ({\rm X}, \rmG)$, \res). 
\vskip .2cm
One may also consider the category ${\rm Vect}({\rm X}, \rmG)$ of $\rmG$-equivariant vector bundles on $\rmX$. This is an exact category, and one may apply Quillen's construction (see \cite{Qu}) to it to
produce another variant of the equivariant K-theory spectrum of $\rmX$. If we assume that every $\rmG$-equivariant coherent sheaf on $\rmX$ is the $\rmG$-equivariant quotient of a 
$\rmG$-equivariant vector bundle on $\rmX$, then one may observe that this produces a spectrum weakly-equivalent to $\bK({\rm X}, \rmG)$: see 
\cite[2.3.1 Proposition]{ThTr} or 
\cite[Proposition 2.8]{J10}. It follows from \cite[Theorem 5.7 and Corollary 5.8]{Th83} that this holds in many well-known examples. It is shown in 
\cite[section 2]{J02} that, in general, the map
from the K-theory spectrum to the G-theory spectrum (sending a perfect complex to itself, but viewed as a pseudo-coherent complex) is a weak-equivalence
\be \begin{equation}
     \label{PD}
\bK({\rm X}, \rmG) \simeq \bG({\rm X}, \rmG),
    \end{equation} \ee
\vskip .2cm \noindent
provided $\rmX$ is regular. In general, such a result fails to be true for the Quillen K-theory of $\rmG$-equivariant vector bundles, which is the reason for
our preference to the  Waldhausen style K-theory and G-theory considered above. 
\vskip .2cm
In view of our assumption that the base scheme $\rmS$ is the spectrum of a field $k$, clearly the spectra $\bG(\rmS, \rmG)$ and
$\bK(\rmS, \rmG) $ identify, and these also identify with the (Quillen) K-theory spectrum of the exact category of $\rmG$-equivariant vector bundles on $\rmS$.
\vskip .2cm
\subsection{Pseudo-coherence and Perfection}(See \cite[2.5.2, 2,5,3 and 2.5.4]{ThTr}.)
\begin{definition}
 \label{pseudo.coh}
Let $f: \rmX \ra \rmY$ denote a map between schemes of finite type over the base scheme. Then $f$ is {\it $n$-pseudo-coherent} if for each $x \eps \rmX$, there is a
Zariski open neighborhood $\rmU_{\it x}$ of $x$ and a Zariski open $\rmV \subseteq \rmY$ so that $f:\rmU_{\it x} \ra \rmY$ factoring as $g \circ i$, where 
$i:\rmU_{\it x} \ra \rmZ$ is a closed immersion with $i_*(\O_{\rmU_{\it x}})$ $n$-pseudo-coherent as a complex on $\rmZ$ (see \cite[I]{SGA6}), and where $g: \rmZ \ra \rmV$ is smooth. (It is observed there, that $f$ being 
$n$-pseudo-coherent depends only on the map $f$, and is independent on the choice of $\rmZ$.) Then $f$ is said to be {\it pseudo-coherent} if it is
$n$-pseudo-coherent for all integers $n$.
\vskip .1cm
Such a morphism $f: \rmX \ra \rmY$ is {\it perfect} if it is pseudo-coherent and locally of finite tor-dimension.
\end{definition}
Then the following are discussed as examples of morphisms that are either pseudo-coherent and/or perfect:
\begin{enumerate} [\rm(i)]
 \item If $\rmY$ is Noetherian, any map $f: \rmX \ra \rmY$ that is locally of finite type is pseudo-coherent. (For $\rmY$ not Noetherian, the above
  conclusion is false even for $f$ a closed immersion.)
 \item Any smooth map is perfect: so are regular closed immersions and morphisms that are local complete intersection morphisms.
\end{enumerate}
\begin{theorem}(See \cite[2.5.4 Theorem]{ThTr}.) Let $f: \rmX \ra \rmY$ denote a {\it proper} map of schemes of finite type over the base scheme $\rmS$. Assume that $f$ is pseudo-coherent (perfect, \res). Then if $\cE^{\bullet}$ is a pseudo-coherent (perfect) complex on $\rmX$, $Rf_*(\cE^{\bullet})$
 is a pseudo-coherent complex (perfect complex, \res) on $\rmY$.
\end{theorem}

\vskip .2cm
\subsection{Negative K-Theory}
\label{KN.KH.def}
Let $\rmX$ denote a scheme of finite type over $\rmS$ provided with an action by $\rmG$ as above and let $\rmZ$ denote a (possibly empty) closed $\rmG$-stable subscheme.
 Now observe that $\C=\Perf _{\rmZ}({\rm X}, \rmG)$ has the structure of
a complicial bi-Waldhausen category. Therefore, we will adopt the framework of \cite[5.10]{Sch} to define Negative K-theory.
Observe that the cofibrations (fibrations) in the category $\Perf_{\rmZ}({\rm X}, \rmG)$  are the degree-wise split injections (degree-wise split surjections, \res).
Observe also that it is closed under canonical homotopy pushouts and canonical homotopy pull-backs which are defined
as in \cite[1.1.2.1, 1.1.2.5]{ThTr}. The weak-equivalences are the maps that are quasi-isomorphisms. One may also verify that the axioms discussed in
\cite[1.2.11 and 1.9.6]{ThTr} are true for $\Perf _{\rmZ}({\rm X}, \rmG)$, so that $\Perf _{\rmZ}({\rm X}, \rmG)$ has the structure of a {\it Frobenius category}
(in the sense of \cite[Definitions 3.3, 3.4]{Sch}, where the projective-injective objects are the complexes in $\Perf _{\rmZ}({\rm X}, \rmG)$ that are 
acyclic. In fact let $\C_0$ denote the full subcategory of $\C$ consisting of complexes $\rmK$ for which the map $0 \ra \rmK$ is a weak-equivalence.
Then $(\C, \C_0)$ is a {\it Frobenius pair} in the sense of \cite[Definition 3.4]{Sch}. Therefore, we define the negative K-groups of $\C$ as
the homotopy groups in negative degrees of the  K-theory spectrum associated to the Frobenius pair $(\C, \C_0)$: see \cite[3.2.26]{Sch11}.
\vskip .2cm
We will presently recall this construction briefly. 
First we associate to any complicial bi-Waldhausen category $\C$, its K-theory space (defined as in \cite[1.3]{Wald} or \cite[1.5.2]{ThTr}).
This will be denoted ${\rm K}(\C)$.
Next one defines a suspension functor for complicial bi-Waldhausen categories as in \cite[2.4.6, 3.2.33]{Sch11}. For this one first takes the 
countable envelope of $\C$, whose objects are direct systems of cofibrations indexed by the natural numbers with values in $\C$. The morphisms 
are just the morphisms of such ind-objects. This category has the structure of a complicial bi-Waldhausen category. Finally the suspension $S\C$
is the quotient of the countable envelope of $\C$ by the subcategory $\C$. This has the structure of a Waldhausen category with cofibrations and
 weak-equivalences, where the weak-equivalences are the maps in countable envelope of $\C$ which are isomorphisms in the quotient of 
 the  triangulated category associated to the countable envelope by the triangulated category associated to $\C$.
\vskip .2cm
Then it is shown in \cite[3.2.26]{Sch11} that there are natural maps 
\be \begin{equation}
     \label{str.Kth.sp}
{\rm K}(\C) \ra \Omega {\rm K}(S\C)
    \end{equation} \ee
Therefore, one defines the spectrum ${\mathbb K}(\C)$ by the sequence of spaces whose $n$-th space is given by ${\rm K}(S^n\C)$ and where the 
structure maps are defined by the maps in ~\eqref{str.Kth.sp}. Taking $\C = \Perf _{\rmZ}({\rm X}, \rmG)$, this defines the spectrum 
${\mathbb K}({\rm X}\, on \, {\rmZ}, \rmG)$.
\vskip .2cm
\begin{proposition}
 \label{Key.prop.1}
 Let $\rmG$ denote a linear algebraic group acting on the scheme $\rmX$ and let $\tilde \rmG$ denote a linear algebraic group containing $\rmG$ as a 
 closed sub-group scheme. Let $\rmZ$ denote a $\rmG$-stable closed subscheme of $\rmX$. Let $\rmp: \tilde \rmG \times \rmG \ra \tilde \rmG \times _{\rmG} \rmX$ and
 $\rmq:\tilde \rmG \times \rmG \ra \rmX$ denote the obvious maps. Then the pull-backs:
 \[\rmp^*: \Perf_{\tilde \rmG\times_{\rmG}\rmZ}(\tilde \rmG \times_{\rmG} \rmX, \tilde \rmG) \ra \Perf_{\tilde \rmG \times \rmX}(\tilde \rmG \times \rmX, \tilde \rmG \times \rmG) \mbox { and } q^*: Perf_{\rmZ}(\rmX, \rmG) \ra \Perf_{\tilde \rmG \times \rmX}(\tilde \rmG \times \rmX, \tilde \rmG \times \rmG)
 \]
induce equivalences of the associated triangulated categories, showing that induced maps:
\[\rmp^*: {\mathbb K}(\tilde \rmG\times_{\rmG} \rmX\, on \,\tilde \rmG \times_{\rmG} \rmZ, \tilde \rmG) \ra {\mathbb K}(\tilde \rmG \times \rmX \, on \, \tilde \rmG \times \rmZ, \tilde \rmG \times \rmG) \mbox{ and }
 \rmq^*: {\mathbb K}(\rmX \, on \, \rmZ, \rmG) \ra {\mathbb K}(\tilde \rmG \times \rmX \, on \, \tilde \rmG \times \rmZ, \tilde \rmG \times \rmG)
\]
are weak-equivalences.
\end{proposition}
\begin{proof} Observe that both $\rmp$ and $\rmq$ are flat maps. Therefore, the fact that the functors $\rmp^*$ and $\rmq^*$ induce equivalences of the associated triangulated categories follows readily from descent theory.
 Now the last statement follows from \cite[3.2.29]{Sch11}.
\end{proof}
\begin{proposition}
 \label{loc.tori}
 Let $\rmT$ denote a linearly reductive group acting on on the scheme $\rmX$ and let $\rmZ$ denote a $\rmT$-stable closed subscheme of $\rmX$. Let $\rmU$ denote the 
 complement of $\rmZ$ in $\rmX$. 
 The one obtains the fiber sequence:
 \[{\mathbb K}(\rmX \, on  \, \rmZ, \rmT) \ra {\mathbb K}(\rmX, \rmT) \ra {\mathbb K}(\rmU, \rmT)\]
\end{proposition}
\begin{proof} The main observation needed is that since $\rmT$ is linearly reductive, the $\rmT$-equivariant perfect complexes on $\rmX$ (or 
 equivalently the perfect complexes on the quotient stack $[\rmX/\rmT]$) are compact and the derived category ${\rm D}_{\rm qc}([\rmX/\rmT]) \simeq 
 D_{qc}(\rmX, \rmT)$ is compactly generated by the perfect complexes. This is proven in \cite[Theorem 4.5]{HK} using \cite[Theorem D]{HR}. Therefore, the same proof as in \cite[5.5.1 Lemma and 6.6 Theorem]{ThTr} completes the proof,
 modulo an agreement of the negative K-groups defined using the approach in \cite[sections 5 and 6]{ThTr} with the approach using Frobenius pairs as in \cite{Sch}. 
 This follows along the same lines as in \cite[section 7]{Sch}.
 
\end{proof}

\subsection{Homotopy K-Theory and Equivariant Homotopy K-Theory: Definitions}
\label{KH.def}
We let $\Delta_{\rmS}[n] = \rmS[x_0, \cdots, x_n]/(\Sigma_i x_i-1)$. As $n$ varies, we obtain the cosimplicial scheme $\Delta_{\rmS}[\bullet]$.
Given a  scheme of finite type  $\rmX$ over $\rmS$ with $\rmZ$ a closed subscheme, we let $\bKH(\rmX \, on \, \rmZ) = \hocolimD \{{\mathbb K}(\rmX \times_{\rmS} \Delta _{\rmS}[n] \, on \,\rmZ \times_{\rmS} \Delta _{\rmS}[n]| n\}$, which is the homotopy colimit of the
 simplicial spectrum given by ${\rm n} \mapsto {\mathbb K}(\rmX \times_{\rmS} \Delta_{\rmS} [n]\, on \,\rmZ \times_{\rmS} \Delta _{\rmS}[n] )$. 
\vskip .1cm
Let $\rmG$ denote a smooth affine group-scheme defined over $\rmS$ and acting on the scheme 
 $\rmX$ defined over $\rmS$. Let $\rmZ$ denote a closed $\rmG$-stable subscheme of $\rmX$. 
 Then one defines 
 \[\bKH(\rmX \, on \, \rmZ, \rmG) = \hocolimD \{{\mathbb K}(\rmX \times_{\rmS} \Delta _{\rmS}[n] \, on \, \rmZ \times_{\rmS} \Delta _{\rmS}[n], \rmG)|n\}.\]

\subsection{Homotopy K-Theory and Equivariant Homotopy K-theory: Basic properties}
\label{KH.props}
\begin{enumerate}
 \item By first replacing perfect complexes by a functorial flat replacement, one may see that $\bKH$  is a
 contravariant functor from the category of separated schemes of finite type over $\rmS$ to spectra, while $\bKH(\quad, \rmG)$ is a contravariant functor from separated schemes of finite type  over $\rmS$ 
 and provided with an action by $\rmG$, to spectra. By first replacing perfect complexes by a functorial flabby replacement (for example, given by the functorial Godement resolution), one may also see that $\bKH$ ($\bKH(\quad, \rmG)$) is covariantly functorial for proper morphisms of such schemes that are also perfect
 (proper morphisms of such schemes with $\rmG$-action that are also $\rmG$-equivariant and perfect).  
 \item 
 The pairings 
 $\bK(\rmX \times_{\rmS} \Delta_{\rmS}[n]) \wedge \bK(\rmX \times_{\rmS} \Delta_{\rmS}[n]) \ra \bK(\rmX \times_{\rmS} \Delta_{\rmS}[n])$ which are 
 compatible as $n$ varies, provide a multiplicative pairing 
 \be \begin{equation}
 \label{KH.pairing.1}
 \bKH(\rmX) \wedge \bKH(\rmX) \ra \bKH(\rmX), 
 \end{equation} \ee
 which is contravariantly functorial in 
 $\rmX$. The projections $p_n: \rmS[x_0, \cdots, x_n]/(\Sigma_i x_i-1) \ra \rmS$ provide a map $p^*: \bK(\rmX) \ra \bKH(\rmX)$ which 
 sends the multiplicative pairing $\bK(\rmX) \wedge \bK(\rmX) \ra \bK(\rmX)$ to the pairing in ~\eqref{KH.pairing.1}. It follows that 
 $\bKH(\rmX)$ is a ring spectrum and that the ring structure on $\bKH(\rmX)$ is compatible with the ring structure on the spectrum $\bK(\rmX)$, with all 
 of these being contravariantly functorial in $\rmX$.
 \item 
 Similarly one sees that $\bKH(\rmX, \rmG)$ is a ring spectrum, with the ring structure on the spectrum $\bKH(\rmX, \rmG)$ contravariantly functorial
  in $\rmX$ and $\rmG$-equivariant maps.
\item
 Moreover, it follows from \cite[Proposition 2.4]{Wei81} that $\bKH(\rmX \times {\mathbb A}^1) \simeq \bKH(\rmX)$, that is, $\bKH$ is homotopy
 invariant. Similarly, $\bKH(\rmX \times {\mathbb A}^1, \rmG) \simeq \bKH(\rmX, \rmG)$. 
 \item 
 For a closed subscheme $\rmY$ in $\rmX$,  a basic result following from \cite{ThTr} is that there is a natural weak-equivalence:
 $\bKH(\rmX \mbox{ on } \rmY) \simeq \bKH_{\rmY}(\rmX)$, where $\bKH_{\rmY}(\rmX)$ denotes the homotopy fiber of the 
 restriction $\bKH(\rmX) \ra \bKH(\rmX - \rmY)$. (We may also denote $\bKH_{\rmY}(\rmX) $ as $\bKH(\rmX, \rmX - \rmY)$.)
  \vskip .1cm
  In other words, the presheaf of spectra $\rmX \mapsto \bKH(\rmX)$ has localization sequences in the following sense. If $\rmY \subseteq \rmX$ is a closed
 subscheme of $\rmX$, with complement $\rmU = \rmX - \rmY$, there exists a stable cofiber sequence of $\rmS^1$-spectra:
 \[\bKH(\rmX \mbox{ on }\rmY) \simeq \bKH_{\rmY}(\rmX) \ra \bKH(\rmX) \ra \bKH(\rmU).\]
 \item
 More generally if $\rmT$ is a split torus, the presheaf of spectra $\rmX \mapsto \bKH(\rmX, \rmT)$ from the category of normal quasi-projective schemes of finite type over $\rmS$ with $\rmT$-actions
  to $\rmS^1$-spectra has localization sequences in the following sense. If $\rmY \subseteq \rmX$ is a closed
  $\rmT$-stable subscheme of $\rmX$, with complement $\rmU = \rmX - \rmY$, there exists a stable cofiber sequence of
  $\rmS^1$-spectra (see \cite[Theorem 1.2]{KR}, which in fact follows readily from the localization sequence considered in 
 Proposition ~\ref{loc.tori}):
  \[\bKH(\rmX \mbox{ on } \rmY, \rmT) \simeq \bKH_{\rmY}(\rmX, \rmT) \ra \bKH(\rmX, \rmT) \ra \bKH(\rmU, \rmT),\]
  where $\bKH(\rmX \mbox{ on } \rmY, \rmT)$ is defined as the homotopy K-theory of the category of $\rmT$-equivariant perfect complexes 
  on $\rmX$ which are acyclic on $\rmX - \rmY$.

  \item The presheaf of spectra $\rmX \mapsto \bKH(\rmX)$ satisfies {\it cdh} descent on the big $cdh$-site of the base scheme $\rmS$. (See \cite{CHH04} and \cite{Cis}.)

  \item The natural maps $\bK(\rmX) \ra \bKH(\rmX)$ and $\bK(\rmX, \rmG) \ra \bKH(\rmX, \rmG)$ in (2) are weak-equivalences
  when the scheme $\rmX$ is regular.
 \end{enumerate}
 
 \subsection{Remaining Key properties of Equivariant K- and KH-theories}
\label{KHG.props}
\vskip .1cm
Assume as in ~\ref{remark.validity} that the scheme $\rmX$ is provided with an action by the linear algebraic group $\rmH$ and that $\rmG$ is a bigger linear
algebraic group containing $\rmH$ as a closed subgroup scheme.  Then we let $ \rmG \times \rmH$ act on $ \rmG \times \rmX$ by 
 $( g_1, h_1)\circ ( g, x)= ( g_1 gg_1^{-1}, h_1x)$, $ g_1,  g \in  \rmG$, $h_1 \eps \rmH$ and $x \in \rmX$. Now one may observe that $ \rmG \times \rmH$
 has an induced action on $ \rmG{\underset {\rmH} \times} \rmX$ (defined the same way), and that $ \rmG \times \rmH$ acts on $\rmX$ through the given action of $\rmH $  
 on $\rmX$. (Here $ \rmG{\underset {\rmH} \times }{\rmX}$ denotes the quotient of $ \rmG \times {\rmX}$ by the action of $\rmH$ given by 
$h( g, x) = (gh^{-1}, hx)$.) The maps $s:  \rmG \times \rmX \ra  \rmG{\underset {\rmH } \times}\rmX$ and
 $r=pr_2:  \rmG \times \rmX \ra   \rmX $ are $ \rmG\times \rmH$ equivariant maps. 
\begin{lemma}
\label{key.pairings}
Then one obtains the commutative diagram
\[\xymatrix{{\bK(\rmS, \rmG \times \rmH) \wedge \bKH(\rmG{\underset {\rmH} \times}\rmX, \rmG)} \ar@<1ex>[d]_{id \wedge s^*} \ar@<1ex>[r] & {\bKH(\rmG{\underset {\rmH} \times}\rmX, \rmG)} \ar@<1ex>[d]_{s^*}\\
            {\bK(\rmS, \rmG \times \rmH) \wedge \bKH(\rmG \times \rmX, \rmG \times\rmH)} \ar@<1ex>[r] & {\bKH(\rmG \times \rmX, \rmG \times\rmH)}\\
            {\bK(\rmS, \rmG \times \rmH) \wedge \bKH(\rmX, \rmH)} \ar@<-1ex>[u]^{id \wedge r^*} \ar@<1ex>[r] & {\bKH(\rmX, \rmH)} \ar@<-1ex>[u]^{r^*}}
\]
Moreover both the maps $s^*$ and $r^*$ are weak-equivalences. Corresponding results also hold for $\bK$-theory in the place of $\bKH$-theory.
\end{lemma}
\begin{proof}  First one observes that the pairing on the top row factors through  the composite map
 \[{\bK(\rmS, \rmG \times \rmH) \wedge \bKH(\rmG{\underset {\rmH} \times}\rmX, \rmG)} {\overset {pr_1^*} \ra} {\bK(\rmG\times \rmS, \rmG \times \rmH) \wedge \bKH(\rmG{\underset {\rmH} \times}\rmX, \rmG)} \ra {\bK(\rmG{\underset {\rmH} \times} \rmS, \rmG) \wedge \bKH(\rmG{\underset {\rmH} \times}\rmX, \rmG)}\]
where $pr_1$ denotes the projection to the first factor.  Now it suffices to show that a corresponding diagram of tensor-product pairings of categories of equivariant perfect complexes
 commutes, which will prove the commutativity of the top square. The bottom square commutes for similar reasons.
The fact that $s^*$ and $r^*$ are weak-equivalences follows by observing that the equivariance data and faithfully flat descent provides an equivalence of categories,
${\rm Perf}(\rmG {\underset {\rmH} \times} \rmX, \rmG ) \simeq {\rm Perf}(\rmG \times \rmX, \rmG \times \rmH) \simeq {\rm Perf}(\rmX, \rmH)$, which denote the corresponding categories of equivariant perfect complexes.
 \end{proof}

\begin{enumerate}[\rm(i)]
\item Let $\rmH$ denote a closed subgroup scheme of $\rmG$. 
The $\rmG$-equivariant flat  map $\pi:\rmG{\underset {\rmH}  \times}\rmX \ra {\rmX}$, $(gh^{-1}, hx) \mapsto gh^{-1}hx = gx$
induce a  map $\pi^*: \bK({\rm X}, \rmG) \ra \bK(\rmG{\underset {\rmH}  \times}{\rm X}, \rmG) \simeq \bK({\rm X}, \rmH)$
 which identifies with the corresponding map obtained by restricting the group action from $\rmG$
to $\rmH$. A corresponding result also holds for equivariant homotopy K-theory $\bKH$ in the place of $\bK$.
\item
Next assume that ${\rmG}$ is a split reductive group over $\rmS$ and $\rmH=\rmB$, that is, $\rmH$ is a Borel subgroup  
of $\rmG$. Then using 
the observation that $\rmG/\rmB$ is {\it 
proper and smooth} over ${\rm S}$ and $\rmR^n \pi_* =0$ for $n$ large enough, one sees that the
  map $\pi$ also induces  push-forwards $\pi_*: \bK(\rmG{\underset {\rmB} \times}{\rm X}, \rmG) \ra \bK({\rm X}, \rmG)$. (Such a derived direct image functor may be made functorial at the level of complexes by considering perfect complexes which
are also injective $\O_{\rmX}$-modules in each degree.) A corresponding result also holds for $\bK$ replaced by $\bKH$.
\item
Assume the above situation. Then the projection formula applied to $\rmR\pi_*$ shows that the composition $\rmR\pi_*\pi^*(F) = F \otimes R\pi_*(\O_{\rmG{\underset {\rmB} \times}\rmX}) \cong 
F $, since 
\be \begin{align}
     \label{derived.direct.0}
\rmR^n \pi_*(\O_{\rmG{\underset {\rmB} \times}X}) &= \O_{\rmX}, \mbox{ if n=0 and}\\
&=0, \mbox{ if $n>0$}. \notag
    \end{align} \ee
\vskip .2cm \noindent
It follows that $\pi^*$ is a split monomorphism in this case, with the splitting provided by $\pi_*$. (See ~\eqref{splittings} where this is applied to reduce equivariant $\bKH$-theory with respect to the action of a split reductive group $\rmG$ to that of a Borel subgroup $\rmB$ and hence to that of a 
maximal torus $\rmT$ invoking Lemma ~\ref{B.vs.T}.)
\end{enumerate}
 \begin{proposition} 
\label{KH.vanishing}
 Let $\rmX$ denote a normal quasi-projective scheme of finite type over  $\rmS = Spec \, \k$, or more generally
a normal Noetherian scheme over $\rmS$ which is provided with an ample family of line bundles.
  Assume $\rmT$ denotes a split torus acting on $\rmX$. Then $\bKH(\rmX, \rmT)$ is $-n$-connected for a
 sufficiently large positive integer $n$. Moreover, the above conclusion holds if $\rmX$ is a scheme of finite type over the given base scheme, provided with the action of a split torus $\rmT$ and $\rmX$ has an ample family of $\rmT$-equivariant
  line bundles.
\end{proposition}
\begin{proof} A key first step is to observe that, under either of the hypotheses, the corresponding quotient stack $[\rmX/\rmT]$ has
  affine diagonal. When $\rmX$ is a normal quasi-projective scheme, Sumihiro's theorem \cite[Theorem 1]{SumI}, shows that $\rmX$ admits an equivariant locally closed
  immersion into a projective space onto which the action by $\rmT$ extends, so that one concludes that every $\rmT$-equivariant coherent sheaf on $\rmX$ is the
   $\rmT$-equivariant quotient of a $\rmT$-equivariant locally free coherent sheaf. (Observe that the latter property is
    what is called {\it the resolution property} in \cite{Tot04}.) Similarly if $\rmX$ is normal and is provided with an
  ample family of line bundles, one invokes \cite[1.6]{SumII} to arrive at the same conclusion as is worked out in \cite[Lemmas 2.4. 2.6. 2.10. 2.14]{Th87}. Under the more general 
  assumption that $\rmX$ is of finite type over $\rmS$ and is provided with a family of $\rmT$-equivariant ample line bundles, \cite[Theorem 2.1]{Tot04}
   shows that the corresponding quotient stack $[\rmX/\rmT]$ has the resolution property. 
  \vskip .1cm
  Therefore, \cite[Proposition 1.3]{Tot04}
  shows that, under any of the above hypotheses, the above quotient stacks $[\rmX/\rmT]$ all have affine diagonal. At this point it suffices
  to observe that all the hypotheses of  \cite[Theorem 1.1]{HK} are satisfied so that {\it op. cit} provides the required conclusion.
\end{proof}
 
 \vskip .2cm
 
\subsection{Localization sequences}
Next we proceed to show that one has localization sequences in Equivariant Homotopy K-Theory, not just with respect to actions by linearly reductive groups, but with 
respect to any linear algebraic group. Let $\rmX$ denote a scheme of finite type over $\k$ provided with the action of a 
linear algebraic group $\rmG$. Let $\rmZ$ denote a closed and $\rmG$-stable subscheme of $\rmX$. Let ${\rm Perf}_{\rmZ}(\rmX, \rmG)$ denote 
the category of all $\rmG$-equivariant perfect complexes on $\rmX$ with supports contained in $\rmZ$. This has the structure of a complicial bi-Waldhausen category, where 
the cofibrations (fibrations) are split injections (split surjections, \res) and where the weak-equivalences are those maps that are quasi-isomorphisms.
The non-connective K-theory spectrum of this bi-Waldhausen category will be denoted by ${\mathbb K}(\rmX \, on \, \rmZ, \rmG)$. Let $j: \rmU \ra \rmX$ denote 
the obvious open immersion of $\rmU = \rmX - \rmZ$ into $\rmX$.
\begin{theorem}
 \label{loc.seq}
 \begin{enumerate}[\rm(i)]
 \item Under the above assumptions, one obtains a fibration sequence of spectra:
 \[ {\bKH}(\rmX\, on \,\rmZ, \rmG) \ra {\bKH}(\rmX, \rmG) \ra {\bKH}(\rmU, \rmG) \]
 \item The spectrum $\bKH(\rmX \, on \, \rmZ, \rmG)$ is $-n$-connected for $n$ sufficiently large.
 \end{enumerate}
\end{theorem}
\begin{proof}
 Observe that when $\rmG$ is linearly reductive, the above localization sequence follows readily from the localization sequence in 
 Proposition ~\ref{loc.tori}. We will first imbed $\rmG$ as a closed subgroup of some ${\rm GL}_n$. 
 \vskip .1cm
 We let this ambient group ${\rm GL}_n$ be denoted by $\tilde \rmG$ 
 from now on-wards. We will replace $\rmX$ ($\rmZ$, $\rmU$) by $\tilde \rmG \times_{\rmG}\rmX$ ($\tilde \rmG \times_{\rmG}\rmZ$, $\tilde \rmG \times_{\rmG}\rmU$, \res).
 In view of the weak-equivalences $\bKH(\tilde \rmG\times_{\rmG}\rmX \, on  \, \tilde \rmG \times_{\rmG}\rmZ, \tilde \rmG) \simeq \bKH(\rmX \, on \, \rmZ, \rmG)$,
 $\bKH(\tilde \rmG\times_{\rmG}\rmX, \tilde \rmG) \simeq \bKH(\rmX, \rmG)$ and $\bKH(\tilde \rmG\times_{\rmG}\rmU, \tilde \rmG) \simeq \bKH(\rmU, \rmG)$, 
 we may assume that $\rmG$ denotes $ \tilde \rmG$, $\rmX$ denotes $\tilde \rmG\times_{\rmG}\rmX$, $\rmZ$ denotes  $\tilde \rmG\times_{\rmG}\rmZ$, 
 $\rmU$ denotes  $\tilde \rmG\times_{\rmG}\rmU$, $\rmB$ denotes $\tilde \rmB$ which is a Borel subgroup of $\tilde \rmG$, and $\rmT$ denotes
  a maximal torus of $\tilde \rmG$ contained in $\rmB$.
 \vskip .2cm
 Next we observe the homotopy commutative diagram
 \be \begin{equation}
  \label{diagram.local}
  \xymatrix{{\bKH(\rmX \, on \, \rmZ, \rmG)} \ar@<1ex>[r]^{\alpha_{\rmG}} \ar@<1ex>[d]^{\pi^*} & {\bKH(\rmX, \rmG)} \ar@<1ex>[r] \ar@<1ex>[d]^{\pi^*} & {\bKH(\rmU, \rmG)}  \ar@<1ex>[d]^{\pi^*}\\
            {\bKH(\rmG \times_{\rmB}\rmX \, on \, \rmG \times_{\rmB}\rmZ, \rmG)} \ar@<1ex>[r] \ar@<1ex>[d]^{\pi_*} & {\bKH(\rmG \times_{\rmB}\rmX, \rmG)} \ar@<1ex>[r] \ar@<1ex>[d]^{\pi_*} & {\bKH(\rmG \times_{\rmB}\rmU, \rmG)}  \ar@<1ex>[d]^{\pi_*}\\
            {\bKH(\rmX \, on \, \rmZ, \rmG)} \ar@<1ex>[r]^{\alpha_{\rmG}}  & {\bKH(\rmX, \rmG)} \ar@<1ex>[r]  & {\bKH(\rmU, \rmG)} }
\end{equation} \ee
One may observe also that the composition of the maps in each row is null-homotopic and that the composition of the two vertical maps in each column
is the identity: this last observation follows from ~\eqref{derived.direct.0}. Now the terms in middle row identify as follows:
\[ \bKH(\rmG \times_{\rmB}\rmX \, on \, \rmG \times_{\rmB}\rmZ, \rmG) \simeq \bKH(\rmX \, on \, \rmZ, \rmB) \simeq \bKH(\rmX \, on \, \rmZ,  \rmT),\]
\[\bKH(\rmG \times_{\rmB}\rmX, \rmG) \simeq \bKH(\rmX, \rmB) \simeq \bKH(\rmX, \rmT) \mbox{ and } \bKH(\rmG \times_{\rmB}\rmU, \rmG) \simeq \bKH(\rmU, \rmB) \simeq \bKH(\rmU, \rmT).\]
Observe that the first weak-equivalence in each case comes from Proposition ~\ref{Key.prop.1}, while the second weak-equivalence in each case comes from
Lemma ~\ref{B.vs.T} discussed below. In view of the above weak-equivalences, the middle row now identifies with 
 \[\bKH(\rmX \, on \, \rmZ, \rmT) {\overset {\alpha_{\rmT}} \longrightarrow} \bKH(\rmX, \rmT) \longrightarrow \bKH(\rmU, \rmT)\]
 which is a fibration sequence as observed above. 
 \vskip .2cm
 Since the composition of the maps in each row in ~\eqref{diagram.local} is null-homotopic, it follows that one obtains a
 homotopy commutative diagram:
 \[ \xymatrix{{Cone(\alpha_{\rmG})} \ar@<1ex>[r] \ar@<1ex>[d] & {\bKH(\rmU, \rmG)} \ar@<1ex>[d]\\
              {Cone(\alpha_{\rmT})} \ar@<1ex>[r] \ar@<1ex>[d] & {\bKH(\rmU, \rmT)} \ar@<1ex>[d]\\
              {Cone(\alpha_{\rmG})} \ar@<1ex>[r]  & {\bKH(\rmU, \rmG)}.}
 \]
Now it is straightforward to check using the fact that the middle map induces an isomorphism on all homotopy groups 
that the map in the top row is injective on homotopy groups while the map in the bottom row is surjective on homotopy groups.
Since the maps in the top row and the bottom row are the same, it follows that it induces an isomorphism on all homotopy groups, proving the 
first statement in the theorem.
\vskip .2cm
To prove the connectivity statement in (ii), observe that since the composition $\bKH(\rmX \, on \, \rmZ, \rmG) {\overset {\pi^*} \ra} \bKH(\rmG \times_{\rmB}\rmX \, on \, \rmG \times_{\rmB}\rmZ, \rmG) {\overset {\pi_*} \ra } \bKH(\rmX \, on \, \rmZ, \rmG)$
is the identity, it suffices to prove that $\bKH(\rmG \times_{\rmB}\rmX \, on \, \rmG \times_{\rmB}\rmZ, \rmG)$ is $-n$-connected for $n$ sufficiently large.
In view of the identification of the latter with $\bKH(\rmX \, on \,\rmZ,  \rmT)$ this is clear in view of Proposition ~\ref{KH.vanishing}. Observe that the 
connectivity proven here plays a key role in the proof of Proposition ~\ref{key.obs.2}(ii).
\end{proof}

\begin{lemma}
\label{B.vs.T}
 Let $\rmG$ denote a reductive group acting on the scheme $\rmX$, with $\rmB$ denoting a Borel subgroup of $\rmG$ and $\rmT$ denoting a maximal torus
 contained in $\rmB$.
 Then the obvious restriction $\bKH(\rmX, \rmB) \ra \bKH(\rmX, \rmT)$ is a weak-equivalence.
\end{lemma}
\begin{proof} Making use of the fact that $\rmB$ acts on $\rmX$, we first observe the isomorphism $\rmB \times_{\rmT} \rmX \cong \rmB/\rmT \times \rmX$
 of schemes with $\rmB$ actions. Therefore, one obtains the weak-equivalence:
 \[\bKH(\rmB \times_{\rmT} \rmX, \rmB) \simeq \bKH(\rmB/\rmT \times \rmX, \rmB).\]
 Since $\rmB/\rmT = {\rm R}_u(\rmB)$, which is the unipotent radical of $\rmB$ and hence an affine space, one obtains: $\bKH(\rmB/\rmT \times \rmX, \rmB) \simeq \bKH(\rmX, \rmB)$.
 On the other hand the weak-equivalences in Proposition ~\ref{Key.prop.1} with $\rmG = \rmT$ and $\tilde \rmG = \rmB$ provides the 
 weak-equivalence: $\bKH(\rmB \times_{\rmT} \rmX, \rmB) \simeq \bKH(\rmX, \rmT)$.
\end{proof}
\begin{proposition}(Thom isomorphism for equivariant vector bundles) 
\label{Thm.isom}
Let $\cE$ denote a $\rmG$-equivariant vector bundle on the $\rmG$-scheme $\rmY$ and let
  $i: \rmY \ra \cE$ denote {\it the zero-section immersion}. Then  there exists a {\it Thom-class} $\lambda_{-1}(\cE)$ in $\pi_0(\bKH(\cE \, on \, \rmY))$, so that
cup product with $\lambda_{-1}(\cE)$  induces weak-equivalences:
 \[\bKH(\rmY) \ra \bKH_{\rmY}(\cE) = \bKH(\cE, \cE- \rmY) \mbox { and } \bKH(\rmY, \rmG) \ra \bKH_{\rmY}(\cE, \rmG) = \bKH(\cE, \cE- \rmY, \rmG).\]
\end{proposition}
\begin{proof} Let $\epsilon_1$ denote the trivial vector bundle of rank $1$ on $\rmY$. Then the Thom-space $\cE/\cE-\rmY$ identifies up to
 ${\mathbb A}^1$-homotopy with the homotopy cofiber ${\rm Proj}{(\cE\oplus \epsilon_1)}/{\rm Proj}(\cE)$. Now the projective space bundle formula
 for $\bKH({\rm Proj}{(\cE\oplus \epsilon_1)}, \rmG)$ and $\bKH({\rm Proj}{(\cE)}, \rmG)$ show readily that the required conclusion holds.
\end{proof}
\vskip .2cm
\section{Equivariant Homotopy K-theory on the Borel construction}
\subsection{The geometric classifying space}
\label{geom.class.space}
 We begin by recalling
briefly the construction of the {\it geometric classifying space of a linear algebraic group}: see for example, \cite[section 1]{Tot}, \cite[section 4]{MV}. Let 
$\rmG$ denote a linear algebraic group over $\rmS =Spec\, {\it k}$, that is,  a closed subgroup-scheme in $\GL_n$ over $\rmS$ for some n. For a  (closed) embedding 
$i : \rmG \ra \GL_n$, {\it the geometric classifying space} $\rmB_{gm}(\rmG; i)$ of $\rmG$ with respect to $i$ is defined as 
follows. For $m \ge   1$, let 
\be \begin{equation}
\label{EG.general}
\rmE\rmG^{gm,m}=U_m(\rmG)=U({\mathbb A}^{nm})
\end{equation} \ee
be the open sub-scheme of ${\mathbb A}^{nm}$ where the diagonal action of 
$\rmG$ determined by $i$ is free. By choosing $m$ large enough, one can always ensure that 
$\rmU({\mathbb A}^{nm})$ is non-empty and the quotient $\rmU({\mathbb A}^{nm})/\rmG$ is a quasi-projective scheme.
Moreover we will assume that the following conditions are satisfied (see \cite[Definition 2.1, p. 133]{MV}):
\vskip .1cm
(i) $\rmU_m(\rmG)$ has a $k$-rational point and
\vskip .1cm
(ii) For each $m$, let $\rmZ_m = {\mathbb A}^{nm} - \rmU_m(\rmG)$. Then, for any $m$, there exists an $m'>m$ so that the natural map $\rmU_m(\rmG)= {\mathbb A}^{nm} -\rmZ_m \ra \rmU_{m'}(\rmG)= {\mathbb A}^{nm'}-\rmZ_{m'}$ factors through 
the map ${\mathbb A}^{nm} - \rmZ_m \ra ({\mathbb A}^{nm})^2 - \rmZ_m^2$ of the form $v \mapsto (0, v)$.
\vskip .2cm
We recall the particularly nice construction of a geometric classifying space discussed in \cite[(2.1.1)]{CJ23} and \cite[3.1]{K}. We start with 
a faithful representation $\rmW$ of $\rmG$, and an open non-empty $\rmG$-stable subscheme $\rmU$ of $\rmW$ on which $\rmG$ acts freely, so
 that the quotient $\rmU/\rmG$ is a scheme. 
\vskip .1cm
We now let
\be \begin{equation}
\label{adm.gadget.1}
 \rmW_i = \rmW^{\times ^{ i}},  \rmU_1 = \rmU \mbox{ and } E\rmG^{gm,i}= \rmU_{i+1} = \left(\rmU_i \times \rmW \right) \cup
\left(\rmW \times \rmU_i \right) \mbox{ for }i \ge 1.
\end{equation} \ee
where $\rmU_{i+1}$ is  viewed as a subscheme of 
$\rmW^{\times ^{i+1}} $. Moreover, we let the map $\rmU_i \ra \rmU_{i+1}$ be given by $u_i \mapsto (0, u_i)$. Observe that 
\be \begin{equation}
     \label{EG.fin}
\rmE\rmG^{gm,i}=U_{i+1} = \rmU \times \rmW^{\times i} \cup \rmW \times \rmU \times \rmW^{\times i-1} \cup \cdots \cup \rmW^{\times i} \times \rmU.
    \end{equation} \ee
\vskip .2cm \noindent
Setting 
$\rmY_1 = \rmY= \rmW - \rmU$ and $\rmY_{i+1} = \rmU_{i+1} - \left(\rmU_i \times \rmW\right)$ for 
$i \ge 1$, one checks that $\rmW_i - \rmU_i = \rmY^{\times ^{i}}$ and
$\rmY_{i+1} = \rmY^{\times ^{i}} \oplus \rmU$.
In particular, $\codim_{\rmW_i}\left(\rmW_i - \rmU_i\right) =
i (\codim_{\rmW}(\rmY))$ and
$\codim_{\rmU_{i+1}}\left(\rmY_{i+1}\right) = (i+1)d - i(\dim(\rmY))- d =  i (\codim_{\rmW}(\rmY))$,
where $d = \dim(\rmW)$. Moreover, $\rmU_i \to {\rmU_i}/\rmG$ is a principal $\rmG$-bundle and that the quotient $\rmV_i= \rmU_i/\rmG$ exists 
 as a smooth quasi-projective scheme (since the $\rmG$-action on $\rmU_i$ is free and $\rmU/\rmG$ is a scheme).  

\vskip .2cm
In particular,  if $\rmG=GL_n$, one starts with a faithful representation on the affine space $\rmV={\mathbb A}^n$. Let
$\rmW= End(\rmV)$ and $\rmU= GL(\rmV) = GL_n$.  
\vskip .2cm
In case $\rmG=\rmT = {\mathbb G}_m^n$, $\rmE\rmT^{gm,i}$ will be defined
 as follows. Assume first that $\rmT={\mathbb G}_m$. Then one may let $\rmE\rmT^{gm,i}= {\mathbb A}^{i+1}-0$ with the
diagonal action of $\rmT={\mathbb G}_m$ on ${\mathbb A}^i$. Now $\rmB\rmT^{gm,i+1} = {\mathbb P}^i$, which clearly has a Zariski open covering by $i+1$ affine spaces.
If $\rmT= {\mathbb G}_m^n$, then 
\be \begin{equation}
     \label{EG.toric.case}
 \rmE\rmT^{gm,i}= ({\mathbb A}^{i+1}-0)^{\times n} 
\end{equation} \ee
with the $j$-th copy of ${\mathbb G}_m$ acting on the $j$-th factor ${\mathbb A}^{i+1}-0$.  Now $\rmB\rmT^{gm,i} = ({\mathbb P}^i)^{\times n}$.
We imbed $\rmE\rmT^{gm,i}$ into $\rmE\rmT^{gm,i+1}$ by sending $u \mapsto (0, u)$, $u \in ({\mathbb A}^{i+1}-0)^{\times n}$ and $(0, u) \in  ({\mathbb A}^{i+2}-0)^{\times n}$.
 \begin{lemma}
  \label{conds.MV} For the constructions of the geometric classifying spaces given by ~\eqref{adm.gadget.1}, ~\eqref{EG.fin} and ~\eqref{EG.toric.case}, the hypotheses
 (i) and (ii) in  ~\ref{geom.class.space} hold.
 \end{lemma}
\begin{proof} We will first consider the construction in ~\eqref{adm.gadget.1}. 
It is clear that the action of $\rmG$ on each $\rmU_m(\rmG)$ is free and that by making sure $\rmU_1$ has a $k$-rational point,
all the $\rmU_m(\rmG)$ also will have $k$-rational points. One may now identify each $\rmY_m$ in this construction with
the scheme $\rmZ_m= {\mathbb A}^{nm}- \rmU_m(\rmG)$. Now $\rmZ_{2m}= \rmY_{2m} = \rmY^{\times 2m} = \rmY ^{\times m} \times \rmY^{\times m} =
\rmZ_m \times \rmZ_m$. Therefore, $\rmU_{2m}(\rmG) = {\mathbb A}^{n.2m} - \rmZ_{2m}= {\mathbb A}^{nm} \times {\mathbb A}^{nm}- Z_{2m}=
{\mathbb A}^{nm} \times{\mathbb A}^{nm} - \rmZ_{m} \times \rmZ_m$. This shows that by taking $m'=2m$, the condition (ii) in ~\eqref{EG.general}
is also satisfied in this case. We skip the verification that the conditions (i) and (ii) in ~\ref{geom.class.space} are satisfied
 in the case of the construction in ~\eqref{EG.toric.case} as it is quite straightforward to check.
\end{proof} 
 \begin{proposition} Assume the above situation. Then $\rmU_{\infty} =\colimm \rmU_m(\rmG)$ is ${\mathbb A}^1$-acyclic on the cdh-site of Spec\, k.
  \label{EG.acyclic}
 \end{proposition}
\begin{proof}
 We show that the proof given in \cite[Proposition 2.3, p.134]{MV} for the Nisnevich site extends to the cdh site. 
 Let $Sing_*$ denote resolution functor as in \cite[pp. 87-88]{MV} that takes a simplicial sheaf and produces an ${\mathbb A}^1$-fibrant simplicial sheaf.
 The idea of the proof is to show that for any commutative diagram
 \be \begin{equation}
  \label{lift.1}
  \xymatrix{{\delta\Delta[n]}\ar@<1ex>[r] \ar@<1ex>[d] & {Sing_*(\rmU_{\infty})(\S)} \ar@<1ex>[d]\\
            {\Delta[n]} \ar@<1ex>[r] \ar@{-->}[ur] & {*},}
 \end{equation} \ee
where $\Delta[n]$ is the simplicial $n$-simplex and $\delta \Delta[n]$ is its boundary and $\S$ denotes the spectrum of any Hensel ring, 
there exists a lifting indicated by the dotted arrow. Observe that ${Sing_*(\rmU_{\infty})(\S)}$ denotes the stalk of the simplicial presheaf $Sing_*(\rmU_{\infty})$ 
on the cdh-site, and not the Nisnevich site. Apart from this change the proof is exactly the same as proof of \cite[Proposition 2.3, p.134]{MV} which takes place
on the Nisnevich site. We provide the following details mainly for completeness. As observed in the proof of \cite[Proposition 2.3, p. 134]{MV}, by adjunction, the above diagram corresponds to 
\be \begin{equation}
  \label{lift.1}
  \xymatrix{{\delta\Delta[n]_{{\mathbb A}^1_{\S}}}\ar@<1ex>[r] \ar@<1ex>[d] & {\rmU_{\infty}} \ar@<1ex>[d]\\
            {\Delta[n]_{{\mathbb A}^1_{\S}}} \ar@<1ex>[r] \ar@{-->}[ur] & {*},}
 \end{equation} \ee
 where $\Delta[n]_{{\mathbb A}^1_{\S}} = Spec (\O_{\S}[x_0, \cdots, x_n]/\Sigma_i x_i-1)$ and 
 $\delta\Delta[n]_{{\mathbb A}^1_{\S}}$ is its boundary. Making use of the fact that all schemes we consider are affine, we
 conclude first that the closed immersion $\delta \Delta[n]_{{\mathbb A}^1_{\S}} \ra \Delta[n]_{{\mathbb A}^1_{\S}}$
 induces a surjection $Hom(\Delta[n]_{{\mathbb A}^1_{\S}}, {\mathbb A}^{u}) \ra Hom(\delta\Delta[n]_{{\mathbb A}^1_{\S}}, {\mathbb A}^{u})$
 for any $u$, where $Hom$ denotes maps in the category of schemes. Let $f: \delta\Delta[n]_{{\mathbb A}^1_{\S}} \ra \rmU_m(\rmG)$ denote
 any map. Then above observation shows $f$ extends to a map $f': \Delta[n]_{{\mathbb A}^1_{\S}} \ra {\mathbb A}^{nm}$.
  Let $\rmZ_m= {\mathbb A}^{nm}-\rmU_m$. Since $(f')^{-1}(\rmZ_m) \cap \delta\Delta[n]_{{\mathbb A}^1_{\S}}= \phi$, one can find
  a map $g:\Delta[n]_{{\mathbb A}^1_{\S}} \ra {\mathbb A}^{nm}$ so that it restricted to $\delta\Delta[n]_{{\mathbb A}^1_{\S}}$
  is the constant map sending everything to $0$, and it restricted to $(f')^{-1}(\rmZ_m)$ is the constant map sending
  everything to a fixed point $x$ of $\rmU_m(\rmG)$. (The existence of such a map $g$ follows from Lemma ~\ref{extension}.) Then the
  product map $g \times f': \Delta[n]_{{\mathbb A}^1_{\S}} \ra ({\mathbb A}^{nm})^2$ takes values in the complement of
  $\rmZ_m^2$ and agrees with the composition of $f$ with the morphism ${0} \times id: {\mathbb A}^{nm} \ra ({\mathbb A}^{nm})^2$
  on $\delta\Delta[n]_{{\mathbb A}^1_{\S}}$.
\end{proof}
\begin{lemma}
\label{extension}
 Let $\rmA, \rmB$ denote two closed subschemes of an affine scheme $\rmX$, so that $\rmA \cap \rmB =\phi$. Then, given a map
 $f: \rmA \sqcup \rmB \ra {\mathbb A}^n$, there exists an extension $F: \rmX \ra {\mathbb A}^n$, that is, the
 triangle
 \[\xymatrix{{\rmA \sqcup \rmB } \ar@<1ex>[r]^f \ar@<1ex>[d] & {{\mathbb A}^n}\\
             {\rmX} \ar@<1ex>[ur]_{F} }
 \]
commutes.
\end{lemma}
\begin{proof} We skip the proof as it is essentially the same as on the Nisnevich site.
\end{proof}

\vskip .2cm
Let $\rmB\rmG^{gm,m}=\rmV_m(\rmG)=U_m(\rmG)/\rmG$ denote the quotient 
$\rmS$-scheme (which will be a quasi-projective variety) for the 
action of $\rmG$ on $\rmU_m(\rmG)$ induced by the (diagonal) action of $\rmG$ on ${\mathbb A}^{nm}$; the projection $ \rmU_m(\rmG) \ra \rmV_m(\rmG)$ defines $\rmV_m(\rmG)$ as the 
quotient  of $\rmU_m(\rmG)$ by the free action of $\rmG$ and $\rmV_m(\rmG)$ is thus smooth. We have closed embeddings 
$\rmU_m(\rmG) \ra \rmU_{m+1}(\rmG)$ and $\rmV_m(\rmG) \ra \rmV_{m+1}(\rmG)$ corresponding to the embeddings 
$Id \times  \{\rm0\} : {\mathbb A}^{nm} \ra {\mathbb A}^{nm } \times {\mathbb A}^n$. We set $\rmE\rmG^{gm} = \{U_m(\rmG)|m\} = \{ \rmE\rmG^{gm,m}|m\}$ and 
$ \rmB\rmG^{gm} = \{ \rmV_m(\rmG)|m\}$ which are ind-objects in the category of schemes. (If one prefers, one may view each $\rmE\rmG^{gm,m}$ ($\rmB\rmG^{gm,m}$)
as a sheaf on the big Nisnevich (\'etale) site of smooth schemes over $k$ or on the cdh-site of schemes over $k$, and then view $\rmE\rmG^{gm}$ ($\rmB\rmG^{gm}$) as the 
 the corresponding colimit taken in the corresponding category of sheaves.)
\vskip .2cm
Given a scheme $\rmX$  of finite type over $\rmS$ with a $\rmG$-action satisfying the standing hypotheses ~\ref{stand.hyp.1} , we let $\rmU_m(\rmG){\underset {\rmG} \times} \rmX$ denote the {\it balanced product}, 
where $(u, x)$ and $(ug^{-1}, gx)$ are identified for all $(u, x) \eps \rmU_m \times \rmX$ and $g \eps \rmG$.  Since the $\rmG$-action on $\rmU_m(\rmG)$ is free, $\rmU_m(\rmG){\underset {\rmG} \times} \rmX$
exists as a geometric quotient which is also a quasi-projective scheme in this setting, in case $\rmX$ is assumed to be quasi-projective: see \cite[Proposition 7.1]{MFK}. (In case $\rmX$ is an algebraic space of
finite over $\rmS$, the above quotient also exists, but as an algebraic space of finite type over $\rmS$.)
\vskip .2cm
It needs to be pointed out that the construction of the geometric classifying space is not unique in general. 
When $\rmX$ is an algebraic space or a scheme in general, we will let $\{\rmU_m(\rmG){\underset {\rmG} \times}\rmX|m\}$ denote the
ind-object constructed as above for a chosen ind-scheme $\{\rmU_m(\rmG)|m \ge 0\}$. When $\rmX$ is restricted to the category of
smooth schemes over $\Spec k$, one can apply the result below in ~\eqref{indep.geom.class.sp} to show that the
choice of the ind-scheme $\{\rmU_m(\rmG)|m \ge 0\}$  is irrelevant.
\begin{definition} (Borel style equivariant K-theory and Homotopy K-theory)
\label{Borel.equiv.th}
Assume first that $\rmG$ is a linear algebraic group or a finite group viewed as an algebraic group by imbedding it in some $\GL_n$. 
We define the Borel style equivariant K-theory of $\rmX$ to be $\bK(\rmE\rmG^{gm}{\underset {\rmG} \times}\rmX) = \holimm \bK(\rmU_m(\rmG){\underset {\rmG} \times}\rmX)$.
$\bKH(\rmE\rmG^{gm}{\underset {\rmG} \times}\rmX) = \holimm \bKH(\rmU_m(\rmG){\underset {\rmG} \times}\rmX) = 
\holimm \hocolimD \{\bK(\rmU_m(\rmG){\underset {\rmG} \times}\rmX \times_{\rmS} \Delta_{\rmS}[n])|n\}$.
 If $\ell$ is a prime different from the characteristic of the field $k$,
then $\bK(\rmE\rmG^{gm}{\underset {\rmG} \times}\rmX) \compl_{\rho_{\ell}} =  \holimm \bK(\rmU_m(\rmG){\underset {\rmG} \times}\rmX)\compl_{\rho_{\ell}}$.
$\bKH(\rmE\rmG^{gm}{\underset {\rmG} \times}\rmX) \compl_{\rho_{\ell}}$ is defined similarly.
\end{definition}
\vskip .2cm
Since the $\rmG$-action on $\rmE\rmG^{gm}$ is free, one obtains the weak-equivalences:
\vskip .2cm
$\bK(\rmE\rmG^{gm} \times {\rm X}, \rmG) \simeq \bK(\rmE\rmG^{gm}{\underset {\rmG} \times}\rmX)$, 
$\bKH(\rmE\rmG^{gm} \times {\rm X}, \rmG) \simeq \bKH(\rmE\rmG^{gm}{\underset {\rmG} \times}\rmX)$ 
\vskip .2cm \noindent
and similarly for the $\rho_{\ell}$-completed version.

Next we make the following observations.
\begin{itemize}
 \item Let $\{\rmE\rmG^{gm,m}|m \}$ denote an ind-scheme defined above associated to the algebraic group $\rmG$.
If $\rmX$ is any scheme or algebraic space over $k$, then viewing everything as simplicial presheaves on
the Nisnevich (or cdh) site, we obtain $\colimm \rmE\rmG^{gm,m}{\underset {\rmG} \times }X \cong
(\colimm \rmE\rmG^{gm,m}){\underset {\rmG} \times}X = \rmE\rmG^{gm}{\underset {\rmG} \times}X$. (This follows readily from the
observation that the $\rmG$ action on  $\rmE\rmG^{gm,m}$ is free and that filtered colimits commute with the
balanced product construction above.)
\item It follows therefore, that if $\rmE$ is any ${\mathbb A}^1$-local spectrum and $\rmX$ is a smooth scheme of finite type over $k$, then one obtains a weak-equivalence:
\[ \rmMap(\rmE\rmG^{gm}{\underset {\rmG} \times}{\rm X}, E) \simeq \holimm \{Map(\rmE\rmG^{gm,m}{\underset {\rmG} \times}{\rm X}, E)|m\}.\]
\end{itemize}
where $\rmMap(\quad, E)$ denotes the simplicial mapping spectrum. 
It is shown in \cite[Appendix B]{CJ23}, making use of the properties of motivic slices that for any two different models of geometric classifying spaces
given by $\{\rmE\rmG^{gm,m}|m \}$ and $\{\widetilde {\rmE\rmG}^{gm,m}|m\}$, one obtains a weak-equivalence for any ${\mathbb A}^1$-local spectrum $\rmE$ and any smooth scheme
$\rmX$:
\be \begin{equation}
\label{indep.geom.class.sp}
 \holimm \{\rmMap({\widetilde {\rmE\rmG}}^{gm,m}{\underset {\rmG} \times}{\rm X}, E)|m\} \simeq  \holimm \{\rmMap ({\rmE\rmG}^{gm, m}{\underset {\rmG} \times}{\rm X}, E)|m\}.
\end{equation} \ee
We proceed to extend such comparison results for the Borel constructions, when the scheme $\rmX$ is not necessarily smooth.
\begin{lemma}
 \label{ss}
Assume $\{\cY_n|n \ge 0\}$ is a direct system of schemes of finite type over the base scheme $\rmS$, and that for each fixed $n$,
$\{\phi_{m,n}:\cX_{m,n} \ra \cY_n|m\}$ is a direct system of maps of schemes of finite type over $\rmS$ which are compatible as $n$ varies, in the
 sense that the squares
 \[\xymatrix{{\cX_{m,n}} \ar@<1ex>[r] \ar@<1ex>[d]^{\phi_{m,n}} & {\cX_{m+1, n+1}} \ar@<1ex>[d]^{\phi_{m+1, n+1}}\\
             {\cY_n} \ar@<1ex>[r] & {\cY_{n+1}}}
 \]
commute. Assume $\rmE$ is a presheaf of 
 $\rmS^1$-spectra on the big $cdh$-site of the base scheme $\rmS$. Then one obtains the following, where $\H_{cdh}$ denotes the 
 hypercohomology spectrum computed on the $cdh$-site, and $\rmE_{|\cX_{m,n}}$ denotes the restriction of $\rmE$ to $\cX_{m.n}$.
 \begin{enumerate}
  \item There exist  spectral sequences \newline \indent
  $\rmE_2^{s,t} = \pi_{-s}( \H_{cdh}(\cY_n, \pi_{-t}\holimm R\phi_{m,n*}(E_{|\cX_{m,n}}))) \Ra \pi_{-s-t}( \holimm \H_{cdh}(\cX_{m,n}, \rmE))$, \newline \indent
  \indent  $\rmE_2^{s,t} = \pi_{-s}( \H_{cdh}(\cY_n, \pi_{-t}(E))) \Ra \pi_{-s-t}( \H_{cdh}(\cY_n, \rmE))$.
 \item The natural maps $\rmE_{|\cY_n} \ra R\phi_{m,n*}(\rmE_{|\cX_{m,n}})$ induces a map of the second spectral sequence to the first.
 \end{enumerate}
 Moreover, one may identify the stalks $\holimm \rmR\phi_{m,n*}(\rmE_{|\cX_{m,n}})_{{\it y}} = \H_{cdh}(\cX_{m,n}{\underset {\cY_n} \times}Spec\,\O_{\cY_n,{\it y}}^{cdh}, \rmE)$, where
  $\O_{\cY_n, {\it y}}^{cdh}$ denotes the Hensel ring forming the stalk of $\O_{\cY_n}$ in the cdh-site, at the point $y \eps \cY_n$.
\end{lemma}
\begin{proof} For an $\rmS^1$-spectrum $\rmE$, we will denote by $ \{\rmE[t, t] \ra \rmE[-\infty, t] \ra \rmE[-\infty, t-1]|t \}$ the 
 Postnikov tower. Considering this tower for the spectrum $\rmR\phi_{m, n*}(\rmE_{|\cX_{m,n}})$, one obtains the tower
 \[ \{\rmR\phi_{m,n*}(\rmE_{|\cX_{m,n}})[t, t] \ra \rmR\phi_{m,n*}(\rmE_{|\cX_{m,n}})[-\infty, t] \ra \rmR\phi_{m,n*}(\rmE_{|\cX_{m,n}})[-\infty, t-1]|t\}, \mbox{ and the tower }\]
 \[ \{\H_{cdh}(\cY_n, \holimm \rmR\phi_{m,n*}(\rmE_{|\cX_{m,n}})[t, t]) \ra  \H_{cdh}(\cY_n, \holimm \rmR\phi_{m,n*}(\rmE_{|\cX_{m,n}})[-\infty, t]\]
 \[\ra  \H_{cdh}(\cY_n, \holimm \rmR\phi_{m,n*}(\rmE_{|\cX_{m,n}})[-\infty, t-1] |t\}.\]
 On taking the homotopy groups of the last tower, one obtains the exact couple that provides the first spectral sequence. One may take 
 $\cX_{m,n}=\cY_n$ for all $m$ with $\phi_{m,n}$ the identity map, to obtain the second spectral sequence. Since the maps $\rmE_{|\cY_n} \ra R\phi_{m,n*}(\rmE_{|\cX_{m,n}})$
 are natural in $\rmE$, we obtain the map of the spectral sequences considered in the lemma. Recall that the $cdh$-site is finer than
 the Nisnevich site. Therefore, the stalk of $\O_{\cY_n}$ in the cdh-site, at the point $y \eps \cY_n$, is a Hensel ring and 
 the identification of the stalks of $\rmR\phi_{m,n*}(\rmE_{|\cX_{m,n}})$ follows.
\end{proof}
\begin{proposition}
 \label{indep.class.sp}
 Let $\rmG$ denote a linear algebraic group and $\rmH$ a closed linear algebraic subgroup so that $\rmH$ is special.
 Let $\{\rmB\rmG^{gm,m}|m\}$ ($\{\rmB\rmH^{gm,m}|m\}$) denote finite degree approximations to the classifying space
 of $\rmG$ ($\rmH$, \res) with $\{\rmE\rmG^{gm,m}|m \}$ ($\{\rmE\rmH^{gm,m}|m \}$) denoting the corresponding universal
  principal $\rmG$-bundle ($\rmH$-bundle, \res). Let $\rmX$ denote a scheme of finite type over the base scheme and provided
  with an action by $\rmG$. Then the inverse system of maps
  \be \begin{equation}
    \label{gen.comp.class.sp}
    \{\bKH(\rmE\rmG^{gm,m}\times _{\rmH} \rmX) \rightarrow \bKH(\rmE\rmG^{gm,m} \times  \rmE \rmH^{gm,m}) \times _{\rmH} \rmX) \leftarrow \bKH(\rmE\rmH^{gm,m} \times_{\rmH} \rmX)|m\}
 \end{equation} \ee
induce weak-equivalences on taking the homotopy inverse limit as $m \ra \infty$.
\end{proposition}
\begin{proof} We invoke the last Lemma with the following choices: either 
  $\cX_{m,n} = (\rmE\rmG^{gm, m}  \times \rmE\rmH^{gm,n})\times_{\rmH} {\rm X}$ and  $\cY_{n}= \rmE\rmH^{gm,n}\times_{\rmH} {\rm X} $
 or $\cX_{m,n} = (\rmE\rmG^{gm, n}  \times \rmE\rmH^{gm,m})\times_{\rmH} {\rm X}$ and $\cY_n= \rmE\rmG^{gm, n} \times_{\rmH} {\rm X} $ with $\phi_{m,n}$ denoting the 
 obvious projections and $\rmE$ denoting the $\rmS^1$-spectrum representing homotopy K-theory. This spectrum will be denoted 
 $\bKH$. 
 \vskip .1cm
 We start with the observation that both $f: \rmE\rmG^{gm, m} \times \rmX \ra \rmE\rmG^{gm, m} \times_{\rmH} \rmX$ and
 $g: \rmE\rmH^{gm,n} \times \rmX \ra \rmE\rmH^{gm,n}\times_{\rmH} {\rm X}$ are $\rmH$-torsors. Therefore, we can find a
 Zariski open cover $\{\rmU_i|i\}$ of $\rmE\rmG^{gm, m} \times_{\rmH} \rmX$ and a Zariski open cover $\{\rmV_j|j\}$ of 
 $\rmE\rmH^{gm, m} \times_{\rmH} \rmX$ so that $f$ ($g$) is trivial over the cover $\{\rmU_i|i\}$ ($\{\rmV_j|j\}$, \res). From this, it follows readily that
  the induced map $(\rmE\rmG^{gm, m}  \times \rmE\rmH^{gm,n})\times_{\rmH} {\rm X} \ra \rmE\rmG^{gm, m} \times_{\rmH} {\rm X}$ restricted to
the Zariski open cover $\{\rmU_i|i\}$ looks like $\{\rmU_i \times \rmE\rmH^{gm,n} \ra \rmU_i|i\}$. Similarly the
induced map $(\rmE\rmG^{gm, m}  \times \rmE\rmH^{gm,n})\times_{\rmH} {\rm X} \ra \rmE\rmH^{gm, n} \times_{\rmH} {\rm X}$ restricted 
to the Zariski open cover $\{\rmV_j|j\}$ looks like $\{\rmV_j \times \rmE\rmG^{gm, m}  \ra \rmV_j|j\}$. Therefore, the fibers of
 the map $\colimm \phi_{m,n} : \colimm \cX_{m,n} \ra \cY_n$  for both choices of $\{\cY_n|n\}$ are ${\mathbb A}^1$-acyclic
 and each map $\phi_{m,n}$ is locally trivial on the Zariski site of $\cY_n$.
 \vskip .2cm
Therefore the stalks of $\colimm \phi_{m,n}$ in the cdh-site, at a point $y_n$ in $\cY_n$, identify with either
  \[\colimm \rmE\rmG^{gm,m}\times_{\Speck}Spec (\O_{\cY_n,y_n}^{cdh})\mbox{ or } \colimm \rmE\rmH^{gm,m}\times_{\Speck}Spec (\O_{\cY_n,y_n}^{cdh}),\]
  where $\O_{\cY_n,y_n}^{cdh}$ denotes the Hensel ring forming the stalk of $\O_{\cY_n}$ in the cdh-site at the point $y_n$. The ${\mathbb A}^1$-acyclicity of these objects is proven in Proposition ~\ref{EG.acyclic} making strong use of
  the hypotheses in (i) and (ii) in ~\ref{geom.class.space}.
 \vskip .1cm
  Recall that homotopy K-theory is an ${\mathbb A}^1$-invariant. It follows therefore that with the spectrum 
 $\rmE$ denoting the
 spectrum representing homotopy K-theory on the cdh-site, the map of spectral sequences considered in Lemma ~\ref{ss}
  induces an isomorphism at the $E_2$-terms. For a fixed $n$, the space $\cY_n $ being either $\rmE\rmG^{gm, n} \times_{\rmH} {\rm X}$ or $\cY_n=  \rmE\rmH^{gm,n}\times_{\rmH} {\rm X}$ 
  has finite cohomological dimension for the cdh topology, so that both the above spectral sequences converge strongly.
  This provides an isomorphism at the abutments of these spectral sequences, and therefore a weak-equivalence, for each fixed $n$:
  \[\holimm \H_{cdh}(\cX_{m,n}, \bKH) \simeq \H_{\it cdh}(\cY_{\it n}, \bKH).\]
  Finally one takes the homotopy inverse limit $\holimn$ to obtain the weak-equivalence:
  \[\holimn \holimm \H_{cdh}(\cX_{m,n}, \bKH) \simeq \holimn \H_{\it cdh}(\cY_{\it n}, \bKH).\]
  This  proves the proposition, since the spectrum representing homotopy K-theory has
  cdh-descent.
\end{proof}

\subsection{Remaining Key properties}
\label{KG.props}
\vskip .1cm
Assume as in ~\ref{remark.validity} that $\rmX$ is provided with an action by the linear algebraic group $\rmH$ and that $\rmG$ is a bigger linear
algebraic group containing $\rmH$ as a closed subgroup scheme.  

\begin{enumerate}[\rm(i)]

\item Let $\rmH$ denote a closed subgroup scheme of $\rmG$. Then, as observed above, one obtains the weak-equivalences: $\bK({\rm X}, \rmH) \simeq \bK(\rmG{\underset {\rmH} \times}{\rm X}, \rmG)$ and
$\bKH({\rm X}, \rmH) \simeq \bKH(\rmG{\underset {\rmH} \times}{\rm X}, \rmG)$.  Under the same hypotheses, one also obtains the weak-equivalences of inverse systems (that is,  a weak-equivalence on taking their homotopy inverse limits): 
\[\{\bK(\rmE\rmG^{gm, m}\times_{ \rmH}X)|m\} \simeq \{\bK(\rmE\rmG^{gm, m}{\underset {\rmG} \times}G{\underset {\rmH} \times}X)|m\}\,  {\it in \, general} \mbox{ and }\]
\[\{\bKH(\rmE\rmH^{gm, m}{\underset {\rmH}  \times}X)|m\} \simeq \{\bKH(\rmE\rmG^{gm, m}{\underset {\rmG} \times}G{\underset {\rmH}  \times}X)|m\},\]
when $\rmH$ is assumed to be {\it special}. (The first is clear and the second
follows readily in view of Proposition ~\eqref{indep.class.sp}.)
\vskip .1cm
For the remaining items (ii) through (vi), we will assume that the action of the subgroup $\rmH$ on $\rmX$ extends
to an action by the ambient group $\rmG$. Moreover, where these properties are used is in section 4, $\rmG$ will
denote a split reductive group, and $\rmH$ will denote either a Borel subgroup or a maximal split subtorus.
\item
The $\rmG$-equivariant flat  map $\pi:\rmG{\underset {\rmH}  \times}\rmX \ra {\rmX}$, $(gh^{-1}, hx) \mapsto gh^{-1}hx = gx$
induce a  map \newline \noindent  $\pi^*: \bK(\rmE\rmG^{gm, m}{\underset {\rmG} \times}\rmX) \ra \bK(\rmE\rmG^{gm, m }{\underset {\rmG} \times} (\rmG {\underset {\rmH}  \times}\rmX))$
 which identifies with the corresponding map obtained by restricting the group action from $\rmG$
to $\rmH$. A corresponding result also holds for equivariant homotopy K-theory $\bKH$ in the place of $\bK$.
\item
In view of the above properties, given any linear algebraic group $\rmG$, we will fix a closed imbedding $\rmG \ra GL_n$ for some $n$ and
identify 
\[\bK(\rmE{\GL_n}^{gm, m}{\underset {\rmG} \times}\rmX) \mbox{ with } \bK(\rmE{\GL_n}^{gm, m }{\underset {\GL_n} \times}(\GL_n{\underset {\rmG} \times}\rmX)) \mbox{ and } \]
\[\bKH(\rmE{\GL_n}^{gm, m}{\underset {\rmG} \times}\rmX) \mbox{ with } \bKH(\rmE{\GL_n}^{gm, m}{\underset {\GL_n} \times}(\GL_n{\underset {\rmG} \times}\rmX)).\]
\item
Next assume that ${\rmG}$ is a split reductive group over $\rmS$ and $\rmH=\rmB$, that is, $\rmH$ is a Borel subgroup  
of $\rmG$. Then using 
the observation that $\rmG/\rmB$ is {\it 
proper and smooth} over ${\rm S}$ and ${\rm R}^n \pi_* =0$ for $n$ large enough, one sees that the
  map $\pi$ also induces  push-forwards 
$\pi_*:  \bK(\rmE\rmG^{gm, m}{\underset {\rmG} \times} G{\underset {\rmB} \times}\rmX) \ra \bK(\rmE\rmG^{gm, m}{\underset {\rmG} \times}\rmX)$, induced by the derived direct image
functor ${\rm R}\pi_*$. (Such a derived direct image functor may be made functorial at the level of complexes by considering perfect-coherent complexes which
are also injective $\O_{\rmX}$-modules in each degree.) A corresponding result also holds for $\bK$ replaced by $\bKH$.
\item
Assume the above situation. Then the projection formula applied to $\rmR\pi_*$ shows that the composition $\rmR\pi_*\pi^*(F) = F \otimes \rmR\pi_*(\O_{\rmG{\underset {\rmB} \times}\rmX}) \cong 
F $, since 
\be \begin{align}
     \label{derived.direct}
\rmR^n \pi_*(\O_{\rmG{\underset {\rmB} \times}X}) &= \O_{\rmX}, \mbox{ if n=0 and}\\
&=0, \mbox{ if $n>0$}. \notag
    \end{align} \ee
\vskip .2cm \noindent
It follows $\pi^*$ is a split monomorphism in this case, with the splitting provided by $\pi_*$. (See ~\eqref{splittings} where this is applied 
to reduce equivariant $\bKH$-theory on the Borel construction with respect to the action of a split reductive group $\rmG$ to that of a Borel subgroup $\rmB$ and hence to that of a maximal torus $\rmT$.)
\item
Assume the situation as above with $\rmT=$ the maximal torus  in ${\rmG}$. Then $\rmB/\rmT = R_u(\rmB)$= an affine space. Now
one obtains the weak-equivalences: 
\[ \bKH(\rmE\rmB^{gm, m }{\underset {\rmB} \times}\rmX) \simeq \bKH(\rmE\rmB^{gm, m}{\underset {\rmB} \times}\rmB{\underset {\rmT} \times}\rmX) \simeq \bKH(\rmE\rmB^{gm, m}{\underset {\rmT} \times}\rmX) \simeq \bKH(\rmE\rmT^{gm, m}{\underset {\rmT} \times}\rmX)\]
 where the first weak-equivalences
are from the homotopy property of $\bKH$-theory and the observation that $\rmR_u(\rmB)$ is an affine space over $\rmS$. The last weak-equivalence is from Proposition ~\ref{indep.class.sp}. A corresponding result also holds for equivariant $\bKH$-theory replaced by equivariant $\bG$-theory.
\end{enumerate}

\section{Derived completions for equivariant Homotopy K-theory}
\label{der.compl.eq}
This section will be quick review, as almost all the basic results we need on {\it derived completions} have been worked out already in  \cite{C08} and \cite{CJ23}.
As before, we will assume that $\rmS$ is the spectrum of a field and that $\rmG$ is a linear algebraic group defined over $k$.
We denote the restriction  map 
\be \begin{equation}
\label{IG}
\bK(\rmS, \rmG) \ra \bK(\rmS) \mbox{ by } \rho_{\rmG} \mbox{ and the Postnikov-truncation map } \bK(\rmS) \ra \H(\pi_0(\bK(\rmS))) \mbox{ by } <0>.
\end{equation} \ee
\vskip .2cm \noindent
The composite map $\bK(\rmS, \rmG) \ra \H(\pi_0(\bK(\rmS)))$ will be denoted $\tilde \rho_{\rmG}$.  Observe that $\pi_0(\bK(\rmS, \rmG)) = R(\rmG)$= the representation ring of $\rmG$.
Clearly the restriction $\rho_{\rmG}$ induces a surjection on taking $\pi_0$: this is clear, since if $V$ is any finite dimensional representation of
$\rmG$ of dimension $d$, the $d$-th exterior power of $V$ will be a $1$-dimensional representation of $\rmG$. The  Postnikov-truncation map $<0>:  \bK(\rmS) \ra \H(\pi_0(\bK(\rmS)))$  clearly induces an 
isomorphism on $\pi_0$. Therefore, the hypotheses of \cite[Theorem 6.1]{C08} are satisfied and it shows that the 
derived completion functors with respect to the two maps $\rho_{\rmG}$ and $\tilde \rho_{\rmG}$ identify up to weak-equivalence.  Therefore, while we will
always make  use of the completion with respect to the map $\rho_{\rmG}$, we will often find it convenient to identify this completion with the completion with respect to
$\tilde \rho_{\rmG}$.
\vskip .2cm
Observe that the above completion $\rho_{\rmG}$ has the following explicit description.
Let $Mod(\bK(\rmS, \rmG))$ ($Mod(\bK(\rmS ))$) denote the category of module-spectra over $\bK(\rmS, \rmG)$ ($\bK(\rmS )$, \res). Then sending a module spectrum
$M \eps Mod(\bK(\rmS, \rmG))$ to $M \wedgeKG \bK(\rmS )$ and then viewing it as a $\bK(\rmS, \rmG)$-module spectrum using the ring map $\bK(\rmS, \rmG) \ra K(\rmS )$ defines
a triple. The corresponding  cosimplicial object of spectra is given by
\vskip .2cm
${\mathcal T}^{\bullet}_{\bK(\rmS, \rmG)}(M, \bK(\rmS )): M \cosimp1 M \wedgeKG \bK(\rmS )  \cosimp1 \cdots M \wedgeKG \bK(\rmS ) \wedgeKG \cdots \wedgeKG \bK(\rmS ) \cdots $
\vskip .2cm \noindent
with the obvious structure maps. Then the derived completion of $M$ with respect to $\rho_{\rmG}$  is  
\be \begin{equation}
     \label{der.compl.1}
M\, \, \compl_{\rho_{\rmG}} = \holimD {\mathcal T}^{\bullet}_{\bK(\rmS, \rmG)}(M, \bK(\rmS )).
\end{equation} \ee
The corresponding partial derived completion to degree $m$ will be denoted $\rho_{\rmG, m}$.
\vskip .2cm \noindent
We may also consider the following variant. 
\vskip .2cm
If $\ell$ is a prime different from the residue characteristics, $\rho_{\ell}: \mbS \ra \H(\Z/\ell)$ will denote the obvious map from the sphere spectrum to the mod$-\ell$ Eilenberg-Maclane spectrum.
Then we let
\be \begin{equation}
\label{IGl}
\rho_{\ell} \circ \rho_{\rmG}:\bK(\rmS, \rmG) \ra \bK(\rmS )/\ell= \bK(\rmS ) {\underset {\mbS} \wedge} \H({\mathbb Z}/\ell) 
\end{equation} \ee
\vskip .2cm \noindent
denote the composition of $\rho_{\rmG}$ and the mod-$\ell$ reduction map $id_{\bK(\rmS )} \wedge \rho_{\ell}$. 
The derived completion with respect to $\rho_{\ell} \circ \rho_{\rmG}$ clearly has a description similar to the one in ~\eqref{der.compl.1}.
\vskip .2cm
If $\rmX$ is a scheme or algebraic space provided with an action by the group $\rmG$, we let 
\be \begin{equation}
\begin{split}
     \label{KGl}
\bKH({\rm X}, \rmG)_{\ell} = {\overset {n} {\overbrace {\bKH({\rm X}, \rmG){\underset {\mbS} \wedge} \H(\Z/\ell) \cdots 
{{\underset {\mbS} \wedge } \H(\Z/\ell)}}}} \\
\bKH(\rmE\rmG^{gm}{\underset \rmG \times}X)_{\ell} = {\overset {n} {\overbrace {\bKH(\rmE\rmG^{gm}{\underset \rmG \times}X){\underset {\mbS} \wedge} \H(\Z/\ell) \cdots 
{{\underset {\mbS} \wedge } \H(\Z/\ell)}}}}
\end{split}
    \end{equation} \ee
\vskip .2cm \noindent 
for some positive integer $n$, where the smash products above are all derived smash products. If it becomes important to indicate the $n$ above, we will denote $\bKH({\rm X}, \rmG)_{\ell} $ 
($\bKH(\rmE\rmG^{gm}{\underset \rmG \times}X)_{\ell}$) by $\bKH({\rm X}, \rmG)_{\ell,n}$ ($\bK(\rmE\rmG^{gm}{\underset \rmG \times}X)_{\ell,n}$, \res).
One defines $\bK({\rm X}, \rmG)_{\ell}$ and $\bK(\rmE\rmG^{gm}{\underset \rmG \times}X)_{\ell}$ similarly.
\vskip .2cm
Next we proceed to compare the derived completions with respect to the map $\rho_{\rmG}:\bK(\rmS, \rmG) \ra \bK(\rmS )$ for a split reductive group $\rmG$
and the derived completion with respect to the map $\rho_{\rmT}:\bK(\rmS, \rmT) \ra \bK(\rmS )$, where $\rmT$ is a maximal torus in $\rmG$.  Given a module
spectrum $M$ over $\bK(\rmS, \rmT)$, $res_*(M)$ will denote $M$ with the induced $\bK(\rmS, \rmG)$-module structure, where
$res:\bK(\rmS, G) \ra \bK(\rmS, \rmT)$ denotes the restriction map. Clearly there is a map: $res_*(M) \ra M$ (which is the identity map on
the underlying spectra) compatible with the restriction map $\bK(\rmS, \rmG) \ra \bK(\rmS, \rmT)$. Therefore one obtains induced maps
$res_*(M){\overset L {\underset {\bK(\rmS, \rmG)} \wedge}} \bK(\rmS ) \cdots {\overset L {\underset {\bK(\rmS, \rmG)} \wedge}} \bK(\rmS ) \ra
M {\overset L {\underset {\bK(\rmS, \rmT)} \wedge}} \bK(\rmS ) \cdots {\overset L {\underset {\bK(\rmS, \rmT)} \wedge}} \bK(\rmS )$, which 
induces a map from the derived completion of $res_*(M)$ with respect to $\rho_{\rmG}$ to the derived completion of $M$ with respect 
to $\rho_{\rmT}$.
(Similar conclusions hold for the maps $\rho_{\ell} \circ \rho_{\rmG}$ and $\rho_{\ell} \circ \rho_{\rmT}$.) These observation enable us
to prove the following theorem: see \cite[Theorem 3.6]{CJ23}.

\begin{theorem}
\label{T.comp.vs.G.comp}
Assume that $\rmS = Spec \, \k$ for a field $k$ and that the split reductive group $\rmG$ has $\pi_1(\rmG)$ torsion-free. 
 Let $M$ be a $-n$-connected module spectrum over $\bK(\rmS, \rmT)$ for some integer $n$. Then the inverse system of maps defined above, 
\[\{res_*(M)  \, {\widehat {}}_{\rho _{\rmG},m} \ra M \, {\widehat {}}_{\rho_{\rmT},m}|m\}\]
induces a  weak-equivalence on taking the homotopy inverse limit as $m \ra \infty$. A corresponding result holds for the completions with respect to $\rho_{\ell} \circ \rho_{\rmG}$ and $\rho_{\ell} \circ \rho_{\rmT}$.
\end{theorem}

\begin{remark} One may take $M= \bKH(\rmX, T)$: then Proposition ~\ref{KH.vanishing} shows that $M= \bKH(\rmX, T)$ is $-n$-connected for large $n$.
 For the applications in the rest of the paper, we will restrict mostly to the case where the group $\rmG=\GL_n$ for some $n$ or a finite product $\Pi_{i=1}^m \GL_{n_i}$, since the
hypotheses of Theorem ~\ref{T.comp.vs.G.comp}are satisfied in this case.
\end{remark}

\section{Reduction to the case of a torus in ${\rmG}$ and proof of Theorem ~\ref{main.thm.1}}
\label{reduct.torus}
Recall the following hypotheses of Theorem ~\ref{main.thm.1}: let $\rmX$ denote a scheme of finite type over the
base scheme $\rmS = Spec \, {\it k}$, and provided with an action by the not-necessarily connected linear algebraic group $\rmG$.
Let $\tilde \rmG$ denote a finite product of $\GL_n$s containing $\rmG$ as a closed subgroup scheme. 
We begin with the following Proposition which shows we may reduce to just considering actions by the ambient group $\tilde \rmG$.
\begin{proposition}
 If the derived completion with respect to the map $\rho_{\tilde \rmG}$ induces a weak-equivalence:
 \[\bKH(\tilde {\rmG}{\underset {\rmG}  \times}\rmX, \tilde {\rmG})\compl_{\rho_{\tilde \rmG}} {\overset {\simeq} \ra} \bKH(\rmE\tilde \rmG^{gm}{\underset {\tilde \rmG} \times} (\tilde \rmG{\underset {\rmG} \times}\rmX))\]
then one obtains a weak-equivalence:
\[\bKH(\rmX, \rmG)\compl_{\rho_{\rmG}} {\overset {\simeq} \ra} \bKH(\rmE\tilde \rmG^{gm}{\underset {\tilde \rmG} \times} (\tilde \rmG{\underset {\rmG} \times}\rmX)) \simeq \bKH(\rmE\tilde \rmG^{gm}{\underset {\rmG} \times} \rmX).\]
\end{proposition}
\begin{proof} 
We first make use of 
the identification $\bKH({\rm X}, \rmG) \simeq \bKH(\tilde {\rmG}{\underset {\rmG}  \times}{\rm X}, \tilde {\rmG})$ (see Lemma ~\ref{key.pairings}),  and Proposition ~\ref{key.obs.2} 
to obtain the weak-equivalence: $\bKH(\rmX, \rmG) \compl_{\rho_{\rmG}} \simeq \bKH(\tilde {\rmG}{\underset {\rmG}  \times}\rmX, \tilde {\rmG})\compl_{\rho_{\tilde \rmG}}$. 
The identification 
\newline 
$\bKH(\rmE\tilde \rmG^{gm}{\underset {\rmG}  \times}\rmX) \simeq 
\bKH(\rmE\tilde {\rmG}^{gm}{\underset {\tilde \rmG}  \times}(\tilde {\rmG}{\underset {\rmG}  \times}\rmX))$ is clear.
\end{proof}
Observe also that if $\rmX$ is a {\it normal quasi-projective} scheme over $\rmS$, then so is $\tilde \rmG {\underset \rmG \times} \rmX$. Therefore, we may assume without loss of generality, in the rest of this section that $\rmG= GL_n$ for some $n \ge 1$ or a finite product of such ${\rm GL}_n$s.
\vskip .2cm
Next we recall the construction of the geometric classifying spaces when the linear algebraic group $\rmG$ is ${\rm GL}_n$, a finite product
of ${\rm GL}_{n_j}s$ and when $\rmG =\rmT$ is a split torus. In the case of
 ${\rm GL_n}$, we let $\rmV$ denote a faithful representation of ${\rm GL}_n$, $\rmW = End(\rmV)$ and then choose $\rmE\rmG^{gm,i}$ as in 
 ~\eqref{adm.gadget.1}, ~\eqref{EG.fin}. When $\rmG = \Pi_{j=1}^m {\rm GL}_{n_j}$, we let $\rmE\rmG^{gm,i}= \Pi_{j=1}^m \rmE{\rm GL}_{n_j}^{gm,i}$.
 When $\rmG= \rmT$ is a split torus, we let $\rmE\rmG^{gm,i}$ be defined as in ~\eqref{EG.toric.case}.
\vskip .2cm
Then we obtain  the following preliminary result.
\begin{proposition} (See \cite[Proposition 2.2]{CJ23}.)
 \label{coveringlemma} If $\rmG=\GL_n$, a finite product of $\GL_n$s or a split torus, then there exists a finite Zariski open covering, $\{V_{\alpha}|\alpha\}$ of each $\rmB\rmG^{gm,i}$ with each
$\rmV_{\alpha}$ being ${\mathbb A}^1$-acyclic, that is,  each $\rmV_{\alpha}$ has a $k$-rational point $p_{\alpha}$ so that $\rmV_{\alpha}$ and $p_{\alpha}$ are equivalent in the ${\mathbb A}^1$-homotopy category.
In fact, one may choose a common $k$-rational point $p$ in the intersection of all the open $\rmV_{\alpha}$. Moreover, in all three cases, one may choose
a Zariski open covering of $\rmB\rmG^{gm,i}$ of the above form,  to consist of $i+1$ or more open sets.
\end{proposition}
\vskip .2cm
\begin{remark}
\label{two.class.spaces}
Let $\rmG=\rmT={\mathbb G}_m^n$. Now we have already two constructions for $\rmE\rmT^{gm, i}$.  The first, which we
denote by  $\rmE\rmT^{gm,i}(1)$ starts with $\rmE\rmT^{gm,1}(1) = \GL_n$ with the successive $\rmE\rmT^{gm,i}(1)$ obtained by applying the construction in 
~\eqref{adm.gadget.1} with $\rmE\rmT^{gm,i}(1) = \rmE\GL_n^{gm,i}$. Let $\rmE\rmT^{gm, i}(2)$ denote the second construction with $\rmE\rmT^{gm,i} = ({\mathbb A}^{i+1}-0)^n$ as in ~\eqref{EG.toric.case}. The explicit relationship between two such constructions
is often delicate. The main result that is relevant to us in this context is Proposition ~\ref{indep.class.sp}.
\end{remark}

\begin{proposition} 
\label{factoring.through.part.compl}
Let $\rmG$ denote $\GL_n$ for some $n$ or a split torus. Let $\rmX$ denote a scheme  of finite 
type over $k$. Then, with the choice of the classifying spaces as in ~\eqref{adm.gadget.1} and ~\eqref{EG.toric.case}, the following hold.
\vskip .1cm
(i) If $\rmI_G$ denotes 
the homotopy fiber of the restriction map
 $\bK(\Speck, \rmG) \ra \bK(\Speck)$, then the (obvious) map 
\[\bKH({\rm X}, \rmG) \ra \bKH(\rmE\rmG^{gm, m} \times {\rm X}, \rmG) 
\simeq \bKH(\rmE\rmG^{gm,m} \times _{\rmG} \rmX)\]
 factors
through $\bKH({\rm X}, \rmG)/\rmI_{\rmG}^{Lm'+1}\bKH({\rm X}, \rmG)$ for some positive integer $m' \ge m$. Here
\[\rmI_{\rmG}^{L m'+1}\bKH({\rm X}, \rmG) = {\overset {m'+1} {\overbrace { \rmI_{\rmG} {\overset L {\underset {\bK(\Speck, \rmG)} \wedge}}  \rmI_{\rmG}  \cdots {\overset L {\underset {\bK(\Speck, \rmG)} \wedge}}  \rmI_{\rmG}}}}{\overset L {\underset {\bK(\Speck, \rmG)} \wedge}} \bKH({\rm X}, \rmG),\]
$\rmI_{\rmG}$ is the homotopy fiber of the restriction $\bK(\Speck, \rmG) \ra \bK(\Speck)$ and 
$\bKH({\rm X}, \rmG)/\rmI_{\rmG}^{Lm'+1}\bKH({\rm X}, \rmG)$ is the homotopy cofiber of the map $\rmI_{\rmG}^{Lm'+1}\bKH({\rm X}, \rmG) {\overset L {\underset {\bK(\Speck, \rmG)} \wedge}} \bKH({\rm X}, \rmG) \ra \bKH({\rm X}, \rmG)$.
\vskip .1cm
(ii) It follows that the map $\bKH({\rm X}, \rmG) \ra  
\bKH(\rmE\rmG^{gm,m} \times {\rm X}, \rmG)$ factors through the partial
derived completion $\bKH({\rm X}, \rmG)\compl_{\rho_{\rmG}, m'}$ for some positive integer $m' \ge m+1$. If $\pi: \rmX \ra \rmY$ is any $\rmG$-equivariant map ($\rmG$-equivariant proper map which is also perfect as in Definition ~\ref{pseudo.coh}), then
the above factorization is compatible with the induced map $\pi^*:\bKH(\rmY, \rmG) \ra \bKH({\rm X}, \rmG)$ 
(the induced map $\pi_*:\bKH({\rm X}, \rmG) \ra \bKH(\rmY, \rmG)$, \res.)
\vskip .2cm
(iii) More generally, if $\rmG_i$, $i=1, \cdots, q$ are either general linear groups or split tori, the map
\[\bKH({\rm X}, \rmG_1 \times \cdots \times \rmG_q) \ra \bKH(\rmE\rmG_1^{gm, m} \times \cdots \times \rmE\rmG_q^{gm, m} \times {\rm X}, \rmG_1 \times \cdots \times \rmG_q)\]
factors through 
\[\bKH({\rm X}, \rmG_1 \times \cdots \times \rmG_q) \compl_{\rho_{\rmG_1 \times \cdots \rmG_q}, m'}\] for 
some positive integer $m' \ge m+1$. 
 If $\pi:\rmX \ra \rmY$ is any equivariant map with respect to the action of $\rmG= \rmG_1 \times \cdots \times \rmG_q$ map (which is also proper and  perfect as well), then
the above factorization is compatible with the induced map $\pi^*:\bKH(\rmY, \rmG_1 \times \cdots \times \rmG_q) \ra \bKH({\rm X}, \rmG_1 \times \cdots \times \rmG_q)$ (the induced map $\pi_*:\bKH({\rm X}, \rmG_1 \times \cdots \times \rmG_q) \ra \bKH(\rmY, \rmG_1 \times \cdots \times \rmG_q)$, \res.)
\end{proposition}
\begin{proof} This parallels the original proof in \cite{AS69} for the usual completion for equivariant topological K-theory. We proved
a corresponding result for equivariant G-theory in \cite[Proposition 4.5]{CJ23}. There, some of the key arguments needed are shown
 to hold for equivariant K-theory, and then one makes use of the fact that the equivariant G-theory spectrum is a module spectrum over
  the corresponding equivariant K-theory spectrum. In fact, one may invoke the proof there to observe that the map
\[\bK(\Speck, \rmG) \ra \bK(\rmE\rmG^{gm, m}, \rmG) \mbox { factors through } \bK(\Speck, \rmG)/\rmI_{\rmG}^{L\wedge ^n}\]
for some $n \ge m$, where $\bK(\Speck, \rmG)/\rmI_{\rmG}^{L\wedge ^n}$ is the homotopy cofiber of the map $\rmI_{\rmG}^{L\wedge ^n} \ra \bK(\Speck, \rmG)$. Finally one makes use of the pairing:
$\bK(\Speck, \rmG) {\overset L \wedge} \bKH({\rm X}, \rmG) \ra \bKH({\rm X}, \rmG)$ and the homotopy commutative square (where the vertical maps are the obvious ones)
\be \begin{equation}
     \label{null.homotopic}
\xymatrix{{{\rmI_{\rmG}^{L\wedge ^n}}  {\overset {\rm L} {\wedge}} \bKH({\rm X}, \rmG)} \ar@<1ex>[r] \ar@<1ex>[d] & {\bK(\Speck, \rmG)  {\overset L {\wedge}} \bKH({\rm X}, \rmG)} \ar@<1ex>[r] \ar@<1ex>[d] & {\bKH({\rm X}, \rmG)} \ar@<1ex>[d]\\
{\bK(\rmE\rmG^{gm, m}, \rmG)  {\overset L {\wedge}} \bKH({\rm X}, \rmG)} \ar@<1ex>[r] & {\bK(\rmE\rmG^{gm, m}, \rmG)  {\overset L {\wedge}} \bKH({\rm X}, \rmG)} \ar@<1ex>[r] & \bKH(\rmE\rmG^{gm,m} \times {\rm X}, \rmG),}
\end{equation} \ee
\vskip .2cm \noindent
which shows the composite map from the top left-corner to the bottom-right corner is null-homotopic.  
This proves the first statement since we have let $m'=n\ge m+1$.  
\vskip .2cm
Next we consider the second statement as well as the functoriality of the factorization in (i). Recall (see \cite[Corollary 6.7]{C08}) that the partial derived 
completion $\bKH({\rm X}, \rmG)\compl_{\rho_{\rmG}, m'} =$ 
the homotopy cofiber of the map
\[{\overset {m'+1} {\overbrace {\tilde \rmI_{\rmG} {\underset {\bK(\Speck, \rmG)} \wedge} \tilde \rmI_{\rmG}  \cdots {\underset {\bK(\Speck, \rmG)} \wedge} \tilde \rmI_G}}} 
 {\underset {\bK(\Speck , \rmG)} \wedge} {\widetilde {\bKH({\rm X}, \rmG)}} \ra 
 \rmI_G^{\wedge^{m'+1}} {\underset {\bK(\Speck , \rmG)} \wedge} {\widetilde {\bKH({\rm X}, \rmG)}} \]
\[ \ra \bK(\Speck, \rmG) {\underset {\bK(\Speck , \rmG)} \wedge} {\widetilde {\bKH({\rm X}, \rmG)}} = {\widetilde {\bKH({\rm X}, \rmG)}}\]
which maps into  $(\bK(\Speck, \rmG)/I_{\rmG}^{L\wedge^{m'+1}} ){\underset {\bK(\Speck , \rmG)} \wedge} {\widetilde {\bKH({\rm X}, \rmG)}}$.
Here $\tilde \rmI_{\rmG} \ra \rmI_{\rmG}$ (${\widetilde {\bKH({\rm X}, \rmG)}} \ra \bKH({\rm X}, \rmG)$) is a cofibrant replacement in the category of $\bK(\Speck, \rmG)$-module spectra. This proves 
the first statement in (ii).  The
functoriality of the factorization in (i) and (ii) follows readily by observing that 
both $\pi^*$ (when $\pi$ is flat) and $\pi_*$ (when $\pi$ is proper and also perfect)
are module maps over $\bK(\Speck, \rmG)$. Therefore, these maps will induce maps between the diagrams in ~\eqref{null.homotopic} corresponding to
$\rmX$ and $\rmY$. 
\vskip .2cm
To prove the last statement, one first observes that $\bK({\rm X}, \rmG_1 \times \cdots \times \rmG_q)$ is a module
spectrum over $\bK(\Speck, \rmG_1 \times \cdots \times \rmG_q)$. Now the homotopy commutative square
 (where the vertical maps are the obvious ones)
\vskip .2cm
\xymatrix{{\bK(\Speck, \Pi_{i=1}^q\rmG_i ) \wedge \bKH({\rm X}, \Pi_{i=1}^q \rmG_i)} \ar@<1ex>[r] \ar@<1ex>[d] & {\bKH({\rm X}, \Pi_{i=1}^q \rmG_i )} \ar@<1ex>[d]\\
{\bKH(\rmE\rmG_1^{gm, m} \times \cdots \rmE\rmG_n^{gm, m}, \Pi_{i=1}^q\rmG_i )  \wedge \bKH({\rm X}, \Pi_{i=1}^q\rmG_i )} \ar@<1ex>[r] & \bKH(\rmE\rmG_1^{gm,m} \times \cdots \rmE\rmG_n^{gm, m} \times {\rm X}, \Pi_{i=1}^q\rmG_i)}
\vskip .2cm \noindent
and an argument exactly as in the case of a single group, proves the factorization in the last statement. The functoriality of this factorization in $\rmX$ may be proven as in the case
of a single group. 
\end{proof}

\vskip .2cm 
\begin{definition}
 \label{funct.alpha} For each linear algebraic group $\rmG$ acting on a scheme $\rmX$ satisfying the  hypotheses as in ~\ref{stand.hyp.1} and each
choice of geometric classifying spaces as in ~\ref{EG.general}, we will
define a function $\alpha:\N \ra \N$, by $\alpha(m) =m'$ where $m' \ge m+1$ is some choice of $m'$ as in the last Proposition.
\end{definition}

\begin{lemma}
 \label{comp.towers}
 Let $\rmG$ denote either ${\rm GL}_n$ or a finite product of general linear groups. Then adopting the terminology above, one
 obtains a strictly commutative diagram  of towers
 \[\xymatrix{ {} \ar@<1ex>[d] & {} \ar@<1ex>[d]\\
 {\bK(\Speck, \rmG)\compl_{\rho_{\rmG}, \alpha(m+1)} = \bK( \Speck, \rmG)/\rmI_G^{L \wedge \alpha (m+1)}} \ar@<1ex>[r] \ar@<1ex>[d] & {\bK(\rmE\rmG^{gm, m+1}, \rmG)} \ar@<1ex>[d]\\
               {\bK(\Speck, \rmG)\compl_{\rho_{\rmG}, \alpha(m)} =\bK( \Speck, \rmG)/\rmI_G^{L \wedge \alpha (m)}} \ar@<1ex>[r] \ar@<1ex>[d] & {\bK(\rmE\rmG^{gm, m}, \rmG)} \ar@<1ex>[d]\\
                 {}  & {} }
 \]
\end{lemma}
\begin{proof}
 One first applies  Lemma ~\ref{tower.1} to the strictly commutative diagram
  \[\xymatrix{{{\rmI_{\rmG}^{L\alpha(m+1)}} } \ar@<1ex>[r] \ar@<1ex>[d] & {\bK(\Speck, \rmG)} \ar@<1ex>[r] \ar@<1ex>[d]^{id} & {\bK(\rmE\rmG^{gm, m+1}, \rmG)} \ar@<1ex>[d]\\
             {\rmI_{\rmG}^{L\alpha(m)}} \ar@<1ex>[r] & {\bK(\Speck, \rmG)} \ar@<1ex>[r] &{\bK(\rmE\rmG^{gm, m}, \rmG)}} \]
which results in the homotopy commutative diagram:
\be \begin{equation}
\label{comp.towers.1}
\xymatrix{ {\bK(\Speck, \rmG)\compl_{\rho_{\rmG}, \alpha(m+1)} = \bK( \Speck, \rmG)/\rmI_G^{L \wedge \alpha (m+1)}} \ar@<1ex>[r] \ar@<1ex>[d] & {\bK(\rmE\rmG^{gm, m+1}, \rmG)} \ar@<1ex>[d]\\
               {\bK(\Speck, \rmG)\compl_{\rho_{\rmG}, \alpha(m)} =\bK( \Speck, \rmG)/\rmI_G^{L \wedge \alpha (m)}} \ar@<1ex>[r] & {\bK(\rmE\rmG^{gm, m}, \rmG)}.} 
\end{equation} \ee
Now one invokes Lemma ~\ref{tower.fibs} with the fibration $q_m$ denoting the right vertical map to complete the proof.
\end{proof}

\begin{proposition}
 \label{comp.towers.2}
 Assume $\rmG$ is as in Lemma ~\ref{comp.towers}, and that $f: \rmX \ra \rmY$ is a $\rmG$-equivariant map ($f: \rmY \ra \rmX$ is a $\rmG$-equivariant 
 map that is also proper and perfect).
 Then one obtains a  diagram (of towers)

\be \begin{equation}
 \label{tower.diagm}
 \xymatrix{ {\bKH(\rmY, \rmG)\compl_{\rho_{\rmG}, \alpha(m+1)}} \ar@<1ex>[ddd]^{} \ar@<1ex>[dr]^{h_{m+1}} \ar@<1ex>[rrr]^{}  &&&  {\bKH(\rmY, \rmG)\compl_{\rho_{\rmG}, \alpha(m)}} \ar@<1ex>[dl]_{h_m} \ar@<1ex>[ddd]^{}\\
  & {\bKH(\rmX, \rmG)\compl_{\rho_{\rmG}, \alpha(m+1)}} \ar@<1ex>[r]^(.4){} \ar@<1ex>[d]^{} & {\bKH(\rmX, \rmG)\compl_{\rho_{\rmG}, \alpha(m)}}  \ar@<1ex>[d]^{} \\
  &{\bKH(\rmE\rmG^{gm, m+1}\times \rmX, \rmG)} \ar@<1ex>[r]^(.4){} &{\bKH(\rmE\rmG^{gm, m}\times \rmX, \rmG)}\\
  {\bKH(\rmE\rmG^{gm, m+1}\times \rmY, \rmG)} \ar@<1ex>[ur]^{g_{m+1}}  \ar@<1ex>[rrr]^{}   &&&    {\bKH(\rmE\rmG^{gm, m}\times \rmY, \rmG)} \ar@<1ex>[ul]_{g_m}}
\end{equation} \ee
that strictly commutes for $m \ge 0$, for pull-back by $f$ (coherently homotopy commutes for $m \ge 0$, for derived push-forward by $f$, when $f$ is a $\rmG$-equivariant 
 map that is also proper and perfect, \res). Here the maps $\{h_m|m\}$ and $\{g_m|m\}$ are the maps induced by pull-back by $f$ (push-forward by $Rf_*$, \res).
Therefore one obtains
the homotopy commutative diagram:
\be \begin{equation}
\label{holim.diagm}
\xymatrix{{\holimm \{\bKH(\rmY, \rmG)\compl_{\rho_{\rmG}, \alpha(m)}|m\}} \ar@<1ex>[d] \ar@<1ex>[r] & {\holimm \bKH(\rmE\rmG^{gm, m}\times \rmY, \rmG)} \ar@<1ex>[d]\\
           {\holimm \{\bKH(\rmX, \rmG)\compl_{\rho_{\rmG}, \alpha(m)}|m\} } \ar@<1ex>[r] & {\holimm \bKH(\rmE\rmG^{gm, m}\times \rmX, \rmG)}}
\end{equation} \ee
The corresponding statements also hold for the equivariant homotopy K-theory replaced by equivariant G-theory provided
the map $f$ is also flat when considering pull-backs by $f$ (proper  when considering
push-forwards, \res). 
\end{proposition}          

\begin{proof} We will explicitly consider only the case of equivariant homotopy K-theory, since the proof for equivariant G-theory is
entirely similar.
We take the derived smash product of the strictly commutative diagram provided by Lemma ~\ref{comp.towers} over
$\bK(\Speck, \rmG)$ with $\bKH(\rmY, \rmG)$ to obtain the strictly commutative diagrams:
\fontsize{8}{12}
\be \begin{equation}
\label{tower.diagm.1}
\xymatrix{{\bK( \Speck, \rmG)/\rmI_G^{L \wedge \alpha (m+1)} {\overset L {\underset {\bK(\Speck, \rmG)} \wedge}} {\bKH(\rmY, \rmG)}} \ar@<1ex>[r] \ar@<1ex>[d] & {\bK(\rmE\rmG^{gm, m+1}, \rmG) {\overset L {\underset {\bK(\Speck, \rmG)} \wedge}} {\bKH(\rmY, \rmG) }} \ar@<1ex>[r] \ar@<1ex>[d]& {\bKH(\rmE\rmG^{gm, m+1} \times \rmY, \rmG)  } \ar@<1ex>[d] \\
            {\bK( \Speck, \rmG)/\rmI_G^{L \wedge \alpha (m)} {\overset L {\underset {\bK(\Speck, \rmG)} \wedge}} {\bKH(\rmY, \rmG)}} \ar@<1ex>[r] &   {\bK(\rmE\rmG^{gm, m}, \rmG) {\overset L {\underset {\bK(\Speck, \rmG)} \wedge}} {\bKH(\rmY, \rmG)   }} \ar@<1ex>[r] &{\bKH(\rmE\rmG^{gm, m} \times \rmY, \rmG)   }  ,}
\end{equation} \ee
\be \begin{equation}
\label{tower.diagm.2}
\xymatrix{{\bK( \Speck, \rmG)/\rmI_G^{L \wedge \alpha (m+1)} {\overset L {\underset {\bK(\Speck, \rmG)} \wedge}} {\bKH(\rmX, \rmG)}} \ar@<1ex>[r] \ar@<1ex>[d] & {\bK(\rmE\rmG^{gm, m+1}, \rmG) {\overset L {\underset {\bK(\Speck, \rmG)} \wedge}} {\bKH(\rmX, \rmG) }} \ar@<1ex>[r] \ar@<1ex>[d]& {\bKH(\rmE\rmG^{gm, m+1} \times \rmX, \rmG)  } \ar@<1ex>[d] \\
            {\bK( \Speck, \rmG)/\rmI_G^{L \wedge \alpha (m)} {\overset L {\underset {\bK(\Speck, \rmG)} \wedge}} {\bKH(\rmX, \rmG)}} \ar@<1ex>[r] &   {\bK(\rmE\rmG^{gm, m}, \rmG) {\overset L {\underset {\bK(\Speck, \rmG)} \wedge}} {\bKH(\rmX, \rmG)   }} \ar@<1ex>[r] &{\bKH(\rmE\rmG^{gm, m} \times \rmX, \rmG)   }  .}
\end{equation} \ee
\normalsize
The strict commutativity of the last squares in ~\eqref{tower.diagm.1} and ~\eqref{tower.diagm.2} follow from the
functoriality of pairings in equivariant K-theory and equivariant homotopy K-theory.
One may now observe that the outer (inner) square in ~\eqref{tower.diagm} corresponds to the diagram ~\eqref{tower.diagm.1} (~\eqref{tower.diagm.2}, \res), with
 the top horizontal map (the second horizontal map) in ~\eqref{tower.diagm} corresponding to the left vertical map in ~\eqref{tower.diagm.1} (in ~\eqref{tower.diagm.2}, \res). (Similarly the left vertical map  (the second vertical map) in ~\eqref{tower.diagm}
  corresponds to the composite map forming the top row in the diagram ~\eqref{tower.diagm.1} (~\eqref{tower.diagm.2}, \res). 
\vskip .1cm
Since K-theory and therefore, homotopy K-theory is contravariantly functorial, and so are the pairings between K-theory and homotopy K-theory appearing in the 
last squares of the diagrams ~\eqref{tower.diagm.1} and ~\eqref{tower.diagm.2}, we first see that if $f: \rmX \ra \rmY$ is a $\rmG$-equivariant map, 
pull-back by $f$ induces a map of the diagram ~\eqref{tower.diagm.1} to the diagram ~\eqref{tower.diagm.2} so that all the resulting squares strictly commute. This proves the statement on the strict commutativity
of the diagram ~\eqref{tower.diagm} for pull-back by $f$, which results in the commutativity of the corresponding diagram ~\eqref{holim.diagm}. 
\vskip .1cm
Next we consider the derived push-forward by the map $f$. In view of Lemma ~\ref{comp.towers}, clearly the derived push-forward by $f$ induces a map from the left-square in ~\eqref{tower.diagm.1}
 to the left square in ~\eqref{tower.diagm.2}, showing that the resulting diagram strictly commutes. Therefore, it suffices to show that the derived 
 push-forward by $f$ induces a map from the right square in ~\eqref{tower.diagm.1} to the right square in ~\eqref{tower.diagm.2}
 resulting in a  coherently homotopy commutative diagram. 
 This should be clear, since the pairings in the right squares in ~\eqref{tower.diagm.1} and ~\eqref{tower.diagm.2}
 are {\it natural}, and in view of the covariant functoriality of the equivariant homotopy K-theory spectrum with respect to proper equivariant maps that are also perfect as pointed out in
 ~\ref{KH.props}(1): however, we will provide a  detailed discussion to show this is indeed the case. 
 \vskip .1cm
 Therefore, we will now consider the following diagram, where the outer square (inner square) corresponds to the right square in ~\eqref{tower.diagm.1} (~\eqref{tower.diagm.2}, \res), and where $\wedge$ denotes
 ${\overset {\rm L} {\underset {\bK(\Speck, \rmG)} \wedge}}$:
 \fontsize{8}{12}
 \be \begin{equation}
   \label{tower.pairings}
  \xymatrix{ {\bK(\rmE\rmG^{gm,m+1}, \rmG)\wedge \bKH(\rmY, \rmG)} \ar@<1ex>[ddd]^{} \ar@<1ex>[dr]^{id \wedge f_*} \ar@<1ex>[rrr]^{}  &&&  {\bKH(\rmE\rmG^{gm, m+1} \times \rmY, \rmG)} \ar@<1ex>[dl]_{(id\times f)_*} \ar@<1ex>[ddd]^{}\\
  & {\bK(\rmE\rmG^{gm, m+1}, \rmG) \wedge \bKH(\rmX, \rmG)} \ar@<1ex>[r]^(.4){} \ar@<1ex>[d]^{} & {\bKH(\rmE\rmG^{gm, m+1} \times \rmX, \rmG)}  \ar@<1ex>[d]^{} \\
  &{\bK(\rmE\rmG^{gm, m}, \rmG) \wedge \bKH(\rmX, \rmG)} \ar@<1ex>[r]^(.4){} &{\bKH(\rmE\rmG^{gm, m}\times \rmX, \rmG)}\\
  {\bK(\rmE\rmG^{gm, m}, \rmG) \wedge \bKH( \rmY, \rmG)} \ar@<1ex>[ur]^{id \wedge f_*}  \ar@<1ex>[rrr]^{}   &&&    {\bKH(\rmE\rmG^{gm, m}\times \rmY, \rmG)} \ar@<1ex>[ul]_{(id \times f)_*}}
 \end{equation} \ee
\normalsize
 The functoriality of homotopy K-theory with respect to pull-backs and pairings, shows immediately that the outer square, the inner square and the 
 left-most square strictly commute. We proceed to consider the remaining squares.
  \vskip .1cm 
 Next we consider the commutativity of the right square in ~\eqref{tower.pairings}. Now it is important to observe that the squares
 \[\xymatrix{{\rmE\rmG^{gm,m} \times \rmY} \ar@<1ex>[d]^{id \times f} \ar@<1ex>[r]^{i_m \times id_Y} & {\rmE\rmG^{gm,m+1} \times \rmY} \ar@<1ex>[d]^{id \times f}\\
             {\rmE\rmG^{gm,m} \times \rmX} \ar@<1ex>[r]^{i_m \times id_{\rmX}} & {\rmE\rmG^{gm,m+1} \times \rmX}}
 \]
{\it are Tor independent} (where two schemes ${\rm A}$ and ${\rm B}$ over ${\rm C}$ are Tor-independent if $Tor_i^{\O_{\rm C}}(\O_{\rm A}, \O_{\rm B}) =0$ for all $i>0$), so that the base-change morphism $L(i_m \times id_X)^* R(id \times f)_* (F) \ra R(id\times f)_* L(i_m \times id_Y)^*(F)$ is a
quasi-isomorphism for all pseudo-coherent $\rmG$-equivariant complexes $F$ on ${\rm E}{ \rmG}^{gm,m+1}\times \rmY$: see \cite[36.21: Cohomology and Base Change]{Stacks}.
By defining the right derived functors using the canonical Godement resolutions and the left derived functors using functorial flat resolutions (in fact, using
the Koszul complexes associated to the regular closed immersions $\rmE\rmG^{gm,m} \ra \rmE\rmG^{gm,m+1} \ra \cdots \ra \rmE\rmG^{gm,m+k}$), we see that 
 the above base-change morphisms are natural, and that there are higher order base-change identities corresponding to chains of Tor-independent squares:
\[\xymatrix{ {\rmE\rmG^{gm, m} \times \rmY} \ar@<1ex>[d]^{id \times f} \ar@<1ex>[r]^{i_m \times id_{\rmY}} & {\rmE\rmG^{gm, m+1} \times \rmY} \ar@<1ex>[d]^{id \times f} \ar@<1ex>[r]^(.7){i_{m+1} \times id_{\rmY}}& {\cdots} \ar@<1ex>[r] &   {\rmE\rmG^{gm, m+k} \times \rmY} \ar@<1ex>[d]^{id \times f}\\
              {\rmE\rmG^{gm, m} \times \rmX} \ar@<1ex>[r]^{i_m \times id_{\rmY}} &  {\rmE\rmG^{gm, m+1} \times \rmY} \ar@<1ex>[r]^(.7){i_{m+1} \times id_{\rmX}} &{\cdots} \ar@<1ex>[r] & {\rmE\rmG^{gm, m+k} \times \rmX}}
\]
which then make the right square in ~\eqref{tower.pairings}
coherently homotopy commutative. These observations show that the maps denoted $(id \times f)_*$ in ~\eqref{tower.pairings}
provides a coherently homotopy commutative map of towers:
\be \begin{equation}
    \label{right.square}
 \xymatrix{ {} \ar@<1ex>[d] & {} \ar@<1ex>[d]\\
 {\bKH(\rmE\rmG^{gm, m+1} \times \rmY, \rmG)} \ar@<1ex>[r]^{(id \times f)_*} \ar@<1ex>[d] & {\bKH(\rmE\rmG^{gm, m+1}\times \rmX, \rmG)} \ar@<1ex>[d]\\
     {\bKH(\rmE\rmG^{gm, m} \times \rmY, \rmG)} \ar@<1ex>[r]^{(id \times f)_*} \ar@<1ex>[d]& {\bKH(\rmE\rmG^{gm, m} \times \rmX, \rmG)} \ar@<1ex>[d]\\
     {} & {} }
\end{equation} \ee
(In fact one may first show this with the spectrum $\bKH$ replaced by ${\mathbb K}$: then the functoriality of the constructions involved in 
the passage from ${\mathbb K}$ to $\bKH$ shows the same conclusion holds for the spectrum $\bKH$.)
The construction of the homotopy limit of a tower, involving all the higher order homotopies between the various compositions of
the structure maps in the tower, shows that one obtains an induced map forming the right vertical map in ~\eqref{holim.diagm}.
 \vskip .1cm
 Next we consider the top and bottom squares in ~\eqref{tower.pairings}. Observe that the passage from the equivariant K-theory spaces to
 the corresponding equivariant homotopy K-theory spectra is functorial. Therefore, it suffices to show that the corresponding squares, when the equivariant homotopy K-theory 
 spectra appearing there have been replaced by the corresponding equivariant K-theory spaces are coherently homotopy commutative.
 Here, one may consider a suitable model for K-theory, so we can provide explicit interpretations for K-theory classes.
 For us it may be simplest to adopt the Gillet-Grayson $G$-construction, which interprets the $\rmG$-equivariant K-theory space associated
 to a $\rmG$-scheme $\rmX$ as given by the simplicial set whose $n$-simplices are a pair of sequences of cofibrations  
 \[(K_0 \rightarrowtail K_1 \rightarrowtail \cdots \rightarrowtail K_n, L_0 \rightarrowtail L_1 \rightarrowtail \cdots \rightarrowtail L_n)\]
 with each $K_i, L_i$ a $\rmG$-equivariant perfect complex on $\rmX$, provided with isomorphisms $K_i/K_0 \simeq L_i/L_0$
 compatible with $\rmG$-actions. 
 \vskip .1cm
 The coherent homotopy \footnote{Here {\it coherent homotopy} means by means of naturally chosen higher order homotopies.} commutativity of the top square
 in ~\eqref{tower.pairings} therefore amounts to showing there is a natural quasi-isomorphism $(P \boxtimes Rf_*(Q)) {\overset {\simeq} \ra}  R(id\times f)_*(P\boxtimes Q)$, where $P$ ($Q$) is 
 a $\rmG$-equivariant perfect complex on $\rmE\rmG^{gm, m+1}$ ($\rmY$, \res). To see this, we make use of the canonical
 Godement resolutions to define $Rf_*$. For a bounded below complex $F$, we let $G(F)$ denote the total complex defined as the total complex of the double complex
  obtained from $F$ by applying the canonical Godement resolution to $F$. Then we send, $(P, f_*(G(Q))$ to $P \boxtimes f_*(G(Q)) $, which has a natural map
   to $G(P) \boxtimes f_*(G(Q))= (id \times f)_*(G(P \boxtimes Q))$, and which is a quasi-isomorphism. Therefore, the top square in ~\eqref{tower.pairings} commutes up to homotopy, by a homotopy which is
   natural in the arguments $P$ and $Q$ as well as the level $m+1$ of the towers that are involved in the above pairings. 
   (The last statement simply means the pull-back of the pairings from the $m+1$-th level of the tower to the $m$-th level of the towers defines a corresponding pairing there:
   that is,
   if $i_m: \rmE\rmG^{gm,m} \ra \rmE\rmG^{gm,m+1}$ is the given closed immersion, then the corresponding pairing is given by
   $(i_m^*(P), Q) \mapsto (i_m^*(P), f_*G(Q)) \mapsto (i_m^*(Q) \boxtimes f_*G(P) \mapsto (id \times f)_*(G(i_m^*(P) \boxtimes P))$.
   It follows that the same argument shows the coherent homotopy commutativity of the bottom square in ~\eqref{tower.pairings}.
   (This approach makes it essential to work with flabby resolutions in the category of all $\O$-modules such as those given by
 the Godement resolutions: however, one may make use of the quasi-coherator (as in \cite[Appendix B]{ThTr} to get back functorially into the category of
 quasi-coherent $\O$-modules, if so desired.) 
 \vskip .1cm
 Next we proceed to show that all the maps in the square ~\eqref{holim.diagm} exist. The functoriality of the derived completion functor with respect to 
 $\rho_{\rmG}: \bK(\Speck, \rm G) \ra \bK(\Speck)$ shows the existence of the left vertical map in ~\eqref{holim.diagm}. The map in the 
 top row of ~\eqref{holim.diagm} is induced by the coherently homotopy commutative map of towers forming the diagram  ~\eqref{tower.diagm.1}, while the map in 
 the bottom row of ~\eqref{holim.diagm} is induced by the coherently homotopy commutative map of towers forming the diagram  ~\eqref{tower.diagm.2}.
We already proved the existence of the 
 right vertical map in one of the earlier paragraphs. 
 \vskip .1cm
 Therefore, it suffices to show that the square ~\eqref{holim.diagm} commutes up to homotopy. Since $f_*$ induces a strictly commutative diagram from
 the first square in ~\eqref{tower.diagm.1} to the first square in ~\eqref{tower.diagm.2} it suffices to 
  invoke \cite[the Proposition on p. 270]{CP} with the coherently homotopy commutative diagrams $F$ ($G$) in {\it op.cit}
 denoting the tower 
 \[\{{\bKH(\rmE\rmG^{gm,m+1}, \rmG)\wedge \bKH(\rmY, \rmG)} \ra {\bKH(\rmE\rmG^{gm,m}, \rmG)\wedge \bKH(\rmY, \rmG)}|m\}\]
 defined by the left-most vertical map in ~\eqref{tower.pairings} (the tower
 \[\{{\bKH(\rmE\rmG^{gm,m+1} \times \rmX, \rmG)} \ra {\bKH(\rmE\rmG^{gm,m}\times \rmX, \rmG)}|m\}\]
 defined by the third
 vertical map in ~\eqref{tower.pairings}, \res) and with $f$ (in {\it op. cit}) denoting the composition of the inclined maps in the left-most-square followed by the horizontal maps in the
  middle square of ~\eqref{tower.pairings}. The map $g$ will then denote the map starting at the same source, and going along the top and bottom horizontal maps in ~\eqref{tower.pairings} followed by the inclined maps in the
   right most square there. Observe that the map denoted $f$ is a strict map of towers and therefore a map of coherent diagrams, while the map denoted $g$ is only a 
   homotopy commutative map of such towers. The above Proposition of \cite{CP} now shows one may replace 
   $g$ by a coherently homotopy commutative map $g'$ of towers, which in each degree of the tower agrees with the given map $g$, together with a coherent homotopy from $f$ to $g'$. These observations therefore show that on taking the homotopy inverse limits, one obtains the homotopy commutative square in ~\eqref{holim.diagm}.

 \end{proof}

\vskip .2cm \noindent

The main result of this section will be the following Theorem.
\begin{theorem}
\label{main.thm.3}
Let $\rmG$ denote a finite product of general linear groups acting on a {\it normal quasi-projective} scheme $\rmX$ of finite type over 
a field $k$. Let $\rmT$ denote a split maximal torus in $\rmG$.
 Then the map 
 \[\{\bKH({\rm X}, \rmG)\compl_{\rho_{\rmG}, \alpha(m)}|m\} \ra  \{\bKH({\rm EG}^{\rm gm,m}{\underset {\rmG} \times}\rmX)|m\} \]
 of pro-spectra induces a weak-equivalence on
taking the homotopy inverse limit as $m \ra \infty$, provided the corresponding map 
\[\{\bKH({\rm X}, \rmT)\compl_{\rho_{\rmT}, \alpha(m)}|m\} \ra  \{\bKH({\rm E}{\rmT}^{\rm gm,m}{\underset {\rmT} \times}\rmX)|m\} \]
of pro-spectra induces a weak-equivalence on
taking the homotopy inverse limit as $m \ra \infty$.
\end{theorem}
\vskip .2cm
\begin{proof}
 The main idea of the proof is as follows.
Observe that the maps $\pi^*$ and $\pi_*$ associated to $\pi: \rmG{\underset {\rmB} \times}X \ra X$ induce maps that make the following
diagram homotopy commute, for each fixed $m$:
\be \begin{equation}
     \label{splittings}
\xymatrix{{\bKH({\rm X}, \rmG)\compl_{\rho_{\rmG}, \alpha(m)}} \ar@<1ex>[d]^{\pi^*} \ar@<1ex>[rr] && {\bKH(\rmE{\rmG}^{gm,m}{\underset {\rmG} \times}\rmX)} \ar@<1ex>[d]^{\pi^*} \\
{\bKH({\rm X}, \rmB)\compl_{\rho_{\rmG}, \alpha(m)} \simeq \bKH(\rmG{\underset {\rmB} \times}{\rm X}, \rmG)_{\rho_{\rmG}, \alpha(m)}} \ar@<1ex>[d]^{\pi_*} \ar@<1ex>[rr] && {\bKH(\rmE{\rmG}^{gm,m}{\underset {\rmB} \times}\rmX) \simeq 
\bKH(\rmE\rmG^{gm, m}{\underset {\rmG} \times}(\rmG{\underset {\rmB} \times}\rmX))} \ar@<1ex>[d]^{\pi_*}\\
{\bKH({\rm X}, \rmG)\compl_{\rho_{\rmG}, \alpha(m)}}  \ar@<1ex>[rr] && {\bKH(\rmE{\rmG}^{gm, m}{\underset {\rmG} \times}\rmX)}}.
\end{equation} \ee
\vskip .2cm \noindent
In fact, since $\pi$ is a proper smooth map, Proposition ~\ref{comp.towers.2} shows that we may view the above diagram, as $m$ varies, 
as a level diagram of pro-spectra or towers of spectra. Observe that ${\rmB}/{\rmT}$ is an affine space and hence ${\mathbb A}^1$-acyclic. Therefore, in view of Lemma ~\ref{B.vs.T}, we may replace the Borel subgroup ${\rmB}$ everywhere by its maximal torus 
${\rmT}$.
\vskip .2cm
We will next show how to replace the middle-map up to weak-equivalence by a corresponding map
when the derived completions with respect to $\rho_{\rmG}$ are replaced by derived completions with respect to 
$\rho_{\rmT}$. We will then show that the middle
map induces a weak-equivalence on taking the homotopy inverse limit as $m \ra \infty$. Since the top horizontal map (which is also the bottom horizontal map) 
is a retract of the middle horizontal map (see Proposition ~\ref{comp.towers.2},   ~\eqref{derived.direct}), it follows
that the top horizontal map also induces a weak-equivalence on taking the homotopy inverse limit as $m \ra \infty$. This will prove the theorem.
\vskip .2cm
Therefore, the rest of the proof will be to show that the middle horizontal map induces a weak-equivalence on 
taking the homotopy inverse limit as $m \ra \infty$. Next we invoke \cite[Lemma 3.5]{CJ23} with $A= \bK(\Speck, \rmG)$,
$B= \bK(\Speck, \rmT)$ and $C = \bK(\Speck)$:  we will assume that we have already replaced $C$ with $\tilde C$
and $B$ with $\tilde B$ following the terminology of \cite[Lemma 3.5]{CJ23}. Let ${\rm I}_{\rmG}$ denote the homotopy fiber of 
the composite restriction $\bK(\Speck, \rmG) \ra \bK(\Speck, \rmT) \ra \bK(\Speck)$ and let ${\rm I}_{\rmT}$ denote the
homotopy fiber of the restriction $\bK(\Speck, \rmT) \ra \bK(\Speck)$. For each integer $n \ge 1$, we will let
\be \begin{equation}
\label{IG.IT.powers}
 {\rm I}_{\rmG} ^{L\wedge ^n}=  {\overset {n} {\overbrace {{\rm I}_{\rmG} {\underset {\bK(\Speck, \rmG)} {\overset {\rm L} {\wedge}}} {\rm I}_{\rmG}  \cdots {\underset {\bK(\Speck, \rmG)} {\overset {\rm L} {\wedge}}} \ {\rm I}_{\rmG} }}} \mbox{ and }  {\rm I}_{\rmT} ^{ {\rm L}\wedge ^n}=  {\overset {n} {\overbrace { {\rm I}_{\rmT} {\underset {\bK(\Speck, \rmT)} {\overset {\rm L} {\wedge}}} {\rm I}_{\rmT}  \cdots {\underset {\bK(\Speck, \rmT)} {\overset {\rm L} {\wedge}}} \ {\rm I}_{\rmT} }}} 
    \end{equation}\ee
Now \cite[Lemma 3.5 and (3.0.2)]{CJ23} provide the following inverse system of commutative squares:
\be \begin{equation}
\label{K.diagm.0}
\xymatrix{ {\{\bKH({\rm X}, \rmT)\compl_{\rho_{\rmG}, \alpha(m)}|m\}} \ar@<1ex>[d] & {\{\bKH({\rm X}, \rmT)/{I_{\rmG}^{L\wedge^{\alpha(m)+1}}}|m\}} \ar@<-1ex>[l]^{\simeq} \ar@<1ex>[d] \\
{\{\bKH({\rm X}, \rmT)\compl_{\rho_{\rmT}, \alpha(m)}|m\}}  & {\{\bKH({\rm X}, \rmT)/{I_{\rmT}^{{\rm L}\wedge^{\alpha(m)+1}}}|m\}}  \ar@<-1ex>[l]^{\simeq}  }
\end{equation} \ee
\vskip .2cm
Making use of Proposition ~\ref{factoring.through.part.compl}(i) and making a suitable choice of the function $\alpha$ 
 following the terminology in Definition ~\ref{funct.alpha}, so that the map
$\bKH({\rm X}, \rmT) \ra \bKH(E\rmG^{gm, m}   \times {\rm X}, \rmT)$ factors through $\bKH({\rm X}, \rmT)/{I_{\rmG}^{{\rm L}\wedge^{\alpha(m)+1}}}   $ and also
the map $\bKH({\rm X}, \rmT) \ra \bKH(E\rmT^{gm, m}   \times {\rm X}, \rmT)$ factors through $\bKH({\rm X}, \rmT)/{I_{\rmT}^{{\rm L}\wedge^{\alpha(m)+1}}}$,
the above
diagrams now provide the inverse system of diagrams:
\be \begin{equation}
\label{K.diagm.1}
\xymatrix{ {\{\bKH({\rm X}, \rmT)\compl_{\rho_{\rmG}, \alpha(m)}|m\}} \ar@<1ex>[dd] & {\{\bKH({\rm X}, \rmT)/{I_{\rmG}^{{\rm L}\wedge^{\alpha(m)+1}}}|m\}} \ar@<1ex>[dd] \ar@<-1ex>[l]^{\simeq} \ar@<1ex>[r] & {\{\bKH(\rmE\rmG^{gm, m}   \times {\rm X}, \rmT)|m\}} \ar@<1ex>[d]\\
& & {\{\bKH(\rmE\rmG^{gm, m}  \times E\rmT^{gm, m} \times {\rm X}, \rmT)|n\}}  \\
{\{\bKH({\rm X}, \rmT)\compl_{\rho_{\rmT}, \alpha(m)}|m\}}  & {\{\bKH({\rm X}, \rmT)/{I_{\rmT}^{{\rm L}\wedge^{\alpha(m)+1}}}|m\}} \ar@<1ex>[r] \ar@<-1ex>[l]^{\simeq}&{\{ \bKH(\rmE\rmT^{gm, m}\times {\rm X}, \rmT)|m\}} \ar@<-1ex>[u] }
\end{equation} \ee
\vskip .2cm \noindent
Here $\bKH({\rm X}, \rmT)/{\rm I}_{\rmG}^{{\rm L}\wedge ^{\alpha(m)+1}}$ denotes the homotopy cofiber of the map 
${\rm I}_{\rmG}^{{\rm L}\wedge ^{\alpha(m)+1}} {\overset {\rm L}{\underset {\bK(\Speck, G)} \wedge}} \bKH({\rm X}, \rmT) \ra \bK({\rm X}, \rmT)$ and
\newline \noindent 
$\bKH({\rm X}, \rmT)/{\rm I}_{\rmT}^{{\rm L}\wedge ^{\alpha(m)+1}}$ denotes the homotopy cofiber of the map ${\rm I}_{\rmT}^{{\rm L}\wedge ^{\alpha(m)+1}} {\overset {\rm L}{\underset {\bK(\Speck, T)} \wedge}} \bKH({\rm X}, \rmT) \ra \bKH({\rm X}, \rmT)$.
The map  
\xymatrix{{\{\bKH({\rm X}, \rmT)/{{\rm I}_{\rmG}^{L\wedge^{\alpha(m)+1}}}|m\}}  \ar@<1ex>[r] & {\{\bKH(\rmE\rmG^{gm, n}   \times {\rm X}, \rmT)|m\}}}
exists because of the identifications: 
\vskip .1cm
$\bKH({\rm X}, \rmT) = \bKH(\rmG{\underset {\rmT} \times}{\rm X}, \rmG) \mbox{ and }\bKH(\rmE\rmG^{gm, n}   \times {\rm X}, \rmT) = \bKH(\rmE\rmG^{gm,n}{\underset {} \times}({\rmG}{\underset {\rmT} \times}X), \rmG).$
\vskip .2cm
Recall from Remark ~\ref{two.class.spaces} that there
are two distinct models of the direct system of geometric classifying spaces for split maximal tori $\rmT$ in $\GL_n$.
The product of the universal $\rmT$-bundles of these two models is yet another model for the universal bundle
over a direct system of geometric classifying space for $\rmT$, which
 shows up  in ${\{\bKH(\rmE\rmG^{gm, m}  \times \rmE\rmT^{gm, m} \times {\rm X}, \rmT)|m\}}$. 
 The projections $\rmE\rmG^{gm, m}   \times \rmX \leftarrow  \rmE\rmG^{gm, m}  \times \rmE\rmT^{gm, m} \times \rmX \ra  \rmE\rmT^{gm,m} \times \rmX$
are flat which provide the right-most two vertical maps.
\vskip .2cm
The left most square is simply the commutative square in ~\eqref{K.diagm.0}, which therefore commutes. To see the commutativity of
the right square, one may simply observe that the two maps are two different factorizations of the map
corresponding to the functor that pulls-back a $\rmT$-equivariant perfect complex on $\rmX \times \Delta[n]$ to 
a $\rmT$-equivariant perfect complex on $\rmE\rmG^{gm, n}  \times \rmE\rmT^{gm, n} \times \rmX \times \Delta[n]$. 
The left-vertical map clearly 
is a weak-equivalence on taking the homotopy inverse limit $n \ra \infty$ by Theorem ~\ref{T.comp.vs.G.comp} and by Proposition ~\ref{KH.vanishing}. The top and bottom horizontal maps in the right square 
 are provided by 
Propositions ~\ref{coveringlemma} and  ~\ref{factoring.through.part.compl}. The bottom horizontal map there
induces a weak-equivalence on taking the homotopy inverse limit as $n \ra \infty$  as shown in Theorem ~\ref{key.thm.1}.
The following Proposition now proves that the homotopy inverse limit of the maps forming the top row in the right square of diagram
~\ref{K.diagm.1} also induces a weak-equivalence. Therefore, modulo Theorem ~\ref{key.thm.1} and the following Proposition, 
this completes the proof of Theorem ~\ref{main.thm.3}. \qed
\end{proof}
\vskip .2cm
 
\begin{proposition}
 \label{derived.compl.pro.1} Assume the above situation. Then, assuming Theorem  ~\ref{key.thm.1},
the map in the top row of the right square in the diagram ~\ref{K.diagm.1} induces a weak-equivalence on taking the homotopy 
inverse limit as $n \ra \infty$. It follows that the middle row in the diagram ~\eqref{splittings} induces a weak-equivalence
on taking the homotopy inverse limit as $n \ra \infty$. 
\end{proposition}
\begin{proof}  Proposition ~\ref{indep.class.sp} with $\rmH$ there replaced by $\rmT$ shows that the vertical maps on the right in the diagram ~\ref{K.diagm.1} induce a weak-equivalence on taking the
homotopy inverse limit as $n \ra \infty$. In view of what we already observed, we now obtain the homotopy commutative diagram:
\be \begin{equation}
\xymatrix{{\bKH({\rm X}, \rmT)\compl_{\rho_{\rmG}}  } \ar@<1ex>[dd]^{\simeq} \ar@<1ex>[r]& {\holimn\{\bKH(\rmE\rmG^{gm, n}  \times {\rm X}, \rmT)|n\}} \ar@<1ex>[d]^{\simeq}\\
& {\holimn\{\bKH(\rmE\rmG^{gm, n}  \times \rmE\rmT^{gm,n}\times {\rm X}, \rmT)|n\}} \\
{\bKH({\rm X}, \rmT)\compl_{\rho_{\rmT}}} \ar@<1ex>[r]^{\simeq} & {\holimn \{ \bKH(\rmE\rmT^{gm, n} \times {\rm X}, \rmT)|n\}} \ar@<-1ex>[u]_{\simeq}}
\end{equation} \ee
\vskip .2cm
Observe that the left vertical map is a weak-equivalence by Theorem ~\ref{T.comp.vs.G.comp} and Proposition ~\ref{KH.vanishing}.
 Since all other maps in the above diagram are
now weak-equivalences except possibly for the map in the top row, it follows that is also a
weak-equivalence. This proves the first statement. The second statement is clear.
\end{proof}

\section{Proof in the case of a split torus}
In view of the above reductions, it suffices to assume that the group scheme $\rmH=\rmT= {\mathbb G}_m^n$= a split torus. In case the group scheme
is a smooth diagonalizable group-scheme, we may imbed that as a closed subgroup scheme of a split torus. Then, making use of 
the remarks following \cite[Theorem 1.6]{CJ23}, one may extend the results in this section to
actions of smooth diagonalizable group schemes also. However, we will not discuss this case explicitly.
Thus assuming $\rmH=\rmT= {\mathbb G}_m^n$= a split torus, a key step is provided
by Theorem ~\ref{key.part.thm} below.
\vskip .2cm
The following observations will play a major role in the proof of Theorem ~\ref{key.part.thm}.
\begin{enumerate}[\rm(i)]
\item First observe that if $\rmT={\mathbb G}_m^n$ is an $n$-dimensional split torus, then
$\rmE\rmT^{gm, m} = ({\mathbb A}^{m+1}-0)^n$ and $\rmB\rmT^{gm,m} = ({\mathbb P}^m)^n$. We now obtain a 
compatible collection of maps 
\[\{\bKH(\rmX, \rmT) \ra \bKH(\rmX \times_{\rmT}\rmE\rmT^{gm,m}, \rmT) | m \ge 0\},\]
where $\rmX$ is a scheme of finite type over $\rmS$ provided with an action by the split torus $\rmT$. 
\item
For $i=1, \cdots, n$,  let $\rho_{\rmT_i}: \bK(\rmS, {\mathbb G}_m) \ra \bK(\rmS)$ denote the map induced by the restriction map where ${\mathbb G}_m$ is the 
$i$-th factor in the split torus $\rmT ={\mathbb G}_m^{\times n}$. Let $\rho_{\rmT}= \rho_{\rmT_1} \wedge \cdots \wedge \rho_{\rmT_n}$. Then making use of 
\cite[Proposition 5.2, (5.0.4) and (5.0.5)]{CJ23}, we proceed to establish the following weak-equivalences
 (in Theorem ~\ref{key.part.thm}), where $\rmX$ is as in (i):
\be \begin{align}
     \label{der.compl.0}
\bKH(\rmX,  \rmT)\compl_{\rho_T, m_1, \cdots, m_n} &\simeq \bKH(\rmX \times _{\rmT}(\rmE\rmT^{gm,m_1, \cdots, m_n}))  = \bKH(\rmX \times _{\rmT}({\mathbb A}^{m_1}-\{0\} \times \cdots {\mathbb A}^{m_n}-\{0\})), \mbox{ and }\\
\bKH(\rmX, \rmT)\compl_{\rho_T}  &\simeq \bKH(\rmX\times_{\rmT}\rmE\rmT^{gm}) = \holimm \bKH(\rmX \times _{\rmT}({\mathbb A}^{m}-\{0\} \times \cdots \times {\mathbb A}^m-\{0\})). \notag
\end{align} \ee
\vskip .2cm \noindent
Corresponding statements hold for the completion with respect to the map $\rho_{\ell} \circ \rho_{\rmT}:\bK(\rmS, T) \ra \bK(\rmS ) \ra  \bK(\rmS ){\underset {\mbS} \wedge} \H(\Z/\ell)$ for a fixed prime $\ell$ different
from $char(k)=p$.
 \end{enumerate}

\begin{theorem}
 \label{key.part.thm}
Assume the base scheme $\rmS$ is a field $k$, that $\rmT= {\mathbb G}_m^n$ is a split torus over $k$ and $\rmX$ is a scheme of finite type over $\rmS$ provided with an action by $\rmT$.
For each $i=1, \cdots, n$, let $\rmT_i$ denote
the ${\mathbb G}_m$ forming the $i$-th factor of $\rmT$. Let $\rho_{\rmT}: \bK(\rmS, \rmT) \ra 
\bK(\rmS )$  denote the restriction map. 
Then the map $\bKH(\rmX, \rmT) \ra \bK(\rmE\rmT^{gm,m} \times \rmX, \rmT) = \bKH(\rmE\rmT^{gm,m} \times_{\rmT} \rmX)$ factors through the multiple partial derived completion
\[\bKH(\rmX, \rmT) \compl_{\rho_T, m, \cdots, m} = \holimDm \sigma _{\le m} \cdots \holimDm \sigma_{\le m}{\mathcal T}^{\bullet, \cdots, \bullet}_{\bK(\rmS, \rmT_1), \cdots, \bK(\rmS, \rmT_n)}(\bKH(\rmX, \rmT), \bK(\rmS))\]
 and induces a weak-equivalence:
\[\bKH(\rmX, \rmT)\compl_{\rho_{\rmT, m, \cdots, m}} \simeq \bKH(\rmE \rmT^{gm, m} \times_{\rmT}\rmX).\]
 Corresponding results also hold for the homotopy K-theory spectra replaced by the mod-$\ell$ homotopy $K$-theory
 spectra defined as in ~\eqref{KGl}.
\end{theorem}
\begin{proof} Since all the arguments in the proof carry over when the $K$-theory spectrum $\bK$ is replaced by the
mod-$\ell$ K-theory spectra, we will only consider the K-theory spectrum. First observe that 
\be \begin{align}
     \label{KST}
\pi_*(\bK(\rmS, \rmT)) &= \rmR(\rmT) {\underset {\Z} \otimes} \pi_*(\bK(\rmS)) \cong {\mathbb Z}[t_1, \cdots, t_n, t_1^{-1}, \cdots, t_n^{-1}] \otimes \pi_*(\bK(\rmS))
\end{align} \ee
\vskip .2cm \noindent
{\it Step 1}. Let $\rmT_i$ denote a fixed $1$-dimensional torus (that is, ${\mathbb G}_m$) forming a factor of $\rmT$: it corresponds to the fixed generator $t_i$ of $\rmR(\rmT)$. 
Let $\lambda_i$ denote the $1$-dimensional representation of $\rmT$ corresponding to $t_i$. For each $m \ge 0$,
$(\lambda_i -1)^{m_i+1}$ defines a class in $\pi_0(\bK(\rmS, \rmT)) = \rmR(\rmT) \otimes \pi_0(\bK(\rmS))$. 
Viewing this class as a map $\mbS \ra \bK(\rmS, \rmT)$ (recall $\mbS$ denotes 
the $\rmS^1$-sphere spectrum), the composition $\mbS \wedge \bK(\rmS, \rmT) \ra \bK(\rmS, \rmT) \wedge \bK(\rmS, \rmT) {\overset \mu \ra} \bK(\rmS, \rmT)$
defines a map of spectra 
\[\bK(\rmS, \rmT) \ra \bK(\rmS, \rmT),\]
which will be still denoted $(\lambda_i -1)^{m+1}$. (Here $\mu$ is the ring structure on the spectrum $\bK(\rmS, \rmT)$.) {\it The heart of the proof of this theorem involves showing that the homotopy cofiber of
the map $(\lambda_i -1)^{m_i+1}$ identifies with the partial derived completion to the order $m_i$, along 
the restriction map $\rho_{\rmT_i}: \bK(\rmS, \rmT_i) \ra \bK(\rmS)$.}
\vskip .2cm \noindent
{\it Step 2}. 
Let $\rmX \times {\mathbb A}^{m_i+1}$ denote the vector bundle of rank $m_i+1$ over $\rmX$ on which the torus $\rmT$ acts as follows:
it acts diagonally on the two factors, and where the action on ${\mathbb A}^{m_i+1}$ is through the character corresponding to $\lambda_i$ on each fiber  ${\mathbb A}^{m_i+1}$. 
The {\it Koszul-Thom class} of this bundle is $\pi^*(\lambda_i -1)^{m_i+1}$, where $\pi: \rmX \times {\mathbb A}^{m_i+1} \ra \rmX$ is the
projection. This defines a class in 
$\pi_0({\bKH(\rmX \times {\mathbb A}^{m_i+1}, \rmX \times ({\mathbb A}^{m_i+1}-0), T)})$. 
 Next consider the  diagram:
\be \begin{equation}
     \label{key.diagm}
\xymatrix{{\bKH(\rmX, \rmT)} \ar@<1ex>[r]^{(\lambda_i-1)^{m+1}} \ar@<-1ex>[d]^{\simeq} & {\bKH(\rmX, \rmT)} \ar@<1ex>[r] \ar@<-1ex>[d] ^{\simeq} & {cofiber ((\lambda_i-1)^{m_i+1})} \ar@<1ex>[d]\\
{\bKH(\rmX \times {\mathbb A}^{m_i+1}, \rmX \times ({\mathbb A}^{m_i+1}-0), T)} \ar@<1ex>[r] & {\bKH(\rmX \times {\mathbb A}^{m_i+1}, \rmT)} \ar@<1ex>[r] & {\bKH(\rmX \times ({\mathbb A}^{m_i+1}-0), \rmT)}}
  \end{equation} \ee
\vskip .2cm \noindent
where the bottom row is the stable homotopy fiber sequence associated to the  homotopy cofiber sequence: $\rmX\times( {\mathbb A}^{m_i+1}-0) \ra \rmX \times {\mathbb A}^{m_i+1} \ra 
(\rmX \times {\mathbb A}^{m_i+1})/(\rmX \times ({\mathbb A}^{m_i+1}-0))$ in the ${\mathbb A}^1$-stable homotopy category. 
The top row is the homotopy cofiber sequence associated to the map denoted $(\lambda_i -1)^{m_i +1}$.
\vskip .2cm
The left-most vertical map
is provided by {\it Thom-isomorphism}, that is,  by cup-product with the Koszul-Thom-class $\pi^*(\lambda_i -1)^{m_i+1}$: see Proposition ~\ref{Thm.isom}. Therefore, this map is
a weak-equivalence. The second vertical map is a weak-equivalence by the homotopy property. It is also clear that the left-most
square homotopy commutes. Therefore, the right square also homotopy commutes and the right most vertical map is also a weak-equivalence.
If $\rmT = \rmT'_i \times \rmT_i$, then one also obtains the identification:
\[\bKH(\rmX \times ({\mathbb A}^{m_i+1}-\{0\}), \rmT) \simeq \bKH(\rmX \times_{\rmT_i} ({\mathbb A}^{m_i+1}-\{0\}), \rmT_i').\]
\vskip .2cm \noindent
{\it Step 3}. Next we take $m_i=0$. Then it follows from what we just showed that the map $cofiber(\lambda_i -1) \ra \bKH(\rmX \times ({\mathbb A}^1-0), \rmT) 
\simeq \bKH(\rmX \times_{\rmT_i} ({\mathbb A}^1-\{0\}), \rmT_i')= \bKH(\rmX , \rmT_i')$
(forming the last vertical map in ~\eqref{key.diagm}) is a weak-equivalence. Now the map 
$\bKH(\rmX \times {\mathbb A}^{1}, \rmT) \ra \bKH(\rmX, \rmT_i)$
forming the bottom row in the right-most  square of ~\eqref{key.diagm} identifies with restriction
map $\bKH(\rmX, \rmT) \ra \bKH(\rmX, \rmT_i')$ induced by the inclusion $\rmT_i' \subseteq \rmT$. Therefore, by the homotopy commutativity of the right-most square of 
~\eqref{key.diagm}, it follows that the homotopy fiber of the restriction map $\rho_{\rmT_i}:\bKH(\rmX, \rmT) \ra \bKH(\rmX, \rmT_i')$ identities 
up to weak-equivalence with the homotopy fiber of the map $\bKH(\rmX, \rmT) \ra cofiber(\lambda_i -1)$, that is,  
with the map $(\lambda_i -1):\bKH(\rmX, \rmT) \ra \bKH(\rmX, \rmT)$. 
\vskip .2cm
{\it Step 4}. Next recall (using \cite[Corollary 6.7]{C08}) that the partial derived completion $\bKH(\rmX, \rmT) \compl_{\rho_{\rmT_i}, m_i}$  
may be identified as follows.
Let $\tilde I_{\rmT_i} \ra I_{\rmT_i}$ denote a cofibrant replacement in the 
category of module spectra over $\bK(\rmS, \rmT)$, with $I_{\rmT_i}$ denoting the homotopy fiber of the restriction $\bK(\rmS, \rmT) \ra \bK(\rmS, \rmT_i')$. Then
\[\bKH(\rmX, \rmT) \compl_{\rho_{\rmT_i}, m_i}=  Cofib( {\overset {m+1} {\overbrace {{\tilde I_{\rmT_i}} {\underset {\bK(\rmS, \rmT)} \wedge} \tilde I_{\rmT_i} {\underset {\bK(\rmS, \rmT)} \wedge} \cdots {\underset {\bK(\rmS, \rmT)} \wedge} \tilde I_{\rmT_i}}}{\underset {\bK(\rmS, \rmT)} \wedge} \bKH(\rmX, \rmT)} \ra \bKH(\rmX, \rmT)).\]
What we have just shown in Step 3 is  that the homotopy fiber $I_{\rmT_i}$ of the restriction  map $ \bK(\rmS, \rmT) \ra \bK(\rmS, \rmT_i')$ identifies with the map $\bK(\rmS, \rmT) {\overset {(\lambda_i -1)} \ra} \bK(\rmS, \rmT)$ and hence is clearly cofibrant over $\bK(\rmS, \rmT)$:
that is,  $\tilde I_{\rmT_i} = I_{\rmT_i}$ is simply $\bK(\rmS, \rmT)$ mapping into
$\bK(\rmS, \rmT)$ by the map $(\lambda_i -1)$. Therefore the homotopy cofiber
\[Cofib( {\overset {m_i+1} {\overbrace {{\tilde I_{\rmT_i}} {\underset {\bK(\rmS, \rmT)} \wedge} \tilde I_{\rmT_i} {\underset {\bK(\rmS, \rmT)} \wedge} \cdots {\underset {\bK(\rmS, \rmT)} \wedge} \tilde I_{\rmT}}}{\underset {\bK(\rmS, \rmT)} \wedge} \bKH(\rmX, \rmT)   } \ra \bKH(\rmX, \rmT_i')) = Cofiber (\bKH(\rmX, \rmT) {\overset {(\lambda_i -1)^{m_i+1}} \ra} \bKH(\rmX, \rmT)).\]
 Making use of the  the diagram ~\eqref{key.diagm}, these observations prove  that
\[\bKH(\rmX, \rmT) \compl_{\rho_{\rmT_i}, m_i} \simeq cofiber (\lambda_i -1)^{m_i+1} \simeq \bKH(\rmX \times_{\rmT_i} ({\mathbb A}^{m_i+1}-0), \rmT_i') =\bKH(\rmX \times_{\rmT_i} \rmE\rmT_i^{gm,m_i}, \rmT_i').\]
\vskip .2cm \noindent
 One may observe that the above argument in fact completes the proof for the case $\rmT$ is a  $1$-dimensional 
torus.  (It may be worthwhile pointing out that classical argument due to Atiyah and Segal for the
usual completion for torus actions in equivariant topological K-theory is similar: see \cite[section 3, Step 1]{AS69}.)
\vskip .2cm
{\it Step 5}. In case the split torus $\rmT = \Pi_{i=1}^n {\mathbb G}_m$, one may now repeat the argument with $\rmX$ replaced by
$\rmX \times_{\rmT_i} \rmE\rmT_i^{gm,m}$ and $\rmT$ replaced by $\rmT_i'$. An ascending induction on $n$, then completes the proof in view
of \cite[Proposition 5.3 (i) and (ii)]{CJ23}.
 \end{proof}
We obtain the following corollary to Theorem ~\ref{key.part.thm}.
 \begin{corollary}
  \label{der.comp.negK}
Assume as in Theorem ~\ref{key.part.thm} that the base scheme $\rmS$ is the spectrum of a field $k$, 
$\rmT= {\mathbb G}_m^n$ is a split torus over $\rmS$ and $\rmX$ is a scheme of finite type over $\rmS$ provided with an action by $\rmT$.
Then we obtain the weak-equivalences for any $n$ invertible in $\k$, where $\mbS/n$ denotes the mod-$n$ Moore spectrum, and where ${\mathbb K}(\quad, \rmT)$ denotes the $\rmK$-theory spectrum defined in ~\eqref{str.Kth.sp}:
\begin{enumerate}[\rm(i)]
\item $(\mbS/n {\underset  {\mbS}\wedge } {\mathbb K}(\rmX, \rmT))_{\rho_{\rmT, m}} \simeq\mbS/n {\underset  {\mbS}\wedge } ({\mathbb K}(\rmX, \rmT)_{\rho_{\rmT, m}}) \simeq \mbS/n {\underset  {\mbS}\wedge } {\mathbb K}(\rmE\rmT^{gm,m} \times_{\rmT}\rmX, \rmT), m \ge 1$, \mbox{ and}
\item $(\mbS/n {\underset  {\mbS}\wedge } {\mathbb K}(\rmX, \rmT))_{\rho_{\rmT}} \simeq \mbS/n {\underset  {\mbS}\wedge } {\mathbb K}(\rmE\rmT^{gm}\times_{\rmT} \rmX, \rmT)$.
\end{enumerate}
\end{corollary}
\begin{proof} We start with the sequence of maps 
 \[\{\bKH(\rmX, \rmT)_{\rho_{\rmT}, m} {\overset {\simeq}\longrightarrow} \bKH(\rmE\rmT^{gm,m} \times_{\rmT}\rmX, \rmT)|m\} \]
provided by Theorem ~\ref{key.part.thm}. Now we smash all the terms with $\mbS/n$ over $\mbS$. In view of \cite[Theorems 9.5, 9.6]{ThTr} and
\cite[p. 391]{Wei83}, one obtains the compatible identifications:
\[\{\mbS/n {\underset  {\mbS}\wedge } \bKH(\rmE\rmT^{gm,m} \times_{\rmT}\rmX, \rmT) \simeq \mbS/n {\underset  {\mbS}\wedge } {\mathbb K}(\rmE\rmT^{gm,m} \times_{\rmT}\rmX, \rmT)|m\}.\]
One also obtains the compatible identifications:
\[\{\mbS/n {\underset  {\mbS}\wedge } (\bKH(\rmX, \rmT)_{\rho_{\rmT}, m}) \simeq (\mbS/n {\underset  {\mbS}\wedge } (\bKH(\rmX, \rmT)))_{\rho_{\rmT}, m} \simeq (\mbS/n {\underset  {\mbS}\wedge } ({\mathbb K}(\rmX, \rmT)))_{\rho_{\rmT}, m} \simeq \mbS/n {\underset  {\mbS}\wedge } ({\mathbb K}(\rmX, \rmT)_{\rho_{\rmT}, m})|m\},\]
in view of \cite[Corollary 5.5(1)]{KR}. (To see that the finite homotopy inverse limits involved in the partial derived completions along $\rho_{\rmT}$ commute
with the smash product with $\mbS/n$, one needs to consider instead the homotopy fiber of the map 
induced by the map $\{\mbS\wedge A_m {\overset {n\wedge id} \ra} \mbS\wedge A_m|m\}$ for any truncated 
cosimplicial object of spectra $\{A_m|m\}$. This clearly commutes with
 homotopy inverse limit of the truncated cosimplicial object.) Therefore, it remains to take the homotopy inverse limit as $m \ra \infty$. While such homotopy inverse limits again do
not necessarily commute with smash products, in this case it does because of the following observation. 
Since we are working stably, it suffices 
to observe that the homotopy inverse limit over $m \ra \infty$ commutes with taking the homotopy fiber of the map induced by the map $\{\mbS\wedge A_m {\overset {n\wedge id} \ra} \mbS\wedge A_m|m\}$ for any
cosimplicial object of spectra $\{A_m|m\}$.
\end{proof}

\vskip .2cm
\begin{theorem}
\label{key.thm.1} Assume the base scheme $\rmS$ is a field $k$ and that $\rmX$ is a scheme of finite type over $k$. Let $\rmH=\rmT={\mathbb G}_m^n$ denote a split torus acting on $\rmX$. Then, for each fixed positive integer $m$,
 the map 
\[\bKH({\rm X}, \rmH)  \ra \bKH({\rm E}\rmH^{gm,m}{ \times}\rmX, \rmH) \simeq\bKH({\rm E}\rmH^{gm,m}{\underset {\rmH}  \times}\rmX)\]
 factors through the partial derived completion $\bKH({\rm X}, \rmH)\compl_{\rho_{\rmH},m}$ and the induced map 
\[\bKH({\rm X}, \rmH)\compl_{\rho_{H}} \ra \holimm \bKH({\rm E}\rmH^{gm,m}{\underset {\rmH}  \times}\rmX)\]
\newline \noindent
is a weak-equivalence. Moreover, this weak-equivalence is natural in 
both $\rmX$ and $\rmH$.
 The corresponding assertions hold
with the ${\rm KH}$-theory spectrum replaced by the mod-$\ell$ ${\rm KH}$-theory spectrum (defined as in ~\eqref{KGl}) in general, and by the $K$-theory and mod-$\ell$ K-theory spectrum
 when $\rmX$ is  regular. It also holds with the spectrum $\bKH$ replaced by the spectrum $\mbS/n {\underset  {\mbS}\wedge } {\mathbb K}$ when $n$ is invertible in 
 the base field $\k$.
\end{theorem}
\begin{proof} The proof follows readily by taking the homotopy inverse limit as $m \ra \infty$ of the weak-equivalence
\[\bKH(\rmX, \rmT)\compl_{\rho_{\rmT, m, \cdots, m}} \simeq \bKH((\rmE \rmT^{gm, m} \times \cdots \times \rmE\rmT^{gm,m}) \times_{\rmT}\rmX)\]
established in Theorem ~\ref{key.part.thm} above. The last statement follows from Corollary ~\ref{der.comp.negK}.
\end{proof}

We conclude this section with the proof of Theorem ~\ref{main.thm.1}.
\vskip .2cm \noindent
{\bf Proof of Theorem ~\ref{main.thm.1}}. Observe that Theorem ~\ref{key.thm.1} proves Theorem ~\ref{main.thm.1} when the group-scheme is 
a split torus and Theorem ~\ref{main.thm.3} then proves Theorem ~\ref{main.thm.1} when the group-scheme is a finite product of $\GL_n$s.
Next let $\rmG$ denote a not necessarily connected linear algebraic group over $\k$: we imbed $\rmG$ as a closed linear algebraic subgroup
of $\tilde \rmG$, which is a connected split reductive group over $k$ satisfying the Standing Hypotheses ~\ref{stand.hyp.1}(ii). (Observe that, in particular, 
$\tilde \rmG$ is special.) Then, we may imbed
$\tilde \rmG$ as a closed sub-group-scheme in a finite product $\GL_{n_1} \times \cdots \times \GL_{n_q}$ so that the restriction map
$\rmR(\GL_{n_1} \times \cdots \times \GL_{n_q}) \ra \rmR(\tilde \rmG)$ is surjective. Therefore we obtain the weak-equivalences:
\[\bKH({\rm X}, \rmG)\compl_{\rho_{ \rmG}} \simeq \bKH({\rm X}, \rmG)\compl_{\rho_{\tilde \rmG}} \simeq \bKH(\tilde \rmG \times _{\rmG}{\rm X}, \tilde \rmG)\compl_{\rho_{\tilde \rmG}} \simeq \bKH(\tilde \rmG\times_{\rmG}{\rm X}, \tilde \rmG) \compl_{\rho_{\GL_{n_1} \times \cdots \times \GL_{n_q}}} \]
\[\simeq \bKH((\GL_{n_1} \times \cdots \times \GL_{n_q})\times_{\rmG} \rmX, \GL_{n_1} \times \cdots \times \GL_{n_q})\compl_{\rho_{\GL_{n_1} \times \cdots \times \GL_{n_q}}}\]
where the first and third weak-equivalences are by Proposition ~\ref{key.obs.2} (see also \cite[Theorem 1.6]{CJ23}), and the second and fourth are clear. 
Now we may assume $\tilde \rmG$ is in fact $\GL_{n_1} \times \cdots \times \GL_{n_q}$.  By Theorems ~\ref{key.thm.1} and ~\ref{main.thm.3},
then $\bKH(\tilde \rmG \times _{\rmG}{\rm X}, \tilde \rmG)\compl_{\rho_{\tilde \rmG}}$ identifies with 
\[\bKH(\rmE\tilde \rmG^{gm} \times _{\tilde \rmG}(\tilde \rmG\times_{\rmG}X)) \simeq \bKH(\rmE\tilde \rmG^{gm} \times _{\rmG}X).\]
This completes the proof of Theorem ~\ref{main.thm.1}(i). The proof of Theorem ~\ref{main.thm.1}(ii) is similar and is therefore skipped.
Since $ \rmG$ is assumed to be special, Proposition ~\ref{indep.class.sp} provides the weak-equivalence:
\[\bKH(\rmE\tilde \rmG^{gm} \times_{\rmG}X)) \simeq \bKH(\rmE \rmG^{gm} \times _{\rmG}X)\]
which proves Theorem ~\ref{main.thm.1}(iii). This completes the proof of Theorem ~\ref{main.thm.1}. \qed
\vskip .2cm
\section{Appendix: diagrams of towers of fibrations of $\rmS^1$-spectra}
We begin with following Lemma.
\begin{lemma}
 \label{tower.fibs} 
  Let $\rmE$ denote an $\rmS^1$-ring spectrum.
 Let $\{p_m:\cX_{m+1} \ra \cX_m|m \ge 0\}$ denote a tower of maps of $\rmE$-module spectra and let $\{q_m:\cY_{m+1} \ra \cY_m|m \ge 0\}$ denote a tower of
 {\it fibrations} of $\rmE$-module spectra.
 
 Assume that one is given, for each $m$, a map $f_m: \cX_n \ra \cY_m$, so that each of the squares
 \[\xymatrix{{\cX_{m+1}} \ar@<1ex>[r]^{p_m} \ar@<1ex>[d]^{f_{m+1}} &{\cX_m} \ar@<1ex>[d]^{f_m}\\
             {\cY_{m+1}} \ar@<1ex>[r]^{q_m} & {\cY_m}}
 \]
 homotopy commutes in the category of $\rmE$-module spectra. Then one may inductively replace each of the maps $f_m$ up to homotopy by a map $f_m'$, $m\ge 0$, so that
 the maps $\{f_m'|m \ge 0\}$ define a map of inverse systems of maps  $\{\cX_m|m\}$ and $\{\cY_m|m\}$ of $\rmE$-module spectra.
\end{lemma}
\begin{proof}
 We will construct the replacement $f_m'$ by ascending induction on $m$, starting with $f_0'=f_0$. Assume we have 
 found replacements of the maps $f_i$, by $f_i'$, for all $0 \le i \le m$, so that (i) each of the maps $f_i'$ is homotopic to 
 the given map $f_i$ for all $0 \le i \le m$ and that the diagram
 \be \begin{equation}
 \label{repl.diagrm}
 \xymatrix{{\cX_m} \ar@<1ex>[d]^{f_m'} \ar@<1ex>[r]^{p_m} &{ \cdots } \ar@<1ex>[r]^{p_1} & {\cX_0} \ar@<1ex>[d]^{f_0'}\\
             {\cY_m} \ar@<1ex>[r]^{q_m} & {\cdots} \ar@<1ex>[r]^{q_1} &{\cY_0}}
 \end{equation} \ee
strictly commutes. Let $\rmH: \cX_{{\it m}+1} \times \Delta [1] \ra \cY_{\it m}$ denote a homotopy so that,
$\rmH(x, 1) ={\it q}_{{\it m}+ 1} \circ {\it f}_{{\it m}+1}$ and $\rmH({\it x_m}, 0) = {\it f_{m}} \circ {\it p}_{{\it m}+1}$, for all $x \eps \cX_{{\it m}+1}$. 
Observe that all we need is that $\cX_{{\it m}+1} \times \Delta [1]$ be a cylinder object for $\cX_{{\it m}+1}$ in the sense of \cite[Definition 1.2.4]{Hov}.
Now we consider the diagram:
\[\xymatrix{{\cX_{m+1}} \ar@<1ex>[r]^{f_{m+1}} \ar@<-1ex>[d]^{i_1} & {\cY_{m+1}} \ar@<-1ex>[d]^{q_{m+1}}\\
          {\cX_{m+1} \times \Delta [1]} \ar@<1ex>[r]^{\rmH} \ar@<1ex>[ur]^{\rmH'} & {\cY_{m}}\\
          {\cX_{m+1}} \ar@<1ex>[u]_{i_0} \ar@<1ex>[r]^{p_{m+1}} &{\cX_m} \ar@<1ex>[u]_{f_m}}\]
Here the homotopy $\rmH'$ is a lifting of the given homotopy $\rmH$: observe that such a lifting $\rmH'$ exists because 
the map $i_1$ is a trivial cofibration and $q_{m+1}$ is a fibration. 
\vskip .1cm
Now we will replace the map $f_{m+1}$ by the map $f_{m+1}'=\rmH' \circ i_0$. Then one may observe that $\rmH'$
provides a homotopy between the two maps $f_{m+1}$ and $f_{m+1}'$. Moreover, the diagram 
\[\xymatrix{{\cX_{m+1}} \ar@<1ex>[d]^{f_{m+1}'} \ar@<1ex>[r]^{p_{m+1}} &{\cX_m} \ar@<1ex>[d]^{f_m'} \ar@<1ex>[r]^{p_m} &{ \cdots } \ar@<1ex>[r]^{p_1} & {\cX_0} \ar@<1ex>[d]^{f_0'}\\
             {\cY_{m+1}} \ar@<1ex>[r]^{q_{m+1}} &{\cY_m} \ar@<1ex>[r]^{q_m} & {\cdots} \ar@<1ex>[r]^{q_1} &{\cY_0}}\]
strictly commutes.
Clearly, one may repeat the above arguments to complete the proof.
\end{proof}
\begin{lemma}
 \label{tower.1}
 Let $\rmE$ denote an $\rmS^1$-ring spectrum. Let 
 \be \begin{equation}
   \label{cone.diagm.1}
 \xymatrix{{\rmA_1} \ar@<1ex>[r]^{i_1} \ar@<1ex>[d]^{p_1^{A}} & {\rmX_1} \ar@<1ex>[d]^{p_1} \ar@<1ex>[r]^{f_1} &{\rmY_1} \ar@<1ex>[d]^{q_1}\\
           {\rmA_0} \ar@<1ex>[r]^{i_0} &{\rmX_0}  \ar@<1ex>[r]^{f_0} & {\rmY_0}}
\end{equation} \ee
denote a strictly commutative diagram of $\rmE$-module spectra, so that the composite maps $f_1 \circ i_1$ and $f_0 \circ i_0$ are null-homotopic. Then the
 resulting diagram
 \be \begin{equation}
 \label{cone.diagm.2}
 \xymatrix{{\rmX_1 \sqcup_{\rmA_1} {\it c}\rmA_1} \ar@<1ex>[r] \ar@<1ex>[d] & {\rmY_1} \ar@<1ex>[d]\\
             {\rmX_0 \sqcup_{\rmA_0} {\it c}\rmA_0} \ar@<1ex>[r] & {\rmY_0}} 
 \end{equation} \ee
homotopy commutes in the category of $\rmE$-module spectra, where $c\rmA_i$ denotes the cone on $\rmA_i$, $i=0, 1$.
\end{lemma}
\begin{proof} Recall that the diagram ~\eqref{cone.diagm.1} commutes strictly. Therefore, it suffices to show
 that the square
 \be \begin{equation}
  \label{cone.diagm.3}
  \xymatrix{{c\rmA_1} \ar@<1ex>[r]^{\rmH_1} \ar@<1ex>[d]^{cp_1^{\rmA}} & {\rmY_1} \ar@<1ex>[d]^{q_1}\\
            {c\rmA_0} \ar@<1ex>[r]^{\rmH_0} & {\rmY_0}}
\end{equation} \ee
commutes up to homotopy, where $\rmH_j: c\rmA_j \ra \rmY_j$, $j=0,1$, are the maps induced by the given null-homotopies of
the composite maps $f_j\circ i_j$, $j=0, 1$. 
\vskip .1cm
Let $\tilde \rmH_j:c\rmA_j \times \Delta[1] \ra \rmY_j$, $j=0,1$, be the maps defined as $\tilde \rmH_j(a_j, t, s) = \rmH_j(a_j, ts)$,
$a_j \eps \rmA_j$, $t_j, s \eps \Delta [1]$. Here we also assume $\rmH_j(a_j, 0) =*$ for all $a_j \eps \rmA_j$, and that
$\rmH_j(\quad, 1)$ is the given composite map $\rmA_j {\overset {\it i_j} \ra} \rmX_j {\overset {\it f_j} \ra }\rmY_j$. 
\vskip .1cm
Therefore, $\tilde \rmH_j$ is a homotopy from the given $\rmH_j: c\rmA_j \ra \rmY_j$ to the trivial map sending everything to the
 base point $*$. It follows that $q_1 \circ \tilde \rmH_1$ is a homotopy from $q_1 \circ \rmH_1$ to the trivial map $c\rmA_1 \ra \rmY_0$.
 Similarly, $\tilde \rmH_0 \circ cp_1^{\rmA}$ is a homotopy from $\rmH_0 \circ p_0$ to the trivial map $c\rmA_1 \ra \rmY_0$. 
 Observe that the above notion of homotopy is actually a {\it left homotopy} in the sense of \cite[Definition 1.2.4]{Hov}.
 Making use of the injective model structures, all our presheaves of spectra are cofibrant, so that the relation of two maps being {\it left-homotopic} 
 is an equivalence relation: see \cite[Proposition 1.2.5 (iii)]{Hov}. Therefore, this proves the Lemma.
\end{proof}

\vskip .2cm 

\end{document}